\documentclass[twoside]{amsart}
\RequirePackage{graphicx}
\RequirePackage[english,hidelinks]{hyperref}

\RequirePackage{amssymb,mathrsfs,amscd,amsmath,amsthm}

\RequirePackage{cite,enumitem}

\RequirePackage{multirow}
\RequirePackage{booktabs}

\RequirePackage{tikz,tikz-cd}
\usetikzlibrary{matrix,arrows,decorations.pathmorphing}

\newtheoremstyle{myplain}
  {8.0pt plus 2.0pt minus 4.0pt}%
  {8.0pt plus 2.0pt minus 4.0pt}%
  {\itshape}%
  {}
  {\bfseries}%
  {~---}%
  { }%
  {\thmname{#1}\thmnumber{ #2}\thmnote{ (#3)}}%

\theoremstyle{myplain}
\newtheorem{thm}{Theorem}[section]
\newtheorem{definition}[thm]{Definition}
\newtheorem{theorem}[thm]{Theorem}
\newtheorem{lemma}[thm]{Lemma}
\newtheorem{proposition}[thm]{Proposition}
\newtheorem{corollary}[thm]{Corollary}

\theoremstyle{remark}
\newtheorem{remark}[thm]{Remark}

\newtheorem{example}[thm]{Example}

\newtheorem{examples}[thm]{Examples}
\newtheorem{question}[thm]{Question}

\numberwithin{equation}{section}

\def\equationautorefname~#1\null{%
(#1)\null
}
\def\itemautorefname~#1\null{%
#1\null
}
\def\thmautorefname~#1\null{%
#1\null
}
\def\subsubsectionautorefname~#1\null{%
#1\null
}
\def\subsectionautorefname~#1\null{%
#1\null
}
\def\sectionautorefname~#1\null{%
#1\null
}

\AtBeginDocument{\renewcommand{\ref}[1]{\autoref{#1}}}

\newcommand{\Mcal}{\mathcal{M}}

\newcommand{\Pcal}{\mathcal{P}}

\newcommand{\Bscr}{\mathscr{B}}

\newcommand{\Xscr}{\mathscr{X}}

\newcommand{\Zscr}{\mathscr{Z}}

\renewcommand{\AA}{\mathbb{A}}

\newcommand{\CC}{\mathbb{C}}

\newcommand{\FF}{\mathbb{F}}

\newcommand{\NN}{\mathbb{N}}

\newcommand{\PP}{\mathbb{P}}
\newcommand{\QQ}{\mathbb{Q}}
\newcommand{\RR}{\mathbb{R}}

\newcommand{\ZZ}{\mathbb{Z}}

\newcommand{\Mfrak}{\mathfrak{M}}
\newcommand{\Nfrak}{\mathfrak{N}}

\DeclareMathOperator{\Pic}{Pic}

\DeclareMathOperator{\PGL}{PGL}

\RequirePackage{calligra}
\DeclareMathOperator{\cHom}{\mathscr{H}\text{\kern -3pt\calligra{om}}\;}
\DeclareMathOperator{\cEnd}{\mathscr{E}\text{\kern -3pt\calligra{nd}}\;}
\DeclareMathOperator{\cExt}{\mathscr{E}\text{\kern -3pt\calligra{xt}}\;}
\DeclareMathOperator{\cTor}{\mathscr{T}\text{\kern -3pt\calligra{or}}\;}

\newcommand{\osh}[1]{\mathcal{O}_{#1}}

\DeclareMathOperator{\im}{im}

\DeclareMathOperator{\ann}{ann}
\DeclareMathOperator{\depth}{depth}
\DeclareMathOperator{\length}{length}

\let\chiletter\chi
\renewcommand*{\chi}{\mathop{\chiletter}}

\renewcommand*{\phi}{\varphi}

\renewcommand*{\theta}{\vartheta}

\renewcommand*{\epsilon}{\varepsilon}

\def\ceil#1{\lceil{}#1\rceil{}}

\def\floor#1{\lfloor{}#1\rfloor{}}

\usepackage{listings}
\definecolor{lightred}{rgb}{0.8,0.1,0}
\definecolor{blue}{rgb}{0.1,0.1,0.8}
\lstdefinelanguage{macaulay2}{
	morecomment=[l]{--},
	morekeywords={
	    Adjacent, AffineVariety, AfterNoPrint, AfterPrint, Algorithm, Alignment, AssociativeExpression, Authors, AuxiliaryFiles, BLAS, Bag, BaseFunction, BaseRow, BasisElementLimit, Bayer, BeforePrint, BesselJ, BesselY, BettiTally, Binary, BinaryOperation, Binomial, Boolean, Boxes, CC, CacheExampleOutput, CacheFunction, CacheTable, CallLimit, Caveat, Center, Certification, ChainComplex, ChainComplexMap, CheckDocumentation, ClosestFit, CodimensionLimit, CoefficientRing, Cofactor, CoherentSheaf, Command, CompiledFunction, CompiledFunctionBody, CompiledFunctionClosure, Complement, CompleteIntersection, ComplexField, Configuration, Consequences, Constant, DebuggingMode, Decompose, DegreeLift, DegreeLimit, DegreeMap, DegreeOrder, DegreeRank, Degrees, Dense, Density, Descending, Descent, Description, Dictionary, Dispatch, Divide, DivideConquer, DocumentTag, Down, EXAMPLE, Eliminate, Email, Encapsulate, EncapsulateDirectory, End, Engine, EngineRing, EulerConstant, ExampleFiles, Exclude, Expression, Ext, FLINT, Fano, FastNonminimal, File, FileName, FilePosition, FlatMonoid, Flexible, FollowLinks, ForestNode, Format, FractionField, Function, FunctionBody, FunctionClosure, GBDegrees, GF, GLex, GRevLex, Gamma, GeneralOrderedMonoid, GenerateAssertions, Generic, Global, GlobalAssignHook, GlobalDictionary, GlobalReleaseHook, GradedModuleMap, Grassmannian, GroebnerBasisOptions, GroupLex, GroupRevLex, HH, HardDegreeLimit, HashTable, HeaderType, Heading, Headline, Heft, Height, Hermitian, Holder, Hom, HomePage, Homogeneous, Homogeneous2, HorizontalSpace, IgnoreExampleErrors, ImmutableType, IndeterminateNumber, IndexedVariable, InexactField, InexactFieldFamily, InfiniteNumber, InfoDirSection, Inhomogeneous, Inputs, InstallPrefix, Intersection, InverseMethod, Inverses, Iterate, Join, Key, Keyword, LUdecomposition, Layout, Left, LengthLimit, Lex, Limit, Linear, LinearAlgebra, LoadDocumentation, Local, LocalDictionary, LongPolynomial, LowerBound, MPIR, Macaulay2, Macaulay2Doc, MakeDocumentation, MakeInfo, MakeLinks, MatrixExpression, MaxReductionCount, MaximalRank, MethodFunction, MethodFunctionWithOptions, MinimalGenerators, MinimalMatrix, Minimize, Minus, Module, ModuleMap, Monoid, MonoidElement, MonomialIdeal, MonomialOrder, MonomialSize, Monomials, MutableHashTable, MutableList, MutableMatrix, NCLex, Net, NetFile, NewFromMethod, NewMethod, NewOfFromMethod, NewOfMethod, NoPrint, NonAssociativeProduct, Number, OO, OneExpression, OptionTable, Order, OrderedMonoid, Outputs, Package, PackageDictionary, PackageExports, PackageImports, PackagePrefix, PairLimit, PairsRemaining, Parenthesize, Partition, Pfaffians, PolynomialRing, Position, Postfix, Power, Prefix, PrimaryTag, PrimitiveElement, Print, Product, ProductOrder, Proj, ProjectiveHilbertPolynomial, ProjectiveVariety, Pseudocode, PushforwardComputation, QRDecomposition, Quotient, QuotientRing, RR, Range, RealField, Reduce, Reload, RemakeAllDocumentation, Resolution, Result, RevLex, Reverse, Right, Ring, RingElement, RingMap, RowExpression, RunExamples, SPACE, SVD, SYNOPSIS, Schubert, ScriptedFunctor, SeeAlso, SelfInitializingType, SeparateExec, Set, SheafOfRings, SizeLimit, SkewCommutative, Sort, SortStrategy, SourceCode, SourceRing, SparseMonomialVectorExpression, SparseVectorExpression, Spec, Standard, StopBeforeComputation, StopWithMinimalGenerators, Strategy, String, SubringLimit, Subscript, Sugarless, Sum, SumOfTwists, Superscript, Symbol, Syzygies, SyzygyLimit, SyzygyMatrix, SyzygyRows, TEST, Tally, Task, TeXmacs, Thing, Time, Tor, Toric, TotalPairs, TreeNode, TypicalValue, URL, Undo, Unique, Unmixed, Up, UpdateOnly, UpperTriangular, Usage, UseHilbertFunction, UseSyzygies, UserMode, Variable, VariableBaseName, Variables, Variety, Vector, Verbose, Verify, Version, VerticalList, VerticalSpace, Weights, WeylAlgebra, Wikipedia, ZeroExpression, abs, accumulate, acos, acosh, acot, addCancelTask, addDependencyTask, addEndFunction, addHook, addStartFunction, addStartTask, adjoint, adjoint', agm, alarm, all, allowableThreads, ambient, ancestor, ancestors, and, annihilator, antipode, any, append, applicationDirectory, applicationDirectorySuffix, apply, applyKeys, applyPairs, applyTable, applyValues, apropos, argument, ascii, asin, asinh, assert, associatedPrimes, atEndOfFile, atan, atan2, autoload, backupFileRegexp, baseFilename, baseName, baseRings, basis, beginDocumentation, benchmark, betti, between, binomial, borel, break, cacheValue, cancelTask, capture, catch, ceiling, centerString, chainComplex, char, characters, check, clean, clearAll, clearEcho, clearOutput, close, closeIn, closeOut, code, codim, coefficient, coefficientRing, coefficients, cohomology, coimage, cokernel, collectGarbage, columnAdd, columnMult, columnPermute, columnSwap, columnate, combine, commandLine, commonRing, commonest, comodule, compactMatrixForm, complement, complete, components, compose, compositions, compress, concatenate, cone, conjugate, connectionCount, content, continue, contract, contract', copyDirectory, copyFile, copyright, cos, cosh, cot, cotangentSheaf, coth, cover, coverMap, cpuTime, createTask, csc, csch, current, currentDirectory, currentFileDirectory, currentFileName, currentLayout, currentLineNumber, currentPackage, currentString, currentTime, dd, debug, debugError, debugLevel, debugging, debuggingMode, decompose, deepSplice, default, defaultPrecision, degree, degreeLength, degrees, degreesMonoid, degreesRing, delete, demark, denominator, depth, describe, determinant, diagonalMatrix, dictionary, dictionaryPath, diff, diff', difference, differentiation, dim, directSum, disassemble, dismiss, divideByVariable, document, drop, dual, dumpdata, eagonNorthcott, echoOff, echoOn, edit, eigenvalues, eigenvectors, eint, elapsedTime, elapsedTiming, elements, encapDirectory, end, endPackage, endl, engine, engineDebugLevel, entries, environment, epkg, erf, erfc, error, errorDepth, euler, eulers, examples, exec, exit, exp, expm1, exponents, export, exportMutable, expression, extend, exteriorPower, factor, false, fileExecutable, fileExists, fileExitHooks, fileLength, fileMode, fileReadable, fileTime, fileWritable, fillMatrix, findFiles, findHeft, findSynonyms, first, firstkey, fittingIdeal, flagLookup, flatten, flattenRing, flip, floor, flush, fold, for, forceGB, fork, format, frac, fraction, frames, from, functionBody, gb, gbRemove, gbSnapshot, gbTrace, gcd, gcdCoefficients, genera, generateAssertions, generator, generators, genericMatrix, genericSkewMatrix, genericSymmetricMatrix, genus, get, getChangeMatrix, getGlobalSymbol, getNetFile, getNonUnit, getPrimeWithRootOfUnity, getSymbol, getWWW, getc, getenv, gfan, givaro, global, globalAssign, globalAssignFunction, globalAssignment, globalAssignmentHooks, globalReleaseFunction, gradedModule, gradedModuleMap, graphIdeal, graphRing, groebnerBasis, groupID, hash, hashTable, hashing, heft, height, help, hh, hilbertFunction, hilbertPolynomial, hilbertSeries, hold, homeDirectory, homogenize, homology, homomorphism, homomorphism', horizontalJoin, httpHeaders, hypertext, ideal, ideals, if, ii, image, imaginaryPart, incomparable, independentSets, indeterminate, index, indexComponents, indices, inducedMap, inducesWellDefinedMap, info, infoHelp, input, insert, installAssignmentMethod, installHilbertFunction, installMethod, installPackage, installedPackages, instance, instances, interpreterDepth, intersect, inverse, inversePermutation, irreducibleCharacteristicSeries, isANumber, isAffineRing, isBorel, isCanceled, isCommutative, isConstant, isDirectSum, isDirectory, isField, isFinite, isFinitePrimeField, isFreeModule, isGlobalSymbol, isHomogeneous, isIdeal, isInfinite, isInjective, isInputFile, isIsomorphism, isListener, isModule, isMonomialIdeal, isOpen, isOutputFile, isPolynomialRing, isPrime, isPrimitive, isPseudoprime, isQuotientModule, isQuotientOf, isQuotientRing, isReady, isReal, isRegularFile, isRing, isSkewCommutative, isSorted, isSquareFree, isSubmodule, isSubquotient, isSubset, isSurjective, isTable, isUnit, isWellDefined, jacobian, join, kernel, keys, kill, koszul, last, lastMatch, lcm, leadCoefficient, leadComponent, leadMonomial, leadTerm, length, lift, liftable, limitFiles, limitProcesses, lineNumber, lines, linkFile, list, listForm, listLocalSymbols, listSymbols, listUserSymbols, lngamma, load, loadDepth, loadPackage, loaddata, loadedFiles, loadedPackages, local, localDictionaries, locate, log, log1p, lookup, lookupCount, lrslib, makeDirectory, makeDocumentTag, makePackageIndex, map, markedGB, match, mathML, matrices, matrix, max, maxAllowableThreads, maxExponent, maxPosition, member, memoize, merge, mergePairs, method, methodOptions, methods, min, minExponent, minPosition, mingens, mingle, minimalBetti, minimalPresentation, minimalPresentationMap, minimalPresentationMapInv, minimalPrimes, minimizeFilename, minors, minus, mkdir, mod, module, modules, modulo, monoid, monomialIdeal, monomialSubideal, monomials, moveFile, multidegree, mutable, mutableIdentity, mutableMatrix, nauty, needs, needsPackage, net, netList, new, newCoordinateSystem, newNetFile, newPackage, newRing, newline, nextPrime, nextkey, norm, normaliz, not, notImplemented, null, nullaryMethods, nullhomotopy, numColumns, numRows, number, numerator, numeric, numgens, of, ofClass, on, oo, oooo, openDatabase, openDatabaseOut, openFiles, openIn, openInOut, openListener, openOut, openOutAppend, operatorAttributes, operators, options, order, override, pack, package, packages, pad, pager, pairs, pari, part, partition, partitions, parts, path, pdim, peek, peek', permanents, permutations, pfaffians, pi, pivots, plus, poincare, poincareN, position, positions, power, powermod, powers, precision, prefixDirectory, prefixPath, preimage, prepend, presentation, pretty, print, printString, printWidth, printingLeadLimit, printingPrecision, printingSeparator, printingTimeLimit, printingTrailLimit, processID, product, profile, profileSummary, projectiveHilbertPolynomial, promote, protect, prune, pruningMap, pseudoRemainder, pseudocode, pushForward, quit, quotient, quotient', quotientRemainder, quotientRemainder', radical, random, randomKRationalPoint, randomMutableMatrix, rank, read, readDirectory, realPart, realpath, recursionDepth, recursionLimit, reduceHilbert, regex, registerFinalizer, regularity, relations, relativizeFilename, remainder, remainder', remove, removeDirectory, removeFile, removeHook, removeLowestDimension, reorganize, replace, reshape, resolution, restart, return, returnCode, reverse, ring, rootPath, rootURI, roots, rotate, round, rowAdd, rowMult, rowPermute, rowSwap, rsort, run, runHooks, runLengthEncode, saturate, scan, scanKeys, scanLines, scanPairs, scanValues, schedule, schreyerOrder, scriptCommandLine, searchPath, sec, sech, seeParsing, select, selectInSubring, selectVariables, separate, separateRegexp, sequence, set, setEcho, setGroupID, setIOExclusive, setIOSynchronized, setIOUnSynchronized, setRandomSeed, setup, setupEmacs, sheaf, sheafExt, sheafHom, show, showClassStructure, showHtml, showStructure, showTex, showUserStructure, sin, singularLocus, sinh, size, size2, sleep, smithNormalForm, solve, someTerms, sort, sortColumns, source, splice, splitWWW, sqrt, stack, standardForm, standardPairs, stashValue, status, stderr, stdio, step, stopIfError, sublists, submatrix, submatrix', submatrixByDegrees, subquotient, subscript, subsets, substitute, substring, subtable, sum, super, superscript, support, switch, symbol, symlinkDirectory, symlinkFile, symmetricAlgebra, symmetricPower, synonym, syz, syzygyScheme, table, take, tally, tan, tangentSheaf, tanh, target, taskResult, tensor, tensorAssociativity, terms, tex, texMath, throw, time, times, timing, to, toAbsolutePath, toCC, toExternalString, toField, toList, toLower, toRR, toSequence, toString, toUpper, topCoefficients, topComponents, topLevelMode, topcom, trace, transpose, trim, true, truncate, truncateOutput, try, tutorial, typicalValues, ultimate, unbag, uncurry, undocumented, uniform, uninstallAllPackages, uninstallPackage, unique, unsequence, unstack, use, userSymbols, value, values, variables, varieties, variety, vars, vector, version, viewHelp, wedgeProduct, weightRange, when, while, width, wrap, xor, youngest, zero, zeta
	  }}
\usetikzlibrary{positioning}
\tikzset{
    state/.style={
           rectangle,
           rounded corners,
           draw=black,
           minimum height=2em,
           inner sep=0pt,
           text centered,
           },
}

\title{Gorenstein stable surfaces with $K_{X}^{2}=2$ and $\chi(\osh{X})=4$}
\author{Ben Anthes}
\date{December 2018}
\email{anthes@mathematik.uni-marburg.de}

\newcommand{\p}{\mathop{p}\nolimits}
\newcommand{\q}{\mathop{q}\nolimits}
\newcommand{\olMfrak}{\vphantom{Mfrak}\smash{\overline \Mfrak}}
\newcommand{\olF}{\vphantom{F}\smash{\overline F}}
\newcommand{\olMcal}{\vphantom{Mcal}\smash{\overline \Mcal}}
\newcommand{\olPcal}{\vphantom{Pcal}\smash{\overline \Pcal}}
\newcommand{\olone}{\vphantom{1}\smash{\overline1}}
\newcommand{\oltwo}{\vphantom{2}\smash{\overline2}}
\newcommand{\XXX}{\olMfrak_{2,4}^{\textnormal{Gor}}}
\newcommand{\YYY}[4]{\Nfrak_{1^{#1}\olone^{#2}2^{#3}\oltwo^{#4}}}
\newcommand{\XXXst}{\Mcal^{\textnormal{Gor}}_{2,4}}

\newcommand{\minmod}[1]{#1_{\textnormal{min}}}
\begin{document}
\pagenumbering{arabic}

\begin{abstract}
We define and study a concrete stratification of the moduli space of Gorenstein stable surfaces $X$ satisfying $K_{X}^2 = 2$ and $\chi(\mathcal{O}_{X}) = 4$, by first establishing an isomorphism with the moduli space of plane octics with certain singularities, which is then easier to handle concretely.
In total, there are 47 inhabited strata with altogether 78 components.
\end{abstract}

\maketitle
\tableofcontents

\section{Introduction}

One of the most important bounds on the geography of surfaces of general type is Noether's inequality $2\chi(\osh{X})\leq K_{X}^2+6$.
A minimal surface of general type satisfying equality here is said to be on the Noether line. 
Since they have been studied intensively by Horikawa~\cite{Horikawa:1976}, they are also called \emph{Horikawa surfaces}.
The smallest possible invariants on the Noether line are $K_{X}^2 = 2$ and $\chi(\osh{X}) = 4$ and it is a classical fact that (the canonical model of) the corresponding surfaces are double-covers of $\PP^{2}$, branched over an octic curve with at worst simple singularities, via the morphism defined by the canonical linear system $|K_{X}|$.
Conversely, a double-cover of the plane branched over an octic with at worst simple singularities gives an example of such a surface.
Therefore, the Gieseker-moduli space $\Mfrak_{2,4}$ of canonical models $X$ of surfaces of general type with invariants $K_{X}^2 = 2$ and $\chi(\osh{X}) = 4$ is in bijection with (in fact, isomorphic to) the moduli space of plane curves of degree $8$ with at worst simple singularities.

The subject of this article is the study of the modular compactification $\olMfrak_{2,4}$ of $\Mfrak_{2,4}$ parametrising stable surfaces $X$ with the same numerical invariants $K_{X}^{2} = 2$ and $\chi(\osh{X}) = 4$.
We refer to this as the \emph{KSBA-compactification}, for Koll\'{a}r and Shepherd-Barron \cite{KSB:1988} and Alexeev \cite{Alexeev:1994}.
We thereby continue the series of works by Franciosi, Pardini and Rollenske \cite{FPR:2015,FPR:2015b,FPR:2017} who investigated the moduli spaces parametrising Gorenstein stable surfaces $X$ with $K_{X}^2 = 1$ using similar methods.

We can only handle the \emph{Gorenstein} stable surfaces since this allows us to consider the canonical map.
As in the classical case, the canonical linear system $|K_{X}|$ defines a double-cover of the plane, branched over an octic curve; this is the content of Corollary~\ref{cor:is-double-cover-of-pp2}.
Conversely, if $B\subset \PP^2$ is a curve of degree $8$ such that the pair $(\PP,\tfrac{1}{2}B)$ is log-canonical, then the double-cover $X$ of $\PP^{2}$ branched over $B$, which is essentially unique since $\Pic(\PP^2)$ is torsion-free, is a Gorenstein stable surface satisfying $K_{X}^2 = 2$ and $\chi(\osh{X}) = 4$.
This way, we obtain an isomorphism between the moduli spaces of those plane octics and the moduli space $\XXX\subset \olMfrak_{2,4}$ of Gorenstein stable surfaces $X$ satisfying $K_{X}^2 = 2$ and $\chi(\osh{X}) = 4$; see Theorem~\ref{theorem:stack-isomorphism}.
This allows us to use the rich theory of plane curves and the computer algebra system Macaulay2 \cite{M2} to get some understanding of the boundary components of $\Mfrak_{2,4}$ in $\XXX$.

More precisely, we will define a stratification of $\XXX$ by means of three indicators: the degree of non-normality, the number and degrees of isolated irrational singularities and whether the irrational singularities are simply elliptic or cusps.
The interest in the third indicator, even though not necessary to understand the birational isomorphism type of the surface, comes from the relation with another stratification induced by the degeneration of mixed Hodge structures on $H^2(X)$ as defined by Green, Griffiths, Kerr, Laza and Robles \cite{GGR:2014,GGLR:201X,KR:201X,Robles:2016a,Robles:2016}.

We will see that all inhabited strata are of the expected dimension, but many of the numerically characterised strata decompose further into disjoint components; this is mostly reflected by the birational types of the minimal resolutions.
On the level of curves, the different components correspond to special configurations of the non-simple singularities.

Moreover, we will define a \emph{Hodge type} of our surfaces under investigation (Definition~\ref{def:Hodge type}).
For the stratification of the locus parametrising normal surfaces, the Hodge type is constant on the strata, as shown in Proposition~\ref{prop:hodge-diamond-classification}.
For the locus of non-normal surfaces, however, this is much more complicated, as we will indicate in Example~\ref{example:wild-Hodge type-on-non-normal1} and  Example~\ref{example:wild-Hodge type-on-non-normal2}.
This is why on the locus of non-normal surfaces, the stratification will not be fine enough to control the Hodge type.

The article is organised follows:
In Chapter~\ref{section:surface-geometry}, we investigate the geometry of the surfaces of interest, i.e., we prove that they are canonically double-covers of the plane, discuss the singularities they may have, prove some constraints on the possible birational isomorphism types and we study the mixed Hodge structure on second cohomology.

Going back and forth between curves and branched double-covers defines an isomorphism between our moduli space of interest $\XXX$ and the moduli space of certain plane curves; this will be the subject of Chapter~\ref{section:moduli-spaces}.
This moduli space has at least one more notable compactification, due to Hacking \cite{Hacking:2004}; in Chapter~\ref{section:further-remarks}, we will present a few remarks and questions in this direction.

Before that, in Chapter~\ref{section:stratification}, we will define and study the stratification; first, for the locus parametrising normal surfaces and then for the remaining part.
The full degeneration diagram would be incomprehensible, which is why we restrict the presentation to two fragments which give a sufficiently good idea of the situation.

To ease the flow of presentation, we have two appendices, one about the possible configurations of certain singularities on curves of degree $8$ or $6$, \ref{section:curves}, and some explanations concerning the Macaulay2-code, \ref{section:M2}.
The code can be obtained from a GitLab repository \cite{Anthes:2018}.
I also uploaded it to \url{http://arxive.org}, accompanying the source code to this article.

\subsection*{Acknowledgements}
This is my PhD thesis \cite{Anthes:thesis}.
To my supervisor, S\"onke Rollenske, I owe deep gratitude and appreciation.
He also spotted a mistake in \cite[Theorem~3.11]{Anthes:thesis} and generously provided help with fixing it for this slightly abbreviated version; cf. Remark~\ref{rmk:error}.
The thesis has benefited a lot from his continuous feedback, but also from conversations with Marco Franciosi, Andreas Krug, Colleen Robles and Michael L\"onne, for which I am grateful as well.
Finally, I am thankful for the funding which we have received from the Deutsche Forschungsgemeinschaft (DFG) through my supervisor's Emmy Noether-program 
\emph{Modulr\"aume und Klassifikation von algebraischen Fl\"achen und Nilmannigfaltigkeiten mit linksinvarianter komplexer Struktur.}

\subsection{Notations and conventions}
We work with schemes over the complex numbers $\CC$.
\emph{Varieties} are reduced and proper schemes of finite type over $\CC$ and a \emph{surface} is a purely two-dimensional variety in this sense.
A \emph{curve} is a possibly non-reduced projective scheme, purely of dimension $1$.
In some places, points (of schemes) are implicitly understood as $\CC$-rational points, but this is always clear from the context.
\subsubsection{Notation}
Let $X$ be a proper complex scheme of pure dimension $n$.
\begin{itemize}[nolistsep,label=--]
\item $\p_{g}(X) = h^{n}(\osh{X})$.
\item $\p_{a}(X) = (-1)^{n}(\chi(\osh{X})-1)$.
\item If $n = 1$, then $\deg(L) = \chi(L)-\chi(\osh{X})$ for all $L\in\Pic(X)$.
\item $\q(X) = h^1(\osh{X})$.
\item To avoid confusion between topological and holomorphic Euler characteristics, we use $\chi(-)$ only for coherent $\osh{X}$-modules and $\chi_{\textnormal{top}}(X)$ for the topological Euler characteristic of the analytification $X^{\textnormal{an}}$.
\item To ease notation, we use $H^*(X;\CC):=H^*(X^{\textnormal{an}};\CC)$ for the singular cohomology of the analytification.
\end{itemize}

\subsubsection{Semi-log-canonical varieties and pairs}
We briefly recall the definition of a semi-log-canonical pair (from Koll\'{a}r's exposition \cite[Chapter~5]{Kollar:2013}, see there for more details).
If a finite type scheme $X$ over $\CC$ satisfies the following two conditions, it is called \emph{demi-normal}.
\begin{enumerate}[label = \arabic*.]
\item $X$ satisfies Serre's condition $(S_{2})$, i.e., for every $x\in X$ we have
$$\depth_{\osh{X,x}}(\osh{X,x})\geq\min\{2,\dim(\osh{X,x})\}.$$
\item $X$ is regular or double normal crossing in codimension $1$, i.e., if $x\in X$ is the generic point of a sub-variety of codimension one, then either $\osh{X,x}$ is regular, or its completion $\osh{X,x}^\wedge$ with respect to the maximal ideal is isomorphic to the complete local ring $\CC[[x,y]]/(xy)$.
\end{enumerate}
For a demi-normal scheme $X$, with normalisation $\pi\colon \overline{X}\to X$, the \emph{conductor locus} $F:=\mathrm{supp}(\pi_{*}\osh{\overline{X}}/\osh{X})\subset X$ is purely one-codimensional, reduced and $\pi$ is generically a double-cover over $F$, as explained by Koll\'{a}r \cite[Section~5.1]{Kollar:2013}.
The ideal sheaf defining $F$, $\ann_{\osh{X}}(\pi_{*}\osh{\overline{X}}/\osh{X}) = \mathop{\mathcal{H}\hspace{-.3ex}\mathit{om}}_{\osh{X}}(\pi_{*}\osh{\overline{X}},\osh{X})$, is also an ideal in $\pi_{*}\osh{\overline{X}}$ which defines the conductor locus $\overline{F}\subset\overline{X}$ in the normalisation.
It is the reduced pre-image of $F$.

Note that a demi-normal scheme satisfies Serre's condition $(S_{2})$ and is Gorenstein at all points of codimension one, i.e., satisfies $(G_{1})$.
Therefore, there is a canonical sheaf $\omega_{X}$ and it is a divisorial sheaf which is locally free in codimension one.
In particular, we can choose a canonical Weil-divisor $K_{X}$ which is Cartier in codimension one.

A pair $(X,D)$ of a  variety $X$ and an effective $\QQ$-divisor $D$ on $X$ is \emph{semi-log-canonical} if $X$ is demi-normal, $F$ and the support of $D$ have no component in common, the divisor $K_{X}+D$ is $\QQ$-Cartier and the pair $(\overline{X},\pi_{*}^{-1}D+\overline{F})$ is log-canonical (cf.\ Koll\'{a}r \cite[Definition~2.8]{Kollar:2013}).

A variety $X$ is \emph{semi-log-canonical} if the pair $(X,0)$ is semi-log-canonical.
In particular, the canonical divisor $K_{X}$ is $\QQ$-Cartier then, i.e., $X$ is $\QQ$-Gorenstein.
A \emph{stable surface} is a projective, connected, semi-log-canonical surface $X$ whose canonical divisor $K_{X}$ is ample.
More generally, a \emph{stable log-surface} is a semi-log-canonical pair $(X,\Delta)$ where  $X$ is a connected and projective surface and such that $K_{X}+\Delta$ is ample.

We will also need the notion of Du Bois-singularities; see Koll\'{a}r \cite[Chapter~6]{Kollar:2013} for a concise introduction.
Since we will ultimately be concerned with Gorenstein stable surfaces, for our purposes, it is enough to note the following two facts.
For one, semi-log-canonical singularities are Du Bois, see Koll\'{a}r \cite[Corollary~6.32]{Kollar:2013} or Kov\'{a}cs, Schwede, Smith \cite[Theorem~4.16]{KSS:2010}.
Conversely, Du Bois-singularities which are demi-normal and Gorenstein are semi-log-canonical by Doherty \cite[Theorem~4.2]{Doherty:2008}.

\subsubsection{Stable surfaces and their normalisations---Koll\'{a}r's glueing}

Let $X$ be a stable surface with normalisation $\pi\colon \overline{X}\to X$ and conductor loci $F\subset X$ and $\olF\subset \overline{X}$.
Then $K_{\overline{X}}+\olF$ is an ample $\QQ$-Cartier divisor.
Moreover, the restriction of $\pi$ to $\overline{F}\to F$ is generically a double-cover and after passing to normalisations $\olF^\nu\to F^\nu$, it is the quotient of a Galois involution $\tau\colon \olF^\nu\to\olF^\nu$.
The surface $X$ can then be recovered from these data as the following diagram is a composition of push-outs:
$$
\begin{CD}
\olF^\nu @>\nu>> \olF @>>> \overline{X}\\
@VV{/\tau}V                         @VV\pi V          @VV\pi V\\
F^\nu @>\nu>> F @>>> X
\end{CD}
$$
This follows from Koll\'{a}r's Glueing Theorem \cite[Theorem~5.13]{Kollar:2013} and its proof there.
More precisely, this theorem captures when exactly the data $(\overline{X},\olF,\tau)$ arise from a stable surface $X$ (or a stable log-surface $(X,\Delta)$ in the presence of a boundary divisor $\overline{\Delta}$).
Moreover, in this correspondence, $X$ is Gorenstein if and only if the involution $\tau$ on $\olF^{\nu}$ induces a fixed-point free involution on the pre-images of the nodes of $\olF$, as shown by Franciosi, Pardini and Rollenske \cite[Addendum to Theorem~3.2]{FPR:2015}.

\subsubsection{Divisors on demi-normal varieties}\label{sec:divisors-on-demi-normal}
A ($\QQ$-)\emph{Weil divisor} $D$ on $X$ is a $\ZZ$-\ (respectively $\QQ$-)linear combination of integral sub-\-varieties of codimension one in $X$.
Its \emph{support} is the reduced union of all sub-varieties with non-trivial coefficient.
If $X$ is demi-normal, a divisor $D$ on $X$ is said to be \emph{well-behaved} if its support and the conductor locus $F\subset X$ do not share a common component.
An effective \emph{Cartier divisor} is a sub-scheme $D\subset X$ whose ideal sheaf $\osh{X}(-D)$ is invertible and $D$ is said to be \emph{almost-Cartier} (cf.\ Hartshorne \cite{Hartshorne:1994,Hartshorne:2007}), if the ideal sheaf is invertible at all points of codimension one, i.e., outside a closed sub-scheme of codimension at least two.

An effective almost-Cartier divisor  $D\subset X$ gives rise to an effective Weil divisor through $\sum_{C\subset X} \length(\osh{D,\eta_{C}})C$, where the sum runs through the integral closed sub-schemes $C\subset X$ and $\eta_{C}\in C$ is the generic point, the length being computed as $\osh{C,\eta_{C}}$-module.
We will, by abuse of language, call a Weil divisor \emph{almost-Cartier}, if it is the difference of two effective divisors arising from almost-Cartier divisors; furthermore, it is called Cartier if the corresponding almost-Cartier generalised divisor is Cartier.
As usual, a $(\QQ\text{-})$Weil divisor is called $\QQ$-Cartier if it has an integral multiple which is Cartier.
For example, if $L$ is a divisorial sheaf on a demi-normal scheme, a regular section $s\in H^0(X,L)$ gives rise to an effective almost-Cartier divisor $Z(s)$ with corresponding ideal sheaf $\mathrm{im}(s^\vee\colon L^\vee\to\osh{X})$.

\subsubsection{Numerical connectedness}
Recall that a Gorenstein curve $C$ is said to be \emph{numerically $m$-connected} if for every generically Gorenstein strict sub-curve $B\subset C$,
$\deg_{B}(\omega_{C}|_{B})-(2\p_{a}(B)-2)\geq m$
(cf.\ Catanese, Hulek, Franciosi, Reid \cite{CFHR:1999}).
This is a very useful generalisation of the classical notion numerical connectedness of curves on smooth surfaces.

\section{The geometry of the surfaces}
\label{section:surface-geometry}

In this first chapter we investigate the geometry of Gorenstein stable surfaces $X$ with $K_{X}^2 = 2$ and $\chi(\osh{X}) = 4$.
At first, we show that they all arise as double-covers of the plane, branched over some curve of degree $8$.
Then we identify the birational isomorphism types of the minimal resolutions, give characterisations of the singularities and we conclude with some results about the mixed Hodge structures on their second cohomology.

\begin{remark}\label{rmk:basic-numerics}
Let $X$ be a Gorenstein stable surface with $K_{X}^2=2$.
If $\chi(\osh{X}) \geq 4$, then the stable Noether inequality due to Liu and Rollenske \cite{LR-geo} implies $\p_{g}(X)<K_{X}^2+2$.
Thus, from
$$4 \leq \chi(\osh{X}) = 1-\q(X)+\p_{g}(X)\leq 4-\q(X)\leq 4$$
we conclude $\chi(\osh{X}) = 4$, $\q(X) = 0$ and $\p_{g}(X) = 3$.
Conversely, if $K_{X}^2 = 2$ and $\p_{g}(X)=3$, then $\q(X) = 0$ and so $\chi(\osh{X}) = 4$.
This will be shown in a manuscript in preparation by Franciosi, Pardini and Rollenske \cite{FPR:201X}.
\end{remark}

\subsection{The canonical linear system}
The aim of this section is to show that the canonical map of a Gorenstein stable surface $X$ satisfying $K_{X}^2 = 2$ and $\chi(\osh{X}) = 4$ is a double-cover of $\PP^2$ branched over an octic.

If $X$ is a reducible Gorenstein stable surface, it may happen that some non-trivial section of the canonical bundle is not regular; that is, it could vanish on a component.
In the case under consideration, however, this does not happen:

\begin{lemma}\label{lemma:general-canonical-curve}
Assume $X$ is a Gorenstein stable surface with $K_{X}^2 = 2$.
Then $X$ has at most two components and every non-trivial section of $\omega_{X}$ is regular.
If, in addition, $\p_{g}(X)\geq 2$, then $|\omega_{X}|$ has no fixed part and a general effective canonical divisor is well-behaved and reduced.
\end{lemma}
\begin{proof}
We consider the decomposition $X = \bigcup_{i=1}^s X_{i}$ into irreducible components and assume that $s\geq 2$.
Then there is a corresponding decomposition of the normalisation into disjoint components $\overline{X} = \coprod_{i=1}^s\overline{X}_{i}$ where $\overline{X}_{i}$ is the component over $X_{i}$.
The conductor locus $\overline{F}\subset\overline{X}$ accordingly decomposes as $\overline{F} = \bigcup_{i=1}^s\olF_{i}$, $\olF_{i}\subset\overline{X}_{i}$.

Since $\omega_{X}$ is ample, $2 = \omega_{X}^2 = \sum_{i=1}^s(\omega_{X}|_{X_{i}})^2$ is a sum of $s$ positive integers, so that we have to have $s=2$ and $\omega_{X}|_{X_{1}}^2 = \omega_{X}|_{X_{2}}^2 = 1$.
In particular, the invertible sheaves $\pi^{*}\omega_{X}|_{\overline{X}_{i}}\cong\omega_{\overline{X}_{i}}(\overline{F}_{i})$, $i=1,2$, are ample with $(\pi^{*}\omega_{X}|_{\overline{X}_{i}})^2 = 1$.
This furthermore implies that every member of $|\omega_{\overline{X}_{i}}(\overline{F}_{i})|$ is reduced and irreducible.

We now show that every non-trivial section of $\omega_{X}$ is regular, i.e., that the natural maps $p_{i}\colon H^0(X,\omega_{X})\to H^0(\overline{X}_{i},\omega_{\overline{X}_{i}}(\overline{F}_{i}))$, $i=1,2$, given by pull-back and restriction, are injective.
Since $\ker(p_{1})$ and $\ker(p_{2})$ only have the trivial element in common, it suffices to show that they agree.
To this end, note that the restriction of $p_{1}$, to $\ker(p_{2})$ factors through the inclusion $H^0(\overline{X}_{1},\omega_{\overline{X}_{1}})\to  H^0(\overline{X}_{1},\omega_{\overline{X}_{1}}(\overline{F}_{1}))$ (and likewise for $p_{2}$ in place of $p_{1}$).
Thus, it suffices to prove $H^0(\overline{X}_{i},\omega_{\overline{X}_{i}}) = 0$ for both, $i=1,2$.
But $(\overline{X}_{i},\olF_{i})$ is a stable log-pair with $\omega_{\overline{X}_{i}}(\olF_{i})^2 = 1$ and these are classified by Franciosi, Pardini and Rollenske \cite[Theorem~1.1]{FPR:2015}.
From this result it follows that $\p_{g}(\overline{X}_{i}) = 0$, as claimed.
Alternatively, it can be shown that a non-trivial section of $\omega_{X}$ had to be nowhere vanishing, hence $\olF_{i}^2 = \omega_{\overline{X}_{i}}(\olF_{i})^2 = 1$, in contradiction with Riemann-Roch.

It remains to show that if $\p_{g}(X)\geq 2$, then $|\omega_{X}|$ has no fixed part and a general member is reduced and well-behaved.
The generic fibre of a morphism with reduced source is reduced (see the Stacks Project \cite[\href{https://stacks.math.columbia.edu/tag/054Z}{Tag 054Z}]{stacks-project}).
Thus, if $\p_{g}(X)\geq 2$, a general member of $|\omega_{X}|$ is generically reduced, hence reduced.
Moreover, an effective and reduced member of the canonical linear system is well-behaved since a section of an invertible sheaf vanishes along the conductor with even multiplicity.

Since every member of $|\omega_{X}|_{X_{i}}|$ is irreducible, if the linear system $|\omega_{X}|$ would fix a curve $C\subset X_{i}$, then the restriction map $|\omega_{X}|\to |\omega_{X}|_{X_{i}}|$ could only be constant, mapping everything to $C$.
But assuming $\dim(|\omega_{X}|) = \p_{g}(X)-1 \geq 1$, the injective map $|\omega_{X}|\to |\omega_{X}|_{X_{i}}|$ is not constant.
Thus, $|\omega_{X}|$ cannot fix a curve.
This completes the proof.
\end{proof}

The following examples show that a special canonical curve could very well be non-reduced or non-well-behaved.

\begin{examples}~
\begin{enumerate}[label=\arabic*.]
\item \emph{(A canonical curve in the conductor)} When we obtain $X$ as a union of two copies of $\PP^2$ along a quartic $\overline{F}\in |\osh{\PP^2}(4)|$ with at worst nodal singularities, then $X$ is a Gorenstein stable surface with $K_{X}^2=2$ and $\chi(\osh{X}) = 4$ and the members of the canonical linear system are the unions of compatible lines in either plane.
Therefore, if $\overline{F}$ is a union of a general line and a smooth cubic, then the corresponding line in the conductor is a member of the canonical linear system.
\item \emph{(A well-behaved, non-reduced canonical curve)} Let $f\colon X\to\PP^2$ be a double-cover branched along the union of a smooth septic and a general line (that is, meeting the septic transversely).
Then $X$ is a normal Gorenstein stable surface with $K_{X}^2 = 2$ and $\chi(\osh{X}) = 4$ and the (non-reduced) pre-image of the line occurs as a member of the canonical linear system on $X$.
\end{enumerate}
\end{examples}

In the following result we collect the most basic numerical properties of an arbitrary member of the canonical linear system.
\begin{lemma}\label{lemma:on-canonical-curves}
Assume that $X$ is a Gorenstein stable surface satisfying $K_{X}^2=2$ and $\chi(\osh{X}) = 4$.
Let $C\in|\omega_{X}|$ be a canonical curve.
Then the following holds.
\begin{enumerate}[label=\alph*)]
\item The curve $C$ is Gorenstein and has at most two components.
\item The identities $h^0(\osh{C}) = 1$ and $\chi(\osh{C}) = -2$ hold. In particular, $C$ is connected and $ h^0(\omega_{C}) = p_{a}(C) = 3$.
\item The invertible sheaf $L := \omega_{X}|_{C}\in\Pic(C)$ is a square root of $\omega_{C}$, i.e., $L^{2}\cong\omega_{C}$, and we have $\chi(L)=0$ and $\deg(L)=h^{0}(L)=2$.
\end{enumerate}
\end{lemma}
\begin{proof}
\begin{enumerate}[label=\alph*),leftmargin=*]
\item Since $X$ is Gorenstein, every canonical curve $C$ on $X$ is Gorenstein. Furthermore, $K_{X}C = 2$ and $K_{X}$ has positive degree on each component of $C$, so that there can be at most two such.
\item By adjunction, $\omega_{C}$ is isomorphic to the cokernel of the inclusion $\omega_{X}^2(-C)\to \omega_{X}^2$.
The relevant fragment of the associated exact cohomology sequence, together with the Kodaira Vanishing Theorem for semi-log-canonical surfaces (Liu and Rollenske \cite[Proposition 3.1]{LR-pluri}) shows $h^1(\omega_{C}) = h^2(\omega_{X})$.
Serre duality implies $h^0(\osh{C}) = h^0(\osh{X}) = 1$ and applying the Riemann-Roch Formula for Cartier divisors on semi-log-canonical surfaces due to Liu and Rollenske \cite[Theorem 3.1]{LR-geo}  gives
$
	-\chi(\osh{C}) = \chi(\osh{X}(-C))-\chi(\osh{X}) = \tfrac{1}{2}(-C)(-C-K_{X}) = K_{X}^2 = 2.
$
Therefore, $\chi(\osh{C}) = -2 $ and $p_{a}(C) = 1-\chi(\osh{C}) =3$.
Finally, from $h^0(\osh{C}) = 1$ and Serre duality we conclude $p_{a}(C) = h^{1}(\osh{C}) = h^0(\omega_{C})$.
\item That $L$ is a square-root of $\omega_{C}$ follows from adjunction.
Since $L$ is the cokernel of the inclusion $\osh{X}\cong \omega_{X}(-C)\to \omega_{X}$ we get $\chi(L) = \chi(\omega_{X})-\chi(\osh{X}) = 0$ by Serre duality.
This readily implies $\deg(L) = p_{a}(C)-1 = 2$.
\end{enumerate}
This completes the proof.
\end{proof}

We aim to show:
\begin{proposition}\label{prop:C-is-honestly-hyperelliptic}
Let $C$ be a general, i.e., well-behaved and reduced, canonical curve on a Gorenstein stable surface $X$ satisfying $K_{X}^2 = 2$ and $\chi(\osh{X})=4$.
Then $C$ is numerically $4$-connected and honestly hyperelliptic, the double-cover $C\to\PP^1$ being defined by the sections of $\omega_{X}|_{C}$.
Moreover, if $C$ is reducible, then its two components are smooth rational curves.
\end{proposition}

In the proof, we will use the following result, which is well known in the smooth case.
In the present version, it is presumably also well known to some experts.
It follows from Rosenlicht's version of the Clifford Inequality for singular surfaces; see Rosenlicht \cite[Theorem~16]{Rosenlicht:1952} for the original proof or Kleiman and Martins \cite[Theorem~3.1]{KleimanMatrins:2009} for a modern account and further references.
\begin{lemma}\label{lemma:degree-one-line-bundle-gives-pp1}
Let $C$ be a reduced and irreducible Cohen-Macaulay curve.
An invertible sheaf $L\in\Pic(C)$ of degree one has at most two linearly independent global sections and if there are in fact two such, then $L$ is globally generated and the associated morphism $\phi_{|L|}\colon C\to \PP^1$ is an isomorphism.
%
\end{lemma}
\begin{proof}
Any $L$ with $h^{0}(L)\geq 2$ and $\deg(L)=1$ violates Clifford's inequality for singular curves $2(h^0(L)-1)\leq \deg(L)$; thus, anything but $h^1(L) = 0$ would lead to a contradiction.
But then $\chi(L) = h^0(L)\geq 2$ and so $\p_{a}(C) = 1-\chi(\osh{C}) = 1+\deg(L)-\chi(L)\leq 0$.
Hence, $C\cong \PP^{1}$ and $h^0(L) = 2$.
\end{proof}

\begin{proof}[Proof of Proposition~\ref{prop:C-is-honestly-hyperelliptic}]
Recall from the earlier Lemmas~\ref{lemma:general-canonical-curve} and~\ref{lemma:on-canonical-curves} that a general member $C\in|K_{X}|$ is Gorenstein, reduced, well-behaved and has $\p_{a}(C) = 3$.
We showed, furthermore, that $L := \omega_{X}|_{C}$ is a square-root of $\omega_{C}$ with $\deg(L) = h^{0}(L) = 2$.

If $C$ is irreducible, then it is clearly numerically $4$-connected and if $L$ were not globally generated, say at $x\in C$, then $L(-x)$ were of degree one with two linearly independent sections and so we had to have $C\cong \PP^1$ by Lemma~\ref{lemma:degree-one-line-bundle-gives-pp1}, in contradiction with $\p_{a}(C) = 3$.
Thus, if $C$ is irreducible, $\phi_{|L|}\colon C\to \PP^1$ is a morphism of degree two, as claimed.

If $C$ is reducible, then $C = C_{1}\cup C_{2}$ with $\deg(L|_{C_{1}}) = \deg(L|_{C_{2}}) = 1$.
Below, we will show that $h^0(L|_{C_{1}}), h^0(L|_{C_{2}}) \geq 2$.
Assuming this for the moment, it follows that $(C_{i},L|_{C_{i}})\cong(\PP^1,\osh{\PP^1}(1))$ for both $i=1,2$, by Lemma~\ref{lemma:degree-one-line-bundle-gives-pp1}.
In particular, $\deg(\omega_{C}|_{C_{i}})-(2\p_{a}(C_{i})-2) = 4$ for both $i=1,2$.
Thus, $C$ is numerically $4$-connected and the Curve Embedding Theorem due to Catanese, Franciosi, Hulek and Reid \cite[Theorem~1.1]{CFHR:1999} implies that $L$ is globally generated.

It remains to show that $h^0(L|_{C_{1}}), h^0(L|_{C_{2}}) \geq 2$.
Since $|K_{X}|$ has no fixed part, every component admits a non-trivial section of $L$; thus, $h^0(L|_{C_{1}}), h^0(L|_{C_{2}}) \geq 1$.
If we had $h^0(L|_{C_{1}}) = 1$, then we had to have a non-trivial section of $L$ vanishing on all of $C_{1}$.
But then the restriction of this section to $H^0(L|_{C_{2}})$ were non-trivial and vanishing on the separating conductor\footnote{The \emph{separating conductor} is the conductor locus of the partial normalisation $C_{1}\amalg C_{2}\to C$. Actually, in this particular case, where we know a-posteriori that $C_{1}\cong C_{2}\cong \PP^1$, this partial normalisation is already the full normalisation and so we are talking about the usual conductor locus.}, which had to have length $\deg(L|_{C_{2}}) = 1$ then, in contradiction with the fact that  in our case the separating conductor has to have even length: Since $C$ is Gorenstein, the length of the separating conductor on $C_{i}$ is precisely $\deg(\omega_{C}|_{C_{i}})-(2\p_{a}(C_{i})-2) = 2-(2\p_{a}(C_{i})-2)$, an even number.
Hence, $h^0(L|_{C_{1}})\geq 2$ and by the same argument for $C_{2}$ we also get $h^0(L|_{C_{2}})\geq 2$.
This finishes the proof.
\end{proof}

\begin{corollary}\label{cor:is-double-cover-of-pp2}
If $X$ is a Gorenstein stable surface satisfying $K_{X}^2 = 2$ and $\chi(\osh{X}) = 4$,
then the canonical linear system on $X$ is base-point free and realises $X$ as a double-cover of $\PP^2$ which is branched over an octic.
\end{corollary}
\begin{proof}
For a general canonical curve $C\in |K_{X}|$, the restriction $\omega_{X}|_{C}$ is base-point free by Proposition~\ref{prop:C-is-honestly-hyperelliptic}.
Since $h^1(\osh{X}) = \q(X) =0$ by assumption (cf.\ Remark~\ref{rmk:basic-numerics}), the restriction map $H^0(\omega_{X})\to H^0(\omega_{X}|_{C})$ is surjective; hence, so is the evaluation map $H^{0}(\omega_{X})\to H^{0}(\omega_{X}|_{p})$ for every $p\in C$.
This implies that $|K_{X}|$ is base-point free.
Since $K_{X}$ is ample, the canonical map $\phi := \phi_{|K_{X}|}\colon X\to \PP^2$ is finite and of degree $K_{X}^2 = 2$.
Finally, if $d\in\NN$ is such that the branch divisor is of degree $2d$, then we have to have $\phi_{*}\omega_{X} = \phi_{*}\phi^{*}(\omega_{\PP^2}(d)) = \omega_{\PP^2}(d)\oplus \omega_{\PP^2}$.
Thus, $3 = \p_{g}(X) = h^0(\omega_{\PP^2}(d)\oplus \omega_{\PP^2}) = h^0(\osh{\PP^2}(d-3))$, which is possible only if $d = 4$; hence, the branch divisor is an octic.
\end{proof}

The following corollary is equivalent to the former; we present a separate proof, though, because it shows how to compute the canonical ring from the canonical ring of a general canonical curve.

\begin{corollary}\label{cor:canonical-ring}
The canonical ring $R(X,\omega_{X})$ of a Gorenstein stable surface $X$ with $\chi(\osh{X}) = 4$ and $K_{X}^2 = 2$ is isomorphic to $\CC[x_{0},x_{1},x_{2},z]/(z^2-f_{8})$, where $x_{0}$, $x_{1}$ and $x_{2}$ are of degree $1$ and $z$ is of degree $4$ and where $f_{8}\in \CC[x_{0},x_{1},x_{2}]$ is a non-trivial homogeneous polynomial of degree $8$.
\end{corollary}
\begin{proof}
Let $x_{0}\in H^0(\omega_{X})$ be a general section, such that its associated canonical divisor $C = (x_{0})_{0}\in |K_{X}|$ is an honestly hyperelliptic curve of genus $3$, as granted by Proposition~\ref{prop:C-is-honestly-hyperelliptic}.
Then the section ring of the invertible sheaf $L = \omega_{X}|_{C}$ is isomorphic to $\CC[y_{1},y_{2},z]/(z^2-g_{8})$ for some homogeneous $g_{8}\in\CC[x_{1},x_{2}]$ of degree $8$, where $\deg(y_{1}) = \deg(y_{2}) = 1$ and $\deg(z) = 4$, as shown by Catanese, Franciosi, Hulek and Reid \cite[Lemma 3.5]{CFHR:1999}.
By Kodaira vanishing and since $\q(X) = 0$, the restriction map $R(X,\omega_{X})\to R(C,L)$ is surjective and the kernel is generated by $x_{0}$.
This is easily seen to imply that the associated map $\CC[x_{0},x_{1},x_{2},z]\to R(X,\omega_{X})$ is surjective with kernel generated by $z^2-f_{8}$ for some homogeneous $f_{8}\in \CC[x_{0},x_{1},x_{2}]$ of degree $8$ which lifts $g_{8}\in\CC[x_{1},x_{2}]$.
\end{proof}

We conclude with the remark that conversely, a sufficiently nice plane octic gives rise to a Gorenstein stable surface with the desired invariants.
Precisely, a double-cover $X\to \PP^2$ branched over a divisor $B$ is semi-log-canonical if and only of the pair $(\PP^{2},\tfrac{1}{2}B)$ is log-canonical, by Alexeev and Pardini \cite[Lemma~2.3]{AP:2012}.
For later reference, we deal not just with a single curve, but with a family of such.
\begin{proposition}\label{prop:curve-to-surface}
If $B\subset \PP_{S}^{2}$ is a flat family of octics, it is in particular a relative Cartier divisor.
Thus, we can form the double-cover $\Xscr \to \PP_{S}$ branched over $B$.
Assume that every fibre $B_{s}$, $s\in S(\CC)$, is such that the pair $(\PP^{2},\tfrac{1}{2}B_{s})$ is log-canonical.
Then the composition $f \colon \Xscr\to\PP_{S}^{2}\to S$ is a flat family of Gorenstein stable surfaces $\Xscr_{s}$, $s\in S(\CC)$, such that $K_{\Xscr_{s}}^2 = 2$ and $\chi(\osh{\Xscr_{s}}) = 4$.
Furthermore, $f_{*}\omega_{\Xscr/S}$ is free and $B\subset\PP_{S}^2$ can be recovered up to isomorphism as the branch divisor of the double-cover $\Xscr\to \PP_{S}(f_{*}\omega_{\Xscr/S})$.
\end{proposition}
\begin{proof}
At first, suppose that we are dealing with a single double-cover $\varphi\colon X\to\PP^{2}$, branched over an octic $B\in|\osh{\PP^{2}}(8)|$ such that the pair $(\PP^{2},\tfrac{1}{2}B)$ is log-canonical, so that $X$ is semi-log-canonical, as discussed above.
Note that since $\varphi$ is finite and $\omega_{X} = \varphi^{*}\omega_{\PP^{2}}(4) = \varphi^{*}\osh{\PP^{2}}(1)$, the canonical divisor on $X$ is Cartier and ample, i.e., $X$ is Gorenstein and stable.
The invariants $K_{X}^{2} = 2$ and $\chi(\osh{X}) = 4$ are computed as follows: $K_{X}$ defines a double-cover onto $\PP^{2}$, hence, $K_{X}^2 = 2$ and $\chi(\osh{X}) = \chi(\varphi_{*}\osh{X}) = \chi(\omega_{\PP^{2}}\oplus\omega_{\PP^{2}}(4)) = 1+3 = 4$.

Now let $S$ be a scheme of finite type over $\CC$ and let $B\subset \PP_{S}^{2}$ be a relative Cartier divisor of degree $8$.
Let $\Xscr\to\PP_{S}^{2}$ be the double-cover branched over $B$; more precisely, the cover taking the square-root of the section of $\osh{\PP_{S}^{2}/S}(8)$ defining $B$.
Since $\varphi_{*}\osh{\Xscr} = \osh{\PP_{S}^{2}}\oplus\osh{\PP_{S}^{2}}(-4)$ is locally free, the double-cover morphism is flat; thus, $f\colon \Xscr\to S$ is flat.

For every $s\in S(\CC)$, the fibre $\Xscr_{s}$ naturally identifies with the double-cover of $\PP^{2}$ branched over $B_{s}$.
Thus, $\Xscr_{s}$ is a Gorenstein stable surface if and only if $(\PP^{2},\tfrac{1}{2}B_{s})$ is log-canonical.
Let us suppose that this is indeed the case for all $s\in S(\CC)$.
Then all fibres $(f_{*}\omega_{\Xscr/S})(s) = H^{0}(\Xscr_{s},\omega_{\Xscr_{s}})$, $s\in  S(\CC)$, are $3$-dimensional and by naturality we observe $\Xscr$ as another double-cover $g\colon \Xscr\to\PP_{S}(f_{*}\omega_{\Xscr/S})$.
On the other hand, $f_{*}\omega_{\Xscr/S}$ is the direct image of $\varphi_{*}\omega_{\Xscr/S} = \omega_{\PP_{S}^{2}/S}\oplus\omega_{\PP_{S}^{2}/S}(4) = \osh{\PP_{S}^{2}/S}(-3)\oplus \osh{\PP_{S}^{2}/S}(1)$ along the projection $\PP_{S}^2\to S$; thus, $f_{*}\omega_{\Xscr/S}\cong \osh{S}\otimes H^0(\PP^{2},\osh{\PP^{2}}(1))\cong \osh{S}^{3}$.
Tracing through the sequence of morphisms involved shows that the induced isomorphism $\PP_{S}^{2}\cong \PP_{S}(f_{*}\omega_{\Xscr/S})$ identifies the double-covers $f$ and $g$.
This proves the claim.
\end{proof}

Recall that the \emph{log-canonical threshold} of an effective divisor $D\subset X$ on a variety $X$ is the number $\mathrm{lcth}(X,D) = \sup\{t\geq 0\mid(X,tD) \text{ is log-canonical} \}$, see Koll\'{a}r \cite[Section~8.2]{Kollar:2013}.
Thus, the a pair $(\PP^{2},\tfrac{1}{2}B)$ is log-canonical if and only if the log-canonical threshold of $B$ (in $\PP^{2}$) is at least $\tfrac{1}{2}$.
For brevity, we introduce a term for the plane curves with this property.

\begin{definition}
A plane curve $C\subset\PP^{2}$ is said to be \emph{half-log-canonical} if the pair $(\PP^{2},\tfrac{1}{2}C)$ is log-canonical; equivalently, if $\mathrm{lcth}(\PP^{2},C)\geq \tfrac{1}{2}$.
\end{definition}

This definition is independent of the embedding since the condition on the singularities is (analytically) local.

\subsection{The normalisation and the minimal resolution}

As we have shown above, every Gorenstein stable surface $X$ with $K_{X}^2 = 2$ and $\chi(\osh{X}) = 4$ arises as a double-cover of the plane, branched over an octic curve $B\in|\osh{\PP^2}(8)|$ such that the pair $(\PP^{2},\tfrac{1}{2}B)$ is log-canonical.
For such an octic, the ceiling $\ceil{\tfrac{1}{2}B}$ is supposed to be reduced, by definition.
That is, all integral components of $B$ have to appear with coefficient $\leq 2$.
In particular, every admissible $B$ decomposes as a sum $B = B'+2B''$ of (possibly trivial) reduced effective divisors.
Moreover, $B''$ can have at worst nodes and is smooth at the points of intersection with $B'$; a proof can be found in Koll\'{a}r \cite[Corollary~2.32]{Kollar:2013}, but this also follows from the classification of semi-log-canonical hypersurface singularities discussed in Proposition~\ref{prop:possible-curve-singularities} below.

If $\pi\colon \overline{X}\to X$ denotes the normalisation, then by a result of Pardini~\cite[Proposition~3.2]{Pardini:1991}, the composition $\overline{\varphi} = \varphi\circ\pi\colon\overline{X}\to\PP^{2}$ is the double-cover branched over $B'$ and the conductor loci $F\subset X$ and $\overline{F}\subset\overline{X}$ are the reduced pre-images of $B''$ under $\varphi$ and $\overline{\varphi}$, respectively.

This proves most of the following statement which we state for later reference.

\begin{proposition}\label{prop:description-of-normalisation}
Let $X$ be a Gorenstein stable surface with numerical invariants $K_{X}^{2} = 2$ and $\chi(\osh{X}) = 4$, let $\phi\colon X\to\PP^2$ be the canonical double-cover and let $B\subset \PP^2$ be the branch divisor.
Then the following holds:
\begin{enumerate}[label=\alph*)]
\item \label{just-for-the-proof1}The pair $(\PP^2,\tfrac{1}{2}B)$ is log-canonical; in particular, there is a unique decomposition $B = B'+2B''$ with $B',B''$ effective and reduced; $B''$ is nodal.
\item \label{just-for-the-proof2}The composition of $\phi\colon X\to\PP^2$ with the normalisation $\pi\colon \overline{X}\to X$ is the double-cover branched over $B'$ and the reduced pre-images of $B''$ in $\overline{X}$ and $X$ are the conductor loci $\overline{F}\subset\overline{X}$ and $F\subset X$, respectively.
\item \label{just-for-the-proof3}The morphism $(\phi\circ \pi)|_{\overline{F}}\colon \overline{F}\to B''$ is a double-cover branched over the Cartier divisor $B'|_{B''}$ and it factors through the isomorphism $\phi|_{F}\colon F\to B''$.
\end{enumerate}
\end{proposition}
\begin{proof}
We have discussed \emph{\ref{just-for-the-proof1}} and \emph{\ref{just-for-the-proof2}} right above the statement of the proposition.

Regarding \emph{\ref{just-for-the-proof3}}: That $(\phi\circ \pi)|_{\overline{F}}\colon \overline{F}\to B''$ is a double-cover branched over the Cartier divisor $B'|_{B''}$ follows from \ref{just-for-the-proof2}.
It also follows that $\phi|_{F}\colon F\to B''$ is of degree one, hence, generically an isomorphism.
It remains to show that it is an isomorphism everywhere.
Since $B''$ has only nodes and $\phi|_{F}$ is finite, it suffices to observe that $\phi|_{F}$ is bijective, which holds by construction.
\end{proof}

\subsubsection{The birational geometry of the minimal resolutions}
The birational geometry of the surfaces under investigation strongly depends on the number and degrees of the irrational singularities.
Recall that an isolated surface singularity $(X,x)$ is called \emph{irrational} if it is not rational; i.e., if the exceptional divisor $E$ in the minimal resolution $(Y,E)\to (X,x)$ has strictly positive arithmetic genus.
Its \emph{degree} is the negative of the self-intersection number $-E^2$.
From the classification of semi-log-canonical hypersurface singularities in dimension two, it follows that if $(X,x)$ is irrational and semi-log-canonical, then $\p_{a}(E) = 1$ (see Proposition~\ref{prop:possible-curve-singularities} below), i.e., $(X,x)$ is \emph{elliptic}.
It also follows that the only elliptic semi-log-canonical singularities occurring on double-covers of a smooth surface have to have degree $1$ or $2$.
We distinguish two cases:
If $E$ is a smooth elliptic curve, $(X,x)$ is said to be \emph{simply elliptic} and otherwise \emph{cuspidal} (or \emph{a cusp}).
The latter may happen if $E$ is a cycle of rational curves.

The holomorphic Euler characteristic of the resolution is easy to compute:
\begin{lemma}\label{lemma:chi-of-resolution}
Let $X$ be a log-canonical surface with $k$ irrational singularities.
For any resolution of singularities $Y$ of $X$, we have $\chi(\osh{Y}) = \chi(\osh{X})-k$.
\end{lemma}
\begin{proof}
Since the holomorphic Euler characteristic is a birational invariant of smooth surfaces, we can suppose that $f\colon Y\to X$ is the minimal resolution.
Then $f$ has connected fibres since $X$ is normal, $\chi(R^1f_{*}\osh{Y}) = h^{0}(R^1f_{*}\osh{Y}) = k$ by Liu, Rollenske \cite[Lemma~A.6]{LR-pluri} and $R^{i}f_{*}\osh{Y} = 0$ for all $i\geq 2$ for dimension reasons.
Thus, $\chi(\osh{Y}) = \chi(f_{*}\osh{Y})-\chi(R^{1}f_{*}\osh{Y}) = \chi(\osh{X})-k,$
as claimed.
\end{proof}

The following two results are part of a manuscript in preparation by Franciosi, Pardini and Rollenske \cite{FPR:201X}.
For simplicity, we restrict both to the relevant case.

\begin{proposition}[Franciosi, Pardini, Rollenske]\label{prop:normalisation-kodaira-neg-infinite}
Assume that $X$ is a normal Gorenstein stable surface satisfying $K_{X}^2 = 2$ and $\chi(\osh{X}) = 4$.
Let $k$ be the number of elliptic singularities of $X$ and suppose that the minimal resolution $Y$ of $X$ satisfies $\kappa(Y) = -\infty$.
Then either $Y$ is rational with $\chi(\osh{Y}) = 1$ and $k = 3$, or $\chi(\osh{Y}) = 0$ and $k = 4$; in the latter case, the four elliptic singularities are simple.
\end{proposition}
\begin{proof}[Proof after Franciosi, Pardini and Rollenske~\cite{FPR:201X}.]
With the same proof as in Franciosi, Pardini, Rollenske~\cite[Lemma~4.5]{FPR:2015} we get that either $Y$ is rational ($\chi(\osh{Y}) = 1$), or $\minmod{Y}$ is ruled of genus $1$ ($\chi(\osh{Y}) = 0$) and that in the latter case all elliptic singularities are simple.
Application of Lemma~\ref{lemma:chi-of-resolution} yields $\chi(\osh{Y}) = 4-k$, which completes the proof.
\end{proof}

\begin{theorem}[Franciosi, Pardini, Rollenske]\label{theorem:FPR}
Let $X$ be a normal Gorenstein stable surface satisfying $K_{X}^2 = 2$.
Denote by $f\colon Y\to X$ its minimal resolution and by $\sigma\colon Y\to \minmod{Y}$ a minimal model of $Y$.
Furthermore, let $d$ be the sum of the degrees of the elliptic singularities.
If $\kappa(Y)\geq 0$, then there are only the following possibilities:
\begin{enumerate}[label=\roman*)]
\item\label{theorem:FPR:itema} $Y = \minmod{Y}$ is of general type, $K_{Y}^2 = 2$ and $X$ is its canonical model ($d = 0$).
\item\label{theorem:FPR:itemb} $\minmod{Y}$ is of general type with $K_{\minmod{Y}}^2 = 1$, $\sigma\colon Y\to \minmod{Y}$ is the blow up in one point and $X$ has a unique elliptic singularity of degree $1$ ($d = 1$).
\item\label{theorem:FPR:itemc} $Y = \minmod{Y}$ is properly elliptic ($\kappa(Y) = 1$); in this case, $d = 2$.
\item\label{theorem:FPR:itemd} $\kappa(Y) = 0$ or $1$,  $\sigma\colon Y\to \minmod{Y}$ is a blow-up in one point and $d = 3$.
\item\label{theorem:FPR:iteme} $\kappa(Y) = 0$, $\sigma\colon Y\to \minmod{Y}$ is a sequence of two blow-ups and $d = 4$.
\end{enumerate}
\end{theorem}
\begin{proof}[Proof after Franciosi, Pardini and Rollenske~\cite{FPR:201X}.]
Let $E_{i}\subset Y$, $i=1,\dots,n$, be the exceptional curves over the elliptic singularities and let $G = \sum_{i=1}^mG_{i}\subset Y$ be the exceptional divisor contracted by $\sigma\colon Y\to Y_{\textnormal{min}}$.
Then the canonical divisor can be written in two different ways as $K_{Y} = f^{*}K_{X}-\sum_{i=1}^nE_{i}$ and $K_{Y} = \sigma^{*}K_{Y_{\textnormal{min}}}+G$.
The sum of degrees can be written as $d = K_{Y}\sum_{i=1}^n E_{i} =  -\sum_{i=1}^n E^2_{i}\geq n$.
We have $G_{i}f^{*}K_{X}\geq 1$ for each component $G_{i}\subset G$ since $K_{X}$ is ample and since no component of $G$ is contracted by $f$.
Moreover, every $(-1)$-curve $G_{i}\subset G$ satisfies $G_{i}E\geq 2$, for $-1 = G_{i}K_{Y} = G_{i}f^{*}K_{X}-G_{i}E\geq 1-G_{i}E$.
In particular, $GE> m$ unless $m = 0$.
We introduce the two central (in-)equalities:
\begin{align*}
d &= K_{Y}E = K_{Y}(f^{*}K_{Y}-K_{Y}) = f^{*}K_{X}^2-K_{Y}^2 = 2-K_{\minmod{Y}}^2+m\\
2 &= K_{Y}f^{*}K_{X} = \sigma^{*}K_{\minmod{Y}}f^{*}K_{X}+Gf^{*}K_{X} \geq \sigma^{*}K_{\minmod{Y}}f^{*}K_{X}+m\geq m
\end{align*}
Using $K_{\minmod{Y}}^2\geq 0$ (from $\kappa(Y)\geq 0$), they yield $d\leq m+2 \leq 4$.
The cases \ref{theorem:FPR:itema}--\ref{theorem:FPR:iteme} correspond to the possible values for $d=0,\dots,5$, respectively:

Let $d = 0$. Then $K_{Y} = f^{*}K_{X}$ is big and nef; this is case \ref{theorem:FPR:itema}.

Let $d = 1$. Then $K_{\minmod{Y}}^2  = 1+m\geq 1$, so that $K_{\minmod{Y}}$ is minimal of general type.
Moreover, $1 = d = GE > m$; thus, $m = 0$.
This is case \ref{theorem:FPR:itemb}.

Let $d = 2$. Then $K_{\minmod{Y}}^2 = m$.
But $m \geq 1$ is impossible, for $1 = d = K_{Y}E = GE+\sigma^{*}K_{\minmod{Y}}E> m$.
Hence, we have to have $Y = \minmod{Y}$ and $K_{Y}^2 = m = 0$.
On the other hand, $\kappa(Y) \not= 0$, since $K_{Y}f^{*}K_{X} = 2 > 0$; thus, $\kappa(Y) = 1$, as in case \ref{theorem:FPR:itemc}.

Let $d = 3$.
Then $K_{\minmod{Y}}^2 = m-1$, hence $m\geq 1$.
In fact, $m = 1$:
If we had $m \geq 2$, then $\minmod{Y}$ would be minimal of general type and by $3 = K_{Y}E = K_{\minmod{Y}}\sigma_{*}E+GE$ and $GE>m\geq2$ would imply $K_{\minmod{Y}}\sigma_{*}E = 0$.
In particular, we would have to have $K_{\minmod{Y}}\sigma_{*}E_{i} = 0$ for each component $E_{i}$ of $E$.
But since $\minmod{Y}$ had to be minimal of general type, the components $E_{i}$ had to be rational then, which is impossible.
Therefore, $m = 1$ and $K_{\minmod{Y}}^2 = 0$, corresponding to case \ref{theorem:FPR:itemd}.

Finally, let $d = 4$.
Then $m = 2$ and $K_{\minmod{Y}}^2 = 0$, for $0\leq K_{\minmod{Y}}^2 = m-2$ and $m\leq 2$.
To complete the proof, it is left to show that $\kappa(Y) = 0$.
Indeed, for $2 = \sigma^{*}K_{\minmod{Y}}f^{*}K_{X}+m$ to hold, we have to have $\sigma^{*}K_{\minmod{Y}}f^{*}K_{X} = 0$; together with $K_{\minmod{Y}}^2 = 0$ this implies that $\kappa(Y) = 0$ and we arrive at case \ref{theorem:FPR:iteme}.
\end{proof}

The birational classification in the non-normal case will be established as needed later.
It will turn out that the only reducible normalisation is $\PP^2\amalg \PP^{2}$ and the possible irreducible normalisations are K3-surfaces, rational, or ruled over a curve of genus $1$.

\subsection{The singularities}
Since double-covers of smooth varieties branched over Cartier divisors are sub-varieties of the total space of a line bundle, defined by a single regular equation, the surfaces under investigation have to have hypersurface singularities, if any.
Therefore, the classification of semi-log-canonical hypersurface singularities (see Liu, Rollenske \cite{LR-hypersurfaces}) gives a complete list of analytic germs of singular points we might get.
Since we are dealing with double-cover singularities, we only need to consider those of multiplicity two.
For simplicity, we restrict our attention to the singularities of the branch curves.

\begin{table}\centering
\begin{tabular}{cccc}
\toprule
Symbol & Equation in $\CC[x,y]$ & Conditions & $\mu$\\ 
\midrule 
$A_{n}$ & $x^2+y^{n+1}$ & $n\geq 1$ & $n$ \\
$D_{n}$ & $y(x^2+y^{n-2})$ & $n\geq 4$  & $n$ \\
$E_{6}$ & $x^3+y^4$ &  & $6$ \\
$E_{7}$ & $x^3+xy^3$ & & $7$ \\
$E_{8}$ & $x^3+y^5$ & & $8$\\
\hline
$X_{9}$ & $x^4+\lambda(xy)^2+y^4$ & $\lambda^2\not=4$ & $9$ \\
$J_{10}$ & $x^3+\lambda(xy)^2+y^6$ & $4\lambda^3+27\not=0$ & $10$ \\
\hline
$X_{p} = T_{2,4,p-5}\,$ & $x^{4}+(xy)^2+y^{4+p-9}$ & $p\geq 10$ & $p$ \\
$\;\;Y_{r,s} = T_{2,4+r,4+s}$ & $x^{4+r}+(xy)^2+y^{4+s}$ & $r,s\geq 1$ & $9+r+s$ \\
$J_{2,p} = T_{2,3,p+6}\;$ & $x^3+(xy)^2+y^{6+p}$ & $p\geq 1$ & $10+p$ \\
\hline
$A_{\infty}$ & $x^2$& & $0$ \\
$D_{\infty}$ & $x^2y$& & $1$ \\
$J_{2,\infty}$ & $x^3+(xy)^2$ & & $4$ \\
$X_{\infty}$ & $x^4+(xy)^2$ & & $5$ \\
$Y_{r,\infty}$ & $x^{r+4}+(xy)^2$ & $r\geq 1$ & $r+5$ \\
$Y_{\infty,\infty}$ & $(xy)^2$&&$4$ \\
\bottomrule
\end{tabular}
\caption{The classification of half-log-canonical curve singularities. We refer to  Arnold, Gusein-Zade, Varchenko \cite[Chapter~15]{AGV:2012} for the notation.}
\label{table:singularity-classification}
\end{table}
\begin{proposition}\label{prop:possible-curve-singularities}
Let $C\subset \PP^2$ be a plane curve of degree $8$.
Then the double-cover of $\PP^2$ branched over $C$ is a Gorenstein stable surface $X$ if and only if the analytic germs of the singular points of $C$ are among those listed in \ref{table:singularity-classification}.
\end{proposition}
\begin{proof}
This is just the relevant part (multiplicity $2$) of the list given by Liu and Rollenske \cite{LR-hypersurfaces} after appropriate transformations where necessary and with the refinement of the series $T_{2,\bullet,\bullet}$ into $X_{\bullet}$, $J_{2,\bullet}$ and $Y_{\bullet,\bullet}$.
\end{proof}

\begin{remark}
Of course, du Val-singularities $A_{\bullet}$, $D_{\bullet}$, $E_{\bullet}$ on the branch curve correspond to canonical singularities on the surface. \ref{table:surface-vs-curve-singularities} provides a correspondence for the remaining types of singularities.
\begin{table}\centering
\begin{tabular}{cll}
\toprule
{Symbol} & \multicolumn{1}{c}{Branch curve singularity} & \multicolumn{1}{c}{Double-cover singularity} \\ 
\midrule
$X_{9}$ & ordinary quadruple-point & simply elliptic of degree $2$ \\

$X_{\bullet}$ & \multirow{2}{*}{degenerate quad\-rup\-le-point} & \multirow{2}*{cuspidal elliptic of degree $2$} \\
$Y_{\bullet,\bullet}$ \\
\hline
$J_{10}$ & non-degenerate $[3;3]$-point & simply elliptic of degree $1$ \\
$J_{2,\bullet}$ & degenerate $[3;3]$-point & cuspidal elliptic of degree $1$ \\
\hline
$A_{\infty}$ & double-line & double normal crossing \\
$D_{\infty}$ & double-line + transversal line & pinch point\\
$J_{2,\infty}$ & double-line + tangential line & \multirow{3}{*}{degenerate cusps}\\
$X_{\infty},Y_{\bullet,\infty}$&double-line + double-point\\
$Y_{\infty,\infty}$ & transversely meeting double-lines\\
\bottomrule
\end{tabular}
\caption{Dictionary: branch curve singularities $\leftrightarrow$ surface
singularities}
\label{table:surface-vs-curve-singularities}
\end{table}
\end{remark}

\begin{remark}
If a Gorenstein stable surface $X$ with $K_{X}^2 = 2$ and $\chi(\osh{X}) = 4$ has an $A_{n}$- or $D_{n}$-singularity, then $n\leq 48$, since the maximal Milnor number of a singular point of the branch curve of the canonical double-cover, which has degree $8$, does not exceed $49$ and this maximal number is attained only by the union of eight concurrent lines (cf.\ Lemma~\ref{lemma:milnor-number-stuff}), which is of course neither $A_{n}$ nor $D_{n}$.
(Similar bounds exist for the other families listed above.)
However, also the upper bound $n\leq 48$ is most probably not sharp and determining the maximal $n$ such that there exists a plane curve $C$ of fixed degree with an $A_{n}$-singularity, e.g., is a hard question\footnote{This was pointed out to the author by Michael Lönne.}.
For example, the maximal $A_{n}$-singularity on a quintic is for $n = 12$, e.g., by Wall \cite{Wall:1996}, and on a sextic, the maximal $A_{n}$ is for $n = 19$, which follows from Yang's classification \cite{Yang:1996}.
This seems to be about everything that is known in this direction, at least according to a related discussion on MathOverflow answered by user JNS \cite{MO:singularities}.
\end{remark}

\subsection{The mixed Hodge structure on $H^2(X)$}\label{subsection:hodge}
The stratification we will define and study later is motivated by recent work (partially in progress) of Green, Griffiths, Kerr, Laza and Robles \cite{GGR:2014,GGLR:201X,KR:201X,Robles:2016a,Robles:2016} about degenerations of Hodge structures.
Roughly, there should be a stratification of the moduli space of our surfaces under investigation, according to the type of polarised mixed Hodge structure on $H^2(X)$.
It should be noted that the details about this Hodge-theoretic stratification are subject to work in progress.
Therefore, we can only give an informal description.
We refer to Robles' exposition \cite{Robles:2016} and the references therein for more details; for the basic theory of mixed Hodge structures see Durfee's short introduction \cite{Durfee:1983} and the comprehensive account by Peters and Steenbrink \cite{PS:2008}.

Given a flat Gorenstein degeneration $\Xscr\to S$, $\Xscr_{s}$ smooth projective, we can associate a \emph{limiting polarised mixed Hodge structure} with the family of Hodge structures $H^2(\Xscr_{s};\CC)$.
The Deligne splitting gives an $\RR$-split polarised mixed Hodge structure.
Furthermore, representation theory gives rise to a relation among these Hodge diamonds, called the \emph{polarised relations}.
They reflect which Hodge structures are more degenerate than others.

Since $h^{2,0}(X) = h^{0,2}(X) = \p_{g}(X) = 3$ for a smooth surface $X$ of general type satisfying $K_{X}^2 = 2$ and $\chi(\osh{X}) = 4$, our case of interest corresponds to the Hodge numbers $h = (3,h^{1,1},3)$.
That is, the Deligne splitting $H^2(X;\CC) = \bigoplus_{p,q}I^{p,q}$ of the mixed Hodge structures has a Hodge diamond (indicating $\dim(I^{p,q})$) of the form
\begin{center}
\begin{tikzpicture}[every circle/.style={radius = 0.05}]
\draw [->] (0,0) -- (2.25,0);
\draw [<-] (0,2.25) -- (0,0);
\node [above] at (0,2.25) {$\scriptstyle{p}$};
\node [right] at (2.25,0) {$\scriptstyle{q}$};

\draw [fill] (0,2) circle;
\draw [fill] (0,1) circle;
\draw [fill] (0,0) circle;
\draw [fill] (1,2) circle;
\draw [fill] (1,1) circle;
\draw [fill] (1,0) circle;
\draw [fill] (2,2) circle;
\draw [fill] (2,1) circle;
\draw [fill] (2,0) circle;

\node [below left] at (0,0) {$\scriptstyle{r}$};
\node [above right] at (2,2) {$\scriptstyle{r}$};
\node [left] at (0,1) {$\scriptstyle{s}$};
\node [below] at (1,0) {$\scriptstyle{s}$};
\node [right] at (2,1) {$\scriptstyle{s}$};
\node [above] at (1,2) {$\scriptstyle{s}$};
\node [left] at (0,2) {$\scriptstyle{3-r-s}$};
\node [below] at (2,0) {$\scriptstyle{3-r-s}$};
\node [above] at (1,1) {$\scriptstyle{h^{1,1}-r-2s}$};
\end{tikzpicture}
\end{center}
where $r,s\geq 0$, $r+s\leq 3$ and $r+2s\leq h^{1,1}$.
With this diamond denoted by $\lozenge_{r,s}$, the polarised relation is defined as $\lozenge_{r,s}\leq \lozenge_{t,u}$ if and only if $r\leq t$ and $r+s\leq t+u$, cf.\ Robles \cite[Example~4.22]{Robles:2016}.
This is illustrated in the degeneration diagram \ref{fig:hodge-diagram-diagram}.
We will ignore $h^{1,1}$, just as we will ignore canonical surface singularities. 

\begin{figure}\centering
\begin{tikzpicture}
 \node (00) 
 {$\lozenge_{0,0}$};
 \node[right of=00,node distance=1.5cm] (01)
 {$\lozenge_{0,1}$};
 \node[below right of=01,node distance=1.5cm] (02)
 {$\lozenge_{0,2}$};
 \node[above right of=01,node distance=1.5cm] (10)
 {$\lozenge_{1,0}$};
 \node[right of=02,node distance=1.5cm] (03)
 {$\lozenge_{0,3}$};
 \node[right of=10,node distance=1.5cm] (11)
 {$\lozenge_{1,1}$};
 \node[ right of=03,node distance=1.5cm] (12)
 {$\lozenge_{1,2}$};
 \node[right of=11,node distance=1.5cm] (20)
 {$\lozenge_{2,0}$};
 \node[,below right of=20,node distance=1.5cm] (21)
 {$\lozenge_{2,1}$};
 \node[right of=21,node distance=1.5cm] (30)
 {$\lozenge_{3,0}$};
 
 \path (00) edge (01);
 \path (01) edge (10)
 			edge (02);
 \path (10) edge (11);
 \path (02) edge (11)
  			edge (03);
 \path (03) edge (12);
 \path (11) edge (20)
 			edge (12);
 \path (20) edge (21);
 \path (12) edge (21);
 \path (21) edge (30); 
\end{tikzpicture}
\caption{Degeneration diagram for the Hodge types.}\label{fig:hodge-diagram-diagram}
\end{figure}
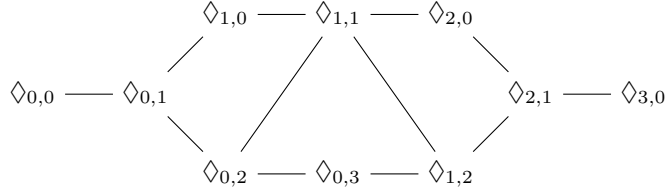

Dolgachev \cite{Dolgachev:1980} has introduced the notion of \emph{cohomologically insignificant degenerations}: If $\Xscr\to\Delta$ is a flat and projective family of varieties, where $\Delta\subset \CC$ is the unit disc and where all fibres $\Xscr_{s}$, $s\in\Delta^* = \Delta-0$ are smooth, we can compare Deligne's natural mixed Hodge structure on the cohomology of the special fibre $H^n(\Xscr_{0};\RR)$ and the limiting mixed Hodge structure on $H^n(\Xscr_{s};\RR)$, $s\not=0$, via the specialisation map.
The variety $\Xscr_{0}$ is said to be \emph{cohomologically $n$-insignificant} if these mixed Hodge structures on $H^{n}$ agree on $(p,q)$-components where $pq = 0$ for all such families with $\Xscr_{0}$ as special fibre.
By a result of Steenbrink \cite[Theorem~2]{Steenbrink:1980}, a projective variety with at worst du Bois singularities is cohomologically insignificant, that is, cohomologically $n$-insignificant for all $n$.
This applies in particular to semi-log-canonical surfaces since they are Du Bois (cf.\ Koll\'{a}r \cite[Corollary~6.32]{Kollar:2013} or Kov\'{a}cs, Schwede, Smith \cite[Theorem~4.16]{KSS:2010}).
For our purposes, it is therefore enough to work with Deligne's mixed Hodge structure.
\begin{definition}\label{def:Hodge type}
Let $X$ be a Gorenstein stable surface satisfying $K_{X}^2 = 2$ and $\chi(\osh{X}) = 4$.
Then $X$ is said to be of \emph{Hodge type $\lozenge_{r,s}$} if $r = \dim(H^2(X;\CC))^{(0,0)}$ and $s = \dim(H^2(X;\CC))^{(1,0)}$, where $(H^2(X;\CC))^{(p,q)}$ is the $(p,q)$-component of Deligne's mixed Hodge structure on $H^2(X;\CC)$.
\end{definition}

It can be shown that the moduli space under investigation in the forthcoming sections is stratified according to the Hodge type; in fact, the irrationality stratification defined below gives a refinement of this stratification, as follows from Proposition~\ref{prop:hodge-diamond-classification}.

For the computations of Hodge types of a Gorenstein stable surface $X$, we have to know the Hodge structure on its minimal resolution.
We will do the computation for our surfaces of interest as soon as we know their minimal resolutions.
To get an idea of how this is done, we show how to read off the number of cuspidal elliptic singularities from the Hodge type.

Let $X$ be a normal Gorenstein stable surface satisfying $\p_{g}(X) = 3$ and with irrational singularities $p_{1},\dots,p_{n}$, of which precisely $0\leq m\leq n$ are cusps.
Let $f\colon Y\to X$ be the resolution at the $p_{i}$, so that $Y$ has only rational singularities.
We denote the exceptional curve over $p_{i}$ by $E_{i}$ and their disjoint unions by $D = \coprod_{i=1}^n\{p_{i}\}$ and $E = \coprod_{i=1}^nE_{i}$.
Moreover, we let $i\colon D^{\textnormal{an}}\to X^{\textnormal{an}}$ and $j\colon E^{\textnormal{an}}\to Y^{\textnormal{an}}$ denote the inclusion maps.
Then we get a Mayer--Vietoris exact sequence of mixed Hodge structures, by Peters and Steenbrink \cite[Corollary-Definition 5.37]{PS:2008}
$$H^k(X;\CC)\xrightarrow{(f^{*},i^{*})} H^{k}(Y;\CC)\oplus H^{k}(D;\CC)\xrightarrow{j^{*}-f|_{E}^{*}} H^{k}(E;\CC)\to H^{k+1}(X;\CC).$$
Focussing on $H^{2}(X;\CC)$ and using that $H^1(D) = H^2(D) = 0$ for dimension reasons, we get the following exact sequence:
$$H^1(Y;\CC)\to \bigoplus_{i=1}^nH^1(E_{i};\CC)\to H^2(X;\CC)\xrightarrow{} H^2(Y;\CC).$$
If $p_{i}$ is simply elliptic, then $E_{i}$ is an elliptic curve and $H^1(E_{i})$ carries a pure Hodge structure of weight $1$ with $h^{1,0}(E_{i}) = 1$.
If $p_{i}$ is a cusp, so that $E_{i}$ is a cycle of rational curves, then $H^1(E_{i})$ is one-dimensional and its mixed Hodge structure is concentrated in weight $0$.
Concerning the mixed Hodge structure on $H^2(Y)$, since we are only interested in the $(p,0)$-components, we can pretend that $Y$ is regular since the rational singularities do not contribute.

Since $(H^1(Y))^{(0,0)} = (H^2(Y))^{(0,0)} = 0$, we conclude that $(H^{2}(X))^{(0,0)}$ is isomorphic to the $(0,0)$-component of $\bigoplus_{i=1}^nH^1(E_{i})$.
That is, $\dim(I^{(0,0)}) = m$, the number of cusps.
Likewise, since $H^2(Y)$ has no part of weight $1$, the part of weight $1$ in $H^2(X)$ entirely comes from $\bigoplus_{i=1}^nH^1(E_{i})$, so that $\dim(I^{(1,0)})\leq n-m$ is at most the number of simply elliptic singularities.
To actually compute the dimension of the $(1,0)$-component, we have to know more about the map $H^1(Y)\to\bigoplus_{i=1}^nH^1(E_{i})$.
So far, this discussion shows:
\begin{lemma}\label{lemma:MHS-computation}
Let $X$ be a normal Gorenstein stable surface with $\p_{g}(X) = 3$ having exactly $r$ cuspidal elliptic singularities.
Then $X$ is of Hodge type $\lozenge_{r,s}$ for a certain $0\leq s\leq 3-r$.
In this case the number of simply elliptic singularities is at least $s$.
\end{lemma}

\section{Remarks about the moduli space}
\label{section:moduli-spaces}

Our moduli space of interest is the KSBA-compactification of the Gieseker moduli space $\Mfrak_{2,4}$ of canonical models of surfaces of general type with invariants $K_{X}^2 = 2$ and $\chi(\osh{X}) = 4$.
For our techniques to apply, we restrict to the open locus of Gorenstein surfaces $\XXX\subset \olMfrak_{2,4}$.
Since this extra condition happens to simplify the definition of the moduli problem, we recall the details only for the space of Gorenstein stable surfaces.

Let $\XXXst$ be the category whose objects are pairs $(T,f\colon\Xscr\to T)$ consisting of a scheme $T$ of finite type over $\CC$ and a flat family $f\colon\Xscr\to T$ of Gorenstein stable surfaces $\Xscr_{t}$, $t\in T(\CC)$, all satisfying $K_{\Xscr_{t}}^2 = 2$ and $\chi(\osh{\Xscr_{t}}) = 4$.
The morphisms are fibre squares, as usual, so that the codomain fibration exhibits $\XXXst$ as a category fibred in groupoids over the category of complex schemes of finite type.
According to the seminal works of Koll\'{a}r, Shepherd-Barron \cite{KSB:1988,Kollar:1990} and Alexeev \cite{Alexeev:1994,AM:2004}, $\XXXst$ is a separated Deligne--Mumford stack (in the \'{e}tale topology), coarsely represented by a quasi-projective scheme $\XXX$.

\begin{remark}
If we consider only smoothable surfaces in $\olMfrak_{2,4}$, then $\XXX$ is dense, since every Gorenstein stable surface $X$ with $K_{X}^2 = 2$ and $\chi(\osh{X}) = 4$ is smoothable by Proposition~\ref{cor:is-double-cover-of-pp2}.
\end{remark}

We will now explain how the results of the previous chapter relate $\XXX$ to the following moduli space of polarised curves.
The linear system of plane octic curves $H := |\osh{\PP^{2}}(8)| (=\mathrm{Hilb}^{8n-20}(\PP^2))$ comes with its universal family $\Bscr \subset H\times\PP^2$, which is a relative Cartier divisor with respect to the projection $p_{1}\colon H\times\PP^2 \to H$.
Performing the relative double-cover branched over $\Bscr$ yields a flat, projective family $\pi\colon\Xscr\to H$ of two-dimensional schemes.
The locus $U\subset H$ parametrising curves $B$ such that the fibre $\Xscr_{B}\subset\Xscr$ is semi-log-canonical is precisely the locus of half-log-canonical curves.

\begin{lemma}\label{lemma:slc-locus-is-open}
In the notation of the preceding paragraph, the locus $U\subset H$ parametrising curves $B\subset \PP^{2}$ such that the fibre $\Xscr_{B}$ is semi-log-canonical is open.
\end{lemma}
\begin{proof}
We will use a proof strategy outlined by Kov\'{a}cs \cite{MO262045}.
By construction, every fibre $\Xscr_{B}$ of $\pi\colon\Xscr\to H$ is a double-cover of the plane, hence Gorenstein.
A Gorenstein singularity is semi-log-canonical if and only if it is Du Bois; thus, the locus $U\subset H$ with semi-log-canonical fibres equals the locus with Du Bois-fibres.
Since $H$ is smooth, the Du Bois-locus is open, by Kov\'{a}cs' and Schwede's inversion of adjunction for Du Bois pairs \cite[Theorem A; Lemma 4.5]{KS:2016}.
\end{proof}

Restricting the family $\Xscr$ to $U$ defines a morphism $U\to \XXX$, $B\mapsto [\Xscr_{B}]$, which is surjective by Corollary~\ref{cor:is-double-cover-of-pp2}.
Moreover, it induces an isomorphism of stacks:

\begin{theorem}\label{theorem:stack-isomorphism}
As above, let $U\subset|\osh{\PP^{2}}(8)|$ be the space of half-log-canonical plane curves of degree $8$.
Taking the double-cover branched over the curves induces an isomorphism of algebraic stacks
$$[U/\PGL(3,\CC)]\to \XXXst.$$
In particular, $\XXXst$ is smooth.
\end{theorem}
\begin{proof}
Using the notations introduced above, we will construct the inverse morphism and show that it is an isomorphism.
By Proposition~\ref{prop:curve-to-surface}, any $(f\colon \Xscr\to T)\in (\XXXst)_{T}$ is a relative double-cover of a projective bundle $\PP_{T}(f_{*}\omega_{\Xscr/T})$ with a relative branch divisor $\Bscr\subset \PP_{T}(f_{*}\omega_{\Xscr/T})$.
Thus, the associated $\PGL(3,\CC) = \mathrm{Aut}(\PP^2)$-torsor $P$ comes with an induced $\PGL(3,\CC)$-equivariant morphism $P\to U$ corresponding to the family of octics.
This defines an element of $[U/\PGL(3,\CC)]_{T}$.
It is straightforward to check that this assignment is functorial for pull-backs along morphisms $T'\to T$, thus defining a morphism $\XXXst\to[U/\PGL(3,\CC)]$.
It is fully faithful and its essential image consists of those objects whose underlying $\PGL(3,\CC)$-torsors are locally trivial in the Zariski topology, yet again by Proposition~\ref{prop:curve-to-surface}.
Thus, it suffices to show that this is the case for all $\PGL(3,\CC)$-torsors $P$ with a $\PGL(3,\CC)$-equivariant morphism $P\to U$.
In other words, what we have to show is that if $\PP\to T$ is an \'{e}tale-locally trivial $\PP^2$-bundle with a relative divisor $\Bscr\subset\PP$ which is fibre-wise an octic, then $\PP$ is locally trivial in the Zariski topology.
But since $\osh{\PP}(-3K_{\PP/T}-\Bscr)$ restricts to $\osh{\PP^2}(1)$ on the geometric fibres, this is indeed the case.
\end{proof}

For later reference, we observe that the classifying morphism $U\to \XXX$ maps $\PGL(3,\CC)$-invariant locally closed sets to locally closed sets, hence, also $\PGL(3,\CC)$-invariant stratifications to stratifications.
For this, it is enough to show that it is a geometric quotient in the sense of Mumford \cite[Definition 0.6]{GIT}, for then $\XXX\cong U/G$ carries the quotient topology.

\begin{corollary}\label{cor:classifying-map-submersive}
The coarse moduli space $U/\PGL(3,\CC)\cong \XXX$ is a quasi-projective scheme and the classifying morphism $U\to U/\PGL(3,\CC)$ is a geometric quotient.
In particular, it maps $\PGL(3,\CC)$-invariant locally closed sets to locally closed sets.
\end{corollary}
\begin{proof}
Since $\XXXst$ has a quasi-projective moduli space $\XXX$ by the works of Koll\'{a}r, Shepherd-Barron \cite{KSB:1988} and Alexeev \cite{Alexeev:1994}, so does $[U/\PGL(3,\CC)]$.
This space, $U/\PGL(3,\CC)$, is then the categorical quotient in the category of schemes.
On the other hand, since half-log-canonical plane octics can be shown to be GIT-stable, cf.\ Remark~\ref{rem:GIT-embedding}, the classifying map $U\to U/\PGL(3,\CC)$ equivariantly factors through the GIT-quotient $H^s\to H^s//\PGL(3,\CC)$, which is geometric by Mumford \cite[Theorem~1.10, cf.\ Chapter~1 \S4]{GIT}.
(Here, $H^s\subset|\osh{\PP^{2}}(8)|$ is the locus of GIT-stable points with respect to the $\PGL(3,\CC)$-action.)
But then the classifying morphism $U\to U/G$ has to be geometric as well, as claimed.
\end{proof}

\begin{remark}
The moduli space under consideration is thus birationally equivalent to the moduli space of plane curves of degree $8$, for which, according to B\"ohning, Graf von Bothmer and Kr\"oker \cite[p.\ 506]{BGvBK:2009}, it is not known whether it is rational or not.
\end{remark}

\begin{question}
It is tempting to call $U/\PGL(3,\CC)$ the moduli space of half-log-canonical octics, but it is not obvious whether $U/\PGL(3,\CC)$ is really a moduli space of curves, or one of polarised curves.
\emph{Are there pairs of abstractly, but not projectively isomorphic half-log-canonical plane octics?}
Note that by Hassett \cite[Proposition~2.1]{Hassett:1999}, two abstractly isomorphic nodal plane octics are indeed projectively isomorphic.
\end{question}

For further questions and comparisons between $\XXXst$ and different moduli spaces of curves, see Chapter~\ref{section:further-remarks}.

\section{A Stratification of the moduli space}
\label{section:stratification}

As is well known, normalising usually does not work well in flat families.
For the moduli space at hand, we can get hands on this quite explicitly, in that we can cover it by strata on which the normalisation can be performed in families.
Recall from Proposition~\ref{prop:description-of-normalisation} that if $f\colon X\to\PP^2$ is the canonical double-cover with branch curve $B$, then $B = B'+2B''$ for reduced effective divisors $B',B''$ and the composition $\overline{f} = f\circ \pi\colon\overline{X}\to\PP^2$ is the double-cover branched over $B'$; furthermore, the conductor on $\overline{X}$ is the pull-back of $B''$. 
In other words, the non-reduced part of the branch curve controls the non-normal locus of $X$.
This motivates the first approximation to a stratification.

\begin{definition}[The (non-)normality stratification]
For a non-negative integer $0\leq a\leq 4$ we let $\Mfrak^{(a)}\subset\XXX$ be the locus of those surfaces whose branch divisor $B = B'+2B''\subset \PP^2$ for the canonical double-cover is such that $B'' = \floor{\tfrac{1}{2}B}$ is of degree $a$.
The open and dense subset consisting of the normal surfaces is denoted by $\Nfrak:=\Mfrak^{(0)}$.
\end{definition}

Since the degree of the branch divisors $B$ is eight, we clearly only need to consider $a=0,\dots,4$ to cover $\XXX$ as $\bigcup_{a=0}^4\Mfrak^{(a)}$.

\begin{proposition}\label{prop:normality-stratification}
For each $a = 0,\dots,4$, the subset $\Mfrak^{(a)}\subset\XXX$ is locally closed and its closure in $\XXX$ is $\olMfrak^{(a)} = \bigcup_{a\leq b\leq4}\Mfrak^{(b)}$.
\end{proposition}
%
\begin{proof}
As before, we let $U\subset |\osh{\PP^{2}}(8)|$ be the open sub-scheme parametrising half-log-canonical plane octic curves and let $f\colon U\to \XXX$ be the classifying map.
By Corollary~\ref{cor:classifying-map-submersive}, it suffices to show that the pre-images $f^{-1}(\Mfrak^{(a)})$ define such a stratification of $U$.
The locus $U_{a}\subset U$ of half-log-canonical curves $B = B'+2B''$ with $\deg(B'') \geq a$ may alternatively be characterised as the locus of $B\in U$ such that $\deg(B_{\textnormal{red}})\leq 8-a$.

Note that quite generally, the space $V_{n}\subset |\osh{\PP^{2}}(8)|$ consisting of the divisors $B$ such that $\deg(B_{\textnormal{red}})\leq n$ is closed, being a union of closed sub-spaces
$$V_{n} = \bigcup_{\substack{\sum_{i}a_{i}n_{i}=8,\\\sum_{i}n_{i}\leq n}}\sum_{i}a_{i}|\osh{\PP^{2}}(n_{i})|,$$
where $\sum_{i}a_{i}|\osh{\PP^{2}}(n_{i})|$ is shorthand notation for the image of the morphism $$\prod_{i}|\osh{\PP^{2}}(n_{i})|\to|\osh{\PP^{2}}(8)|,\;(B_{i})_{i}\mapsto\sum_{i}a_{i}B_{i}.$$

Therefore, $U_{a} = U\cap V_{8-a}$ is closed, $f^{-1}(\Mfrak^{(a)}) = U_{a}\setminus U_{a+1}$ is locally closed and from the above presentation it easily follows that $\overline{f^{-1}(\Mfrak^{(a)})} = U_{a}$.
This completes the proof.
\end{proof}

This very rough stratification will be refined in the following sections.

\subsection{The locus of normal surfaces}
We now turn to the stratification of the moduli space $\Nfrak = \Mfrak^{(0)}$ of normal Gorenstein stable surfaces with $K_{X}^2 = 2$ and $\chi(\osh{X}) = 4$.
We stratify $\Nfrak$ according to the number of irrational singularities, their degree and whether they are simply elliptic or cusps.

\begin{definition}[The irrationality stratification]
Given non-negative integers $a,b,c$ and $d$ satisfying $a+b+c+d\leq 4$ we define the subset $\YYY{a}{b}{c}{d}$ of the stratum of normal surfaces $\Nfrak$ consisting of those surfaces having precisely $a$ simply elliptic singularities of degree $1$, $b$ cusps of degree $1$, $c$ simply elliptic singularities of degree $2$ and $d$ cusps of degree $2$.
To ease notation, indices with exponent $0$ are omitted, an exponent of $1$ will be omitted and so on, e.g., $2^1 = 2$, $1^2 = 11$, etc.
\end{definition}
For example, an $X\in \Nfrak_{122\overline{2}} = \YYY{1}{0}{2}{1}$ has exactly one simply elliptic singularity of degree one, two simply elliptic singularities of degree two and one cusps of degree two.
The empty list of degrees corresponds to the stable surfaces with only canonical singularities, i.e., $\Nfrak_{\emptyset} = \YYY{0}{0}{0}{0} = \Mfrak_{2,4}$ is the dense open of canonical surfaces.

\begin{remark}
By definition, the loci $\YYY{a}{i}{b}{j}\subset \Nfrak$ are pair-wise disjoint and since the normal surfaces under investigation have at most four irrational singularities (Theorem~\ref{theorem:FPR} and Proposition~\ref{prop:normalisation-kodaira-neg-infinite}), $\Nfrak$ is indeed covered by the loci $\YYY{a}{i}{b}{j}$, as $a+b+i+j\leq 4$.
\end{remark}

\begin{remark}
Local singularity theory, most notably Brieskorn's result \cite{Brieskorn:1979}, implies that that a singularity of type $X_{p}$ may degenerate to a singularity of type $X_{q}$ with $q\geq p$ or certain singularities of type $Y_{\bullet,\bullet}$, but none of them can degenerate to a triple-point or to a milder quadruple-point.
Similarly, a $[3;3]$-point may degenerate more and more, or it may even degenerate to a quadruple-point, but none of the series $X_{\bullet}$ or $Y_{\bullet,\bullet}$.
This prevents certain strata to appear at the boundary of other strata.
For example, the boundary of $\Nfrak_{2}$ is covered by all strata parametrising surfaces with at least one (possibly degenerate) quadruple-point.
More generally, the closure of $\YYY{a}{b}{c}{d}$ in $\Nfrak$ is contained in the union of the strata $\YYY{a'}{b'}{c'}{d'}$ where $a'+b'\geq a+b$ and $b'\geq b$, as well as $c'+d'\geq c+d$ and $d'\geq d$.
\end{remark}


\begin{proposition}\label{prop:irrationality-stratification}
The strata $\YYY{a}{b}{c}{d}\subset\Nfrak$ are locally closed and the closure of a stratum is contained in a union of strata.
\end{proposition}
\begin{proof}
Let $U\subset |\osh{\PP^{2}}|$ be the locus parametrising half-log-canonical plane octics and let $V\subset U$ be the open sub-space parametrising reduced curves.
Then the pre-images of the strata under the classifying morphism $V\to\Nfrak\subset\XXX$ are disjoint and $\PGL(3,\CC)$-invariant by construction.
By Corollary~\ref{cor:classifying-map-submersive}, it suffices to show that these pre-images are locally closed.
This can be shown for each case separately by elementary plane curve geometry.
We omit the details and conclude the proof.
\end{proof}

\begin{remark}
The motivation to consider not just the stratification according the number and degree of irrational singularities, which is usually enough to get control over the birational geometry of the minimal resolution, but to also distinguish between simply elliptic and cuspidal singularities, comes from the relevance to the mixed Hodge structure discussed in Section~\ref{subsection:hodge}.
\end{remark}

\subsubsection{The strata of the irrationality stratification}
We are ready to state and prove the main results about the irrationality stratification.
\begin{figure}\centering
\begin{tikzpicture}[on grid]
  \node (N) {$\Nfrak_{\emptyset}$};
  \node (N2) [below left= 3.5em and 5em of N] {$\Nfrak_{2}$};
  \node (N1) [below right= 3.5em and 5em of N] {$\Nfrak_{1}$};
  \node (N22) [below left= 4em and 5em of N2] {$\Nfrak_{22}$};
  \node (N12') [below right= 4em and 2em of N2] {$\Nfrak_{12}'$};
  \node (N12'') [below left= 4em and 2em of N1] {$\Nfrak_{12}''$};
  \node (N11) [below right= 4em and 5em of N1] {$\Nfrak_{11}$};
  \node (N222) [below left= 4em and 4em of N22] {$\Nfrak_{2^3}$};
  \node (N122') [below left= 4em and 5em of N12'] {$\Nfrak_{122}'$};
  \node (N122'') [below left= 4em and 2em of N12'] {$\Nfrak_{122}''$};
  \node (N112') [below left= 4em and 4em of N12''] {$\Nfrak_{112}'$};
  \node (N112'') [below = 4em of N12''] {$\Nfrak_{112}''$};
  \node (N112''') [below right=4em and 4em of N12''] {$\Nfrak_{112}'''$};
  \node (N111') [below right= 4em and 2em of N11] {$\Nfrak_{1^3}'$};
  \node (N111'') [below right= 4em and 5em of N11] {$\Nfrak_{1^3}''$};
  \node (N2222) [below=4em of N222] {$\Nfrak_{2^4}$};
  \node (N1112') [below left =4em and 4em of N111'] {$\Nfrak_{1^32}'$};
  \node (N1112'') [below left=4em and 2em of N111''] {$\Nfrak_{1^32}''$};
  \path (N) edge (N2)
  		    edge (N1);
  \path (N2) edge (N22)
  		     edge (N12')
  		     edge (N12'');
  \path (N1) edge (N11)
  		     edge (N12')
  		     edge (N12'');
  \path (N22) edge (N222)
  		      edge (N122')
  		      edge (N122'');
  \path (N12') edge (N122')
  			   edge (N122'')
  			   edge (N112')
  			   edge (N112'');
  \path (N12'') edge (N122'')
  			    edge (N112'')
  			    edge (N112''');
  \path (N11) edge (N111')
              edge (N111'')
              edge (N112')
              edge (N112'')
              edge (N112''');
  \path (N222) edge (N2222);
  \path (N111') edge (N1112');
  \path (N111'') edge (N1112'');
  \path (N112'') edge (N1112'');
  \path (N112''') edge (N1112')
  				  edge (N1112'');
\end{tikzpicture}
\caption{The degeneration diagram showing the components of the strata parametrising surfaces with simply elliptic singularities}\label{fig:degenerations-simple-elliptic}
\end{figure}
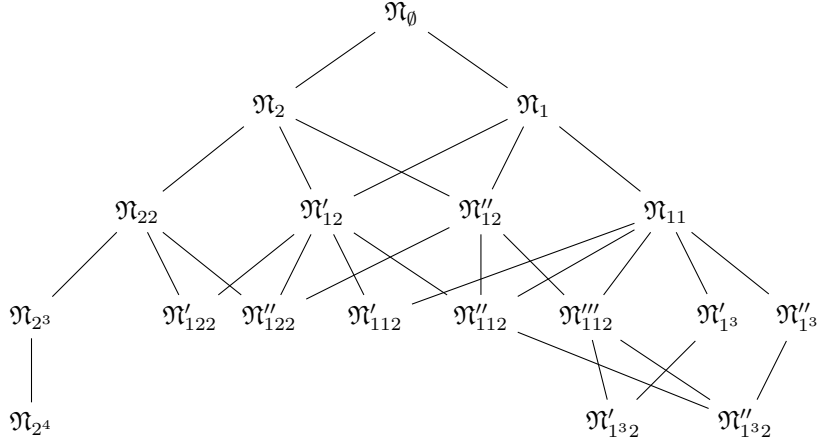

\begin{theorem}\label{thm:strata-dimensions}
All strata $\YYY{a}{b}{c}{d}$ with $a+b+c+d\leq 3$ and the two strata $\Nfrak_{1^32}$ and $\Nfrak_{2^4}$ are equidimensional of expected dimension $36-9 a-10b-8c-9d$.
The remaining strata are empty.
Furthermore, the irreducible components of the strata are pair-wise disjoint.
\end{theorem}
\begin{proof}
It follows from Proposition~\ref{prop:normalisation-kodaira-neg-infinite} and Theorem~\ref{theorem:FPR} that $a+b+c+d\leq 4$ and if $a+b+c+d = 4$, then $b = d = 0$.
Moreover, Proposition~\ref{prop:N1122-and-N1222} shows that $\Nfrak_{1122}$ and $\Nfrak_{12^3}$ are empty, as is $\Nfrak_{1^4}$, by Proposition~\ref{prop:max-no-of-33-pts}.
It remains to show that all other strata are equidimensional, with all components pair-wise disjoint and of expected dimension.
Since we have translated the problem into plane curve geometry, we can systematically use the computer algebra system Macaulay2 \cite{M2} to check which strata are inhabited (by producing elements explicitly), find all irreducible components and compute their dimension.

The scripts can be obtained from \cite{Anthes:2018} and the explanations about how they work are the content of \ref{section:M2}.

Which stratum is dealt with where is listed in \ref{table:where-is-which-stratum}.
\end{proof}

\begin{table}\centering
\begin{tabular}{llcll}
\toprule
\multicolumn{2}{c}{Script reference} && \multicolumn{2}{c}{Script reference} \\
Strata/Components & Section  && Strata/Components & Section \\
\cmidrule{1-2}\cmidrule{4-5}
\multicolumn{2}{c}{\cite[parameterFreeCases.m2]{Anthes:2018}} && \multicolumn{2}{c}{\cite[degenerate122.m2]{Anthes:2018}}\\
$\Nfrak_{2},\Nfrak_{22},\Nfrak_{2^3},\Nfrak_{2^4}$ & I && $\Nfrak_{12\oltwo}',\Nfrak_{\olone22}'$ & I.1\\
$\Nfrak_{1},\Nfrak_{11}$ & II && $\Nfrak_{1\oltwo\oltwo}',\Nfrak_{\olone2\oltwo}'$ & I.2\\
$\Nfrak_{12},\Nfrak_{122},\Nfrak_{112}$\hyperlink{manual-reference}{(*)} & III && $\Nfrak_{\olone\oltwo\oltwo}'$ & I.3\\
\multicolumn{2}{c}{\cite[upToTwoSingularities.m2]{Anthes:2018}} && $\Nfrak''_{12\oltwo},\Nfrak_{\olone22}''$ & II.1\\
$\Nfrak_{\olone},\Nfrak_{1\olone},\Nfrak_{\olone\olone}$ & I && $\Nfrak_{1\oltwo\oltwo}'',\Nfrak''_{\olone2\oltwo}$ & II.2\\
$\Nfrak_{\oltwo},\Nfrak_{2\oltwo},\Nfrak_{\oltwo\oltwo}$ & II && $\Nfrak_{\olone\oltwo\oltwo}''$ & II.3\\
$\Nfrak_{12},\Nfrak_{1\oltwo},\Nfrak_{\olone2},\Nfrak_{\olone\oltwo}$ & III &&\multicolumn{2}{c}{\cite[degenerate112.m2]{Anthes:2018}}\\
\multicolumn{2}{c}{\cite[threeNonDeg.m2]{Anthes:2018}} &&  $\Nfrak'_{11\oltwo},\Nfrak'_{1\olone2}$ & I.1\\
$\Nfrak_{122}$ & I && $\Nfrak'_{1\olone\oltwo},\Nfrak'_{\olone\olone2}$ & I.2\\
$\Nfrak_{112}$ & II && $\Nfrak'_{\olone\olone\oltwo}$ & I.3\\
$\Nfrak_{1^3}$ & III && $\Nfrak''_{11\oltwo},\Nfrak''_{1\olone2}$ & II.1\\
\multicolumn{2}{c}{\cite[1112.m2]{Anthes:2018}} && $\Nfrak''_{1\olone\oltwo},\Nfrak''_{\olone\olone2}$ & II.2\\
$\Nfrak_{1^32}$ & I, II && $\Nfrak''_{\olone\olone\oltwo}$ & II.3\\
\multicolumn{2}{c}{\cite[degenerate222.m2]{Anthes:2018}} && $\Nfrak'''_{11\oltwo},\Nfrak'''_{1\olone2}$ & III.1\\
$\Nfrak_{22\oltwo},\Nfrak_{2\oltwo\oltwo},\Nfrak_{\oltwo^3}$ & --- && $\Nfrak'''_{1\olone\oltwo},\Nfrak'''_{\olone\olone2}$ & III.2\\
\multicolumn{2}{c}{\cite[degenerate111.m2]{Anthes:2018}} && $\Nfrak'''_{\olone\olone\oltwo}$ & III.3\\
$\Nfrak_{11\olone}',\Nfrak_{1\olone\olone}',\Nfrak_{\olone^3}'$ & I.1--I.3 && \multicolumn{2}{c}{\cite[1111.m2]{Anthes:2018}}\\
$\Nfrak_{11\olone}'',\Nfrak_{1\olone\olone}'',\Nfrak_{\olone^3}''$ & II.1--II.3 && $\Nfrak_{1^4} = \emptyset$ & ---\\
\midrule
\multicolumn{5}{p{28em}}{{\footnotesize (\hypertarget{manual-reference}{*}) In \cite[parameterFreeCases.m2~III.1~\&~III.2]{Anthes:2018}, only the dimensions are computed; the rest about $\Nfrak_{12}$ is in \cite[upToTwoSingularities.m2~III.1]{Anthes:2018} and for $\Nfrak_{122}$ and $\Nfrak_{112}$ see \cite[threeNonDeg.m2~I~\&~II]{Anthes:2018}.}}\\
\bottomrule
\end{tabular}
\caption{The catalogue of scripts and strata}\label{table:where-is-which-stratum} 
\end{table}

\begin{remark}
If a stratum $\YYY{a}{b}{c}{d}$ decomposes into the union of multiple components, we mostly use the following ad-hoc notation: we decorate the components with primes, i.e., $$\YYY{a}{b}{c}{d} = \YYY{a}{b}{c}{d}'\cup \YYY{a}{b}{c}{d}''\left(\cup\,  \YYY{a}{b}{c}{d}'''\right).$$
There are four exceptions, $\Nfrak_{1\olone2}$, $\Nfrak_{1\olone\oltwo}$, $\Nfrak_{12\oltwo}$ and $\Nfrak_{\olone2\oltwo}$, where a refined notation is explained (and used only) in the Macaulay2-code.
The choices for the order are somewhat arbitrary.
As a rule of thumb, more primes indicate that the configuration of singularities of the branch curve is more special.
We give a few examples; see Definition~\ref{def:n-n-sings} for the notions:

If $X\in \Nfrak_{12} = \Nfrak_{12}'\cup\Nfrak_{12}''$, then the branch curve of the canonical double cover $X\to\PP^{2}$ has exactly one (non-degenerate) $[3;3]$-point and one (ordinary) quadruple-point and up to automorphisms, there are two possibilities.
Namely, either the distinguished tangent line of the $[3;3]$-point misses the quadruple-point ($X\in \Nfrak_{12}'$), or it passes through it ($X\in \Nfrak_{12}''$).
That this indeed splits $\Nfrak_{12}$ into two components is non-trivial but follows from the Macaulay2-code \cite[upToTwoSingularities.m2~III.1]{Anthes:2018}.
Likewise, $\Nfrak_{122}$ has two components, but no more since there can be only one quadruple-point on the distinguished tangent line of the $[3;3]$-point (Lemma~\ref{lemma:quadruple-on-dist-tangent}).
However, $\Nfrak_{12\oltwo}$ and $\Nfrak_{\olone2\oltwo}$ have three components since either the degenerate, or the non-degenerate quadruple-point lies on the distinguished tangent line of the $[3;3]$-point.


In a different flavour, $\Nfrak_{111} = \Nfrak_{111}'\cup\Nfrak_{111}''$ where the $[3;3]$-points of the branch curve of $X\in \Nfrak_{111}''$ are with tangents along a conic.

The decomposition $\Nfrak_{1^32}=\Nfrak_{1^32}'\cup\Nfrak_{1^32}''$ comes from the two cases described in Proposition~\ref{prop:N1112}.
\end{remark}

\begin{remark}
In total, $\Nfrak$ is covered by two strata with four components, six strata with three components, $13$ strata with two components and $16$ irreducible strata.
Hence, the number of inhabited strata of the irrationality stratification on $\Nfrak$ is $37$ and there are $68$ pair-wise disjoint components.
The degeneration diagram for the components of all strata parametrising surfaces with simply elliptic singularities is shown in \ref{fig:degenerations-simple-elliptic}.
The complete degeneration diagram showing all strata of $\Nfrak$ would be incomprehensibly complicated.
\end{remark}

\begin{theorem}\label{thm:birational-classification-by-minimal-resolution}
\ref{table:birational-classification-by-minimal-resolution} lists the components of the strata of the irrationality stratification and the birational isomorphism type of their members.
\end{theorem}
\begin{proof}
Let $X$ be a normal Gorenstein stable surface with $K_{X}^2 = 2$ and $\chi(\osh{X}) = 4$, let $f\colon Y\to X$ be the minimal resolution and let $\sigma\colon Y\to \minmod{Y}$ be a minimal model.
By Lemma~\ref{lemma:chi-of-resolution}, $\chi(\osh{\minmod{Y}}) = \chi(\osh{Y}) = 4-k$, where $k$ is the number of elliptic singularities of $X$.
Furthermore, Theorem~\ref{theorem:FPR} and Proposition~\ref{prop:normalisation-kodaira-neg-infinite} constrain the possible Kodaira dimensions $\kappa(\minmod{Y}) = \kappa(Y)$ in such a way that from the Enriques--Kodaira classification of algebraic surfaces (cf.\ Barth, Hulek, Peters, Van de Ven \cite[VI Theorem~1.1]{BHPV}), the claim follows for $\Nfrak_{\emptyset}$, $\Nfrak_{1}$, $\Nfrak_{2}$, $\Nfrak_{11}$, $\Nfrak_{22}$, $\Nfrak_{122}$, $\Nfrak_{222}$, $\Nfrak_{1^32}$ and $\Nfrak_{2^4}$ and the cuspidal versions $\Nfrak_{\olone}$, $\Nfrak_{\oltwo}$ etc.
In other words, the only strata which need extra care are $\Nfrak_{12}$, $\Nfrak_{1^3}$ and $\Nfrak_{112}$ and their versions with cusps.

Before we deal with these cases, we introduce some more notation.
Since the rational singularities of $X$ admit a crepant resolution, $K_{Y} = f^{*}K_{X}-E$, where $E\subset Y$ is the sum of the exceptional curves $E_{i}\subset Y$ over the elliptic singularities $p_{i}\in X$, $i = 1,\dots,k$.
Likewise, let $G\subset Y$ be the divisor such that $K_{Y} = \sigma^{*}K_{\minmod{Y}}+G$.

If $X$ has two elliptic singularities, one of degree one and one of degree two, then $\sigma$ is a single blow-up in a smooth point and either $\kappa(Y) = 0$ or $\kappa(Y) = 1$, again by Theorem~\ref{theorem:FPR}.
Since $\chi(\osh{\minmod{Y}}) = 2$ and by the classification of algebraic surfaces, either $\minmod{Y}$ is a K3 surface or properly elliptic.
To distinguish the two cases geometrically, we compute $H^{0}(Y,2K_{Y})$:

Let $\varphi\colon X\to \PP^{2}$ be the canonical double-cover.
Its branch curve $B\subset\PP^2$ has a $[3;3]$-point $p$ and quadruple-point $q$ and no more non-simple singularities.
Let $\tau\colon Z\to\PP^2$ be the sequence of blow-ups at $q$, at $p$ and at the point over $p$ corresponding to the distinguished tangent direction of the $[3;3]$-point, yielding exceptional curves $F_{1}$, $F_{2}$ and $F_{3}$, respectively.
We have
\begin{align*}
K_{Z}&=\tau^{*}K_{\PP^2}+F_{1}+F_{2}+2F_{3} \text{ and}\\
\tau^{*}B &= \tau_{*}^{-1}B+4F_{1}+3F_{2}+6F_{3}.
\end{align*}
Therefore, the double-cover $\psi\colon Y'\to Z$ branched over $\tau_{*}^{-1}B+F_{2}$ induces a partial resolution $Y'\to X$.
The surface $Y'$ has at most simple singularities and up to a (crepant) resolution of those, which will not affect the pluri-canonical sections, $Y$ is obtained from $Y'$ by contracting the $(-1)$-curve sitting above $F_{2}$.
Hence,
\begin{align*}
H^0(Y,2K_{Y})&\cong H^0(Y',2K_{Y'})\\
&\cong H^0(Z,2K_{Z}+\tau_{*}^{-1}B+F_{2})\\
&= H^0(Z,\tau^{*}(2K_{\PP^{2}}+B)-2F_{1}-F_{2}-F_{3}),
\end{align*}
meaning that the sections of $2K_{Y}$ correspond to the plane conics with a double-point at $q$ and passing through $p$ in the distinguished tangent direction of the $[3;3]$-point.
If the line $L\subset \PP^{2}$ joining $p$ and $q$ happens to be the distinguished tangent line of the $[3;3]$-point, then there is exactly a pencil of such conics, namely, the conics of the form $L+L'$ where $L'$ is any line through $q$; thus, $H^0(Y,2K_{Y}) = 2$ and $\minmod{Y}$ is properly elliptic.
Otherwise, i.e., if $L$ is not the distinguished tangent line, then $2L$ is the only conic with a double-point at $q$ and containing the distinguished tangent of the $[3;3]$-point at $p$; that is, $H^0(Y,2K_{Y}) = 1$.
To see that in this case $\minmod{Y}$ is a K3-surface, observe that $2K_{Y} = (\varphi\circ f)^{*}2L$ is twice the $(-1)$-curve $G$, the contraction of which yields $\minmod{Y}$.
In conclusion, this gives the claimed result for the stratum $\Nfrak_{12}$ and its cuspidal versions.

For $X\in \Nfrak_{1^3}$, $\chi(\osh{Y}) = 1$ and we have two possibilities, namely, either $\minmod{Y}$ is Enriques (Theorem~\ref{theorem:FPR}) or rational (Proposition~\ref{prop:normalisation-kodaira-neg-infinite}).
In any case, $\q(Y) = 0$, as we will show below.
Therefore, Castelnuovo's Rationality Criterion implies that $Y$ is rational if and only if
$P_{2}(Y) = 0$.
The branch curve $B\subset\PP^{2}$ of the canonical double cover has three $[3;3]$-points and from Corollary~\ref{cor:constraints-33-pts}, it follows that the they are not collinear and that none of them lies on a distinguished tangent line of another $[3;3]$-point of $B$.
Thus, they either align along a smooth conic, which has to be contained in the octic then, or they do not.
But the sections of $\omega_{Y}^2$ correspond exactly to the conics passing through all three points in distinguished tangent directions, so that $P_{2}(Y)\not=0$ if and only if $X\in\Nfrak_{1^3}''$.
In other words, if $X\in\Nfrak_{1^3}''$, then $\minmod{Y}$ is an Enriques surface and if $X\in\Nfrak_{1^3}'$, then $Y$ is rational.
The same argument applies to the cuspidal versions $\Nfrak_{11\olone}$, $\Nfrak_{1\olone\olone}$ and $\Nfrak_{\olone^3}$.

The last case we have to consider is that $X$ has two elliptic singularities of degree one and one of degree two.
Since $X$ has exactly three elliptic singularities, $\chi(\osh{Y}) = 1$ and from Theorem~\ref{theorem:FPR} and Proposition~\ref{prop:normalisation-kodaira-neg-infinite} we conclude that either $\minmod{Y}$ is an Enriques surface and $\sigma\colon Y\to \minmod{Y}$ is a blow up in two points (possibly infinitely close), or $Y$ is rational.
If $X\in \Nfrak_{112}'''$ or any of the cuspidal versions, i.e., if the branch curve of the canonical double-cover $X\to\PP^{2}$ contains the two distinguished tangents, which meet in the quadruple-point, then the union of those lines defines a trivialisation of $2K_{\minmod{Y}}$; more precisely, the corresponding section of $\osh{\PP^{2}}(2)$ lifts to a section of $2K_{X}-2E$ with vanishing locus twice the disjoint union of two disjoint $(-1)$-curves, which constitute the exceptional locus of $\sigma\colon Y\to \minmod{Y}$.
If $X$ is a member of either $\Nfrak_{112}'$ or $\Nfrak_{112}''$, however, this does not work and $X$ is rational.
An alternative way to see this is as follows.
Let $B\subset\PP^{2}$ be the branch curve of the canonical double-cover.
Let $p_{1},p_{2}\in B$ be the $[3;3]$-points and let $p_{3}\in\PP^{2}$ be the quadruple-point.
Consider the Cremona-transformation $\varphi\colon\PP^{2}\dashrightarrow\PP^{2}$ with centres $p_{1},p_{2},p_{3}$.
Let $B'\subset \PP^{2}$ be the reduced curve supported on the pull-back $(\varphi^{-1})^{*}B$, but neglecting the components with even multiplicity.
Then the double-covers $X$ and $X'$ branched over $B$ and $B'$, respectively, are birational.
In fact, the double cover branched over $B'$ is the normalisation of the double-cover branched over $(\varphi^{-1})^{*}B$, which is birational to $X$.
If $X\in \Nfrak_{112}'$, then $X'\in\Nfrak_{222}''$, which is rational and if $X\in\Nfrak_{112}''$, then $X'\in\Nfrak_{122}''$, which is rational as well.
Furthermore, as the $[3;3]$- or quadruple-points of $X$ degenerate, those of $X'$ degenerate as well, but again this does not affect the birational isomorphism type.

Finally, we compute $\p_{g}(Y)$ and $\q(Y)$.
If $X$ has a single elliptic singularity, then the canonical linear system $|K_{Y}|$ is one-dimensional, corresponding to the pencil of lines through the non-simple singularity of the branch curve.
Hence, $\p_{g}(Y) = 2$ and $\q(Y) = 0$.
If $X$ has two elliptic singularities, then $|K_{Y}|$ is a single point, corresponding to the line joining the two non-simple singularities of the branch curve.
If $X$ has at least three elliptic singularities, then there is no line through all the corresponding singularities of the branch curve, by Lemma~\ref{lemma:degree-bounds}.
Thus, $\p_{g}(Y) = 0$.
Since $4-\chi(\osh{Y})$ is the number of elliptic singularities, this implies $\q(Y) = 0$ if $X$ has at most three, and $\q(Y)=1$ if $X$ has four elliptic singularities.
\end{proof}

\begin{remark}\label{rmk:error}
Theorem~\ref{thm:birational-classification-by-minimal-resolution} above fixes a mistake in the author's thesis \cite{Anthes:thesis} affecting the analogue to \ref{table:birational-classification-by-minimal-resolution} and the corresponding proof. Therefore, it should be stressed that the minimal resolutions of the members of $\Nfrak'_{12}$ are indeed surfaces of K3-type and those $\Nfrak_{12}''$ are indeed properly elliptic, contrary to the heuristic that the Kodaira dimension should decrease as the branch curve becomes more special.
\end{remark}

\begin{table}\centering
\begin{tabular}{ll}
\toprule
Components & Minimal model of the resolution $\minmod{Y}$ \\ 
\midrule
$\Nfrak_{\emptyset}$ & General type, $K_{\minmod Y}^2 = 2$, $\chi(\osh{\minmod{Y}}) = 4$\\
$\Nfrak_{1},\Nfrak_{\olone}$ & General type, $K_{\minmod{Y}}^2 = 1$, $\chi(\osh{\minmod{Y}}) = 3$\\
$\Nfrak_{2},\Nfrak_{\oltwo}$ & Properly elliptic, $\chi(\osh{\minmod{Y}}) = 3$, $\p_{g}(\minmod{Y})=2$\\
$\Nfrak_{11},\Nfrak_{1\olone},\Nfrak_{\olone\olone}$ & Properly elliptic, $\chi(\osh{\minmod{Y}}) = 2$, $\p_{g}(\minmod{Y})=1$ \\
$\Nfrak_{12}',\Nfrak_{1\oltwo}',\Nfrak_{\olone2}',\Nfrak_{\olone\oltwo}'$ & K3\\
$\Nfrak_{12}'',\Nfrak_{1\oltwo}'',\Nfrak_{\olone2}'',\Nfrak_{\olone\oltwo}''$ & Properly elliptic, $\chi(\osh{\minmod{Y}}) = 2$, $\p_{g}(\minmod{Y})=1$\\
$\Nfrak_{22},\Nfrak_{2\oltwo},\Nfrak_{\oltwo\oltwo}$ & K3\\
$\Nfrak_{1^3}',\Nfrak_{11\olone}',\Nfrak_{1\olone\olone}',\Nfrak_{\olone^3}'$ & Rational\\
$\Nfrak_{1^3}'',\Nfrak_{11\olone}'',\Nfrak_{1\olone\olone}'',\Nfrak_{\olone^3}''$ & Enriques\\
$\Nfrak_{112}',\Nfrak_{11\oltwo}',\dots,\Nfrak_{\olone\olone\oltwo}'$ & Rational\\
$\Nfrak_{112}'',\Nfrak_{11\oltwo}'',\dots,\Nfrak_{\olone\olone\oltwo}''$ & Rational\\
$\Nfrak_{112}''',\Nfrak_{11\oltwo}''',\dots,\Nfrak_{\olone\olone\oltwo}'''$ & Enriques\\
$\Nfrak_{122}',\Nfrak_{12\oltwo}',\dots,\Nfrak_{\olone\oltwo\oltwo}'$ & Rational\\
$\Nfrak_{122}'',\Nfrak_{12\oltwo}'',\dots,\Nfrak_{\olone\oltwo\oltwo}''$ & Rational\\
$\Nfrak_{2^3}$ & Rational\\
$\Nfrak_{1^32}',\Nfrak_{1^32}''$ & Ruled of genus 1\\
$\Nfrak_{2^4}$ & Ruled of genus $1$\\
\bottomrule
\end{tabular}
\caption{The birational types of the normalisations.}\label{table:birational-classification-by-minimal-resolution}
\end{table}

The Hodge type $\lozenge_{r,s}$ of $X\in\Nfrak$ roughly behaves as follows:
As we introduce a simply elliptic singularity, $s$ increases by one and as a simply elliptic singularity degenerates to a cusp, $s$ decreases by one and $r$ increases by one.
In fact, this only fails in the case where it is numerically impossible since there are four elliptic singularities.
In this case, the $(1,0)$-classes become linearly dependent. 

\begin{proposition}\label{prop:hodge-diamond-classification}
The Hodge type is constant on every stratum of the irrationality stratification of $\Nfrak$ and they are given as in \ref{fig:Hodge type-diagram}.
\end{proposition}
\begin{proof}
Recall from Lemma~\ref{lemma:MHS-computation} that if $X$ has exactly $r$ cusps, then it is of Hodge type $\lozenge_{r,s}$ for some $0\leq s\leq 3-r$.
It remains to compute $s$ in each possible case.
Recall the set-up in which we proved the lemma.
We let $Y\to X$ be the resolution of the elliptic singularities, with exceptional arithmetically elliptic curves $E_{i}$, $i = 1,\dots,n$, and considered the Mayer--Vietoris exact sequence
$$H^1(Y)\to \bigoplus_{i=1}^nH^1(E_{i})\to H^2(X)\to H^2(Y).$$
In this set-up, renumbering if necessary, we can suppose that the curves $E_{1},\dots,E_{k}$ are smooth elliptic and that the remaining ones, $E_{k+1},\dots,E_{n}$, are cycles of rational curves.
Then the induced exact sequence of $(1,0)$-parts becomes:
$$H^{1,0}(Y)\to \bigoplus_{i=1}^kH^1(E_{i})^{1,0}\to (H^2(X))^{1,0}\to 0.$$
Thus, $s = \dim (H^2(X))^{1,0} = k-\dim\im(H^{1,0}(Y)\to \bigoplus_{i=1}^kH^1(E_{i}))$.
In particular, if $k = 0$, then $s = 0$.
In what follows, we assume $k\geq 1$.

Our claims only concern the dimensions in degree $(0,0)$, $(1,0)$ and $(2,0)$;
for this reason and since the remaining singularities of $Y$ are rational, we can assume without loss of generality that $Y$ is minimal.

It follows from the proof of Theorem~\ref{thm:birational-classification-by-minimal-resolution} that $\q(Y) = 0$, unless $X$ is ruled of genus $1$.
Clearly, if $\q(Y) = 0$, then $s = k$.

If $Y$ is ruled of genus $1$, then $\q(Y) = 1$ and the curves $E_{i}$, $i\leq k$, are multi-sections of the ruling $Y\to C$.
On the one hand, the pull-back morphism $H^1(C)\to H^1(E_{i})$ is multiplication with the degree, hence injective.
On the other hand, it factors through $H^1(Y)\to H^1(E_{i})$, which is injective as well then.
Therefore, $s = k-1$.
Note that the strata where $Y$ is ruled are those with $n = k = 4$ and $r = 0$.
In conclusion, the members of $\Nfrak_{1^32}$ or $\Nfrak_{2^4}$ have Hodge type $\lozenge_{0,3}$.
\end{proof}

\begin{figure}\centering
\begin{tikzpicture}[every node/.style={node distance = 2.5cm}]
 \node[state] (00) 
 {\begin{tabular}{c}
  {$\lozenge_{0,0}$}\\\hline
  {$\Mfrak_{2,4}$}
 \end{tabular}};
 
 \node[state, right of=00] (01)
 {\begin{tabular}{c}
  {$\lozenge_{0,1}$}\\\hline
  {$\Nfrak_{1},\Nfrak_{2}$}
 \end{tabular}
 };

 \node[state,right of=01] (02)
 {\begin{tabular}{c}
  {$\lozenge_{0,2}$}\\\hline
    {$\begin{matrix}
  	\Nfrak_{11},\Nfrak_{12},\\
  	\Nfrak_{22}\\
  	\end{matrix}$}
 \end{tabular}
 };
  
  \node[state,right of=02] (03)
  {\begin{tabular}{c}
   {$\lozenge_{0,3}$}\\\hline
   {$\begin{matrix}
   \Nfrak_{111},\Nfrak_{112},\\
   \Nfrak_{1^32},\Nfrak_{122},\\\Nfrak_{2^3},\Nfrak_{2^4}\\
   \end{matrix}$}
  \end{tabular}
  };
 
 \node[state,below of=01] (10)
 {\begin{tabular}{c}
  {$\lozenge_{1,0}$}\\\hline
  {$\Nfrak_{\olone},\Nfrak_{\oltwo}$}
 \end{tabular}
 };

 \node[state,below of=02] (11)
 {\begin{tabular}{c}
  {$\lozenge_{1,1}$}\\\hline
  {$\begin{matrix}
    \Nfrak_{1\olone},\Nfrak_{\olone2},\\
    \Nfrak_{1\oltwo},\Nfrak_{2\oltwo}\\
    \end{matrix}$}
 \end{tabular}
 };

 \node[state, below of=03] (12)
 {\begin{tabular}{c}
  {$\lozenge_{1,2}$}\\\hline
  {$\begin{matrix}
    \Nfrak_{11\olone},\Nfrak_{1\olone2},\\
    \Nfrak_{11\oltwo},\Nfrak_{\olone22},\\
    \Nfrak_{12\oltwo},\Nfrak_{22\oltwo}\\
    \end{matrix}$}
 \end{tabular}
 };

 \node[state, below of=11] (20)
 {\begin{tabular}{c}
  {$\lozenge_{2,0}$}\\\hline
  {$\begin{matrix}
    \Nfrak_{\olone\olone},\Nfrak_{\olone\oltwo},\\
    \Nfrak_{\oltwo\oltwo}\\
    \end{matrix}$}
 \end{tabular}
 };

 \node[state,below of=12] (21)
 {\begin{tabular}{c}
  {$\lozenge_{2,1}$}\\\hline
  {$\begin{matrix}
    \Nfrak_{1\olone\olone},\Nfrak_{1\olone\oltwo},\\
    \Nfrak_{1\oltwo\oltwo},\Nfrak_{\olone\olone2},\\
    \Nfrak_{2\oltwo\oltwo}\\
    \end{matrix}$}
 \end{tabular}
 };

 \node[state,below of=21] (30)
 {\begin{tabular}{c}
  {$\lozenge_{3,0}$}\\\hline
  {$\begin{matrix}
    \Nfrak_{\olone^3},\Nfrak_{\olone\olone\oltwo}\\
    \Nfrak_{\olone\oltwo\oltwo},\Nfrak_{\oltwo^3}\\
    \end{matrix}$}
 \end{tabular}
 }; 

 \path (00) edge (01);
 \path (01) edge (10)
 			edge (02);
 \path (10) edge (11);
 \path (02) edge (11)
  			edge (03);
 \path (03) edge (12);
 \path (11) edge (20)
 			edge (12);
 \path (20) edge (21);
 \path (12) edge (21);
 \path (21) edge (30); 
\end{tikzpicture}
\caption{Degeneration diagram for Hodge types, cf. Proposition~\ref{prop:hodge-diamond-classification}.}\label{fig:Hodge type-diagram}
\end{figure}
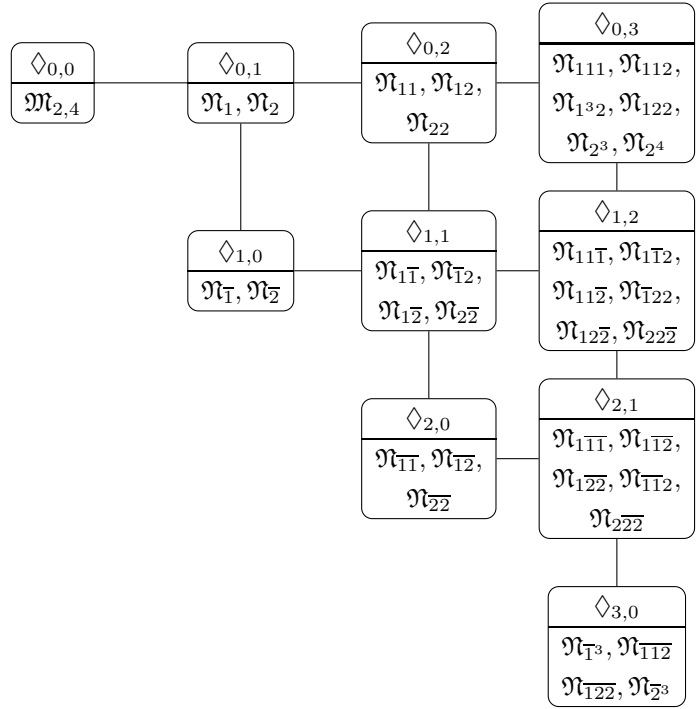

\subsection{The loci of non-normal surfaces}
We define a stratification of $\Mfrak^{(n)}$ analogously to the irrationality stratification of $\Nfrak$:
\begin{definition}
Let $1\leq n\leq 4$.
Given non-negative integers $a,b,c,d\geq 0$, we let $\Mfrak_{n;1^a\olone^{b}2^c\oltwo^d}\subset\Mfrak^{(n)}$ be the locus parametrising surfaces $X\in\Mfrak^{(n)}$ with exactly $a$ simply elliptic and exactly $b$ cuspidal singularities of degree $1$ and exactly $c$ simply elliptic and exactly $d$ cuspidal singularities of degree $2$.
We apply the analogous abbreviation-conventions as before.
\end{definition}

The arguments used to prove Proposition~\ref{prop:normality-stratification} and Proposition~\ref{prop:irrationality-stratification} also show:
\begin{proposition}
The strata $\Mfrak_{n;1^a\olone^{b}2^c\oltwo^d}\subset\Mfrak^{(n)}$ are locally closed and the closure of one stratum is contained in a union of strata.
\end{proposition}

Unlike the irrationality stratification of the locus of normal surfaces, this stratification is not finer than the Hodge type stratification.
Namely, singularities of type $J_{2,\infty}$, $X_{\infty}$ or $Y_{r,\infty}$ strongly affect the Hodge structure in a way that is hard to control.
This is why we concentrate on isolated irrational singularities in the irrationality stratification of the locus of non-normal surfaces, even though it its insufficient for reading off the Hodge type.
The Examples~\ref{example:wild-Hodge type-on-non-normal1} and \ref{example:wild-Hodge type-on-non-normal2} below illustrate this.

We proceed by investigating which strata are inhabited.
Afterwards, we compute their dimensions and the birational types of their members.

{The stratum $\Mfrak^{(4)}$:}
The members of $\Mfrak^{(4)}$ are the double-covers of $\PP^2$ branched over the double-quartics with at worst nodes.
Thus, there is only one inhabited stratum $\Mfrak_{4;\emptyset} = \Mfrak^{(4)}$.
It is isomorphic to the moduli space of nodal plane quartics (cf. Hassett \cite{Hassett:1999}) which is rational (as shown by Katsylo \cite{Katsylo:1992}).

{The stratum $\Mfrak^{(3)}$:}
The members of $\Mfrak^{(3)}$ are double-covers of $\PP^2$ branched over a reduced conic and a double-cubic with at worst nodes.
Since a reduced conic has at worst a node, the members of $\Mfrak^{(3)}$ do not have isolated irrational singularities.
Hence, yet again, only $\Mfrak_{3;\emptyset} = \Mfrak^{(3)}$ is inhabited.

{The stratum $\Mfrak^{(2)}$:}
Since a member $X\in \Mfrak^{(2)}$ is a double-cover of $\PP^2$ branched over $B = B'+2B''$, where $B'$ is a reduced quartic and $B''$ is a reduced conic, the isolated elliptic singularities come from the non-simple singularities of the quartic, of which only one is possible, namely, an ordinary quadruple-point, arising only as the union of four concurrent lines by Hui's classification \cite{Hui:1979}.
Therefore, $\Mfrak^{(2)} = \Mfrak_{2;\emptyset}\cup\Mfrak_{2;2}$.

{The stratum $\Mfrak^{(1)}$:}
A member of $\Mfrak^{(1)}$ has a branch divisor of the form $B'+2B''$ where $B''$ is a line and $B'$ is a reduced sextic.
By Proposition~\ref{prop:sextic-singularities}, the only possible non-simple singularities $B'$ might have are either a (possibly degenerate) quadruple-point, or a (possibly degenerate) $[3;3]$-point, or two non-degenerate $[3;3]$-points (with distinct distinguished tangent lines).
Furthermore, in the last case, $B'$ is the union of three conics meeting in the two $[3;3]$-points.
Thus, the inhabited strata of $\Mfrak^{(1)}$ are $\Mfrak_{1;\emptyset}$, $\Mfrak_{1;1}$, $\Mfrak_{1;2}$, $\Mfrak_{1;\olone}$, $\Mfrak_{1;\oltwo}$ and $\Mfrak_{1;11}$.

\begin{table}\centering
\begin{tabular}{lcl}
\toprule
Strata & Dimension & Birational type of normalisation \\ 
\midrule
$\Mfrak^{(4)} = \Mfrak_{4;\emptyset}$ & 6 & $\PP^{2}\amalg \PP^{2}$\\
$\Mfrak^{(3)} = \Mfrak_{3;\emptyset}$ & 6 & Rational\\
$\Mfrak_{2;\emptyset}$ & 11 & Weak del Pezzo of degree 2\\
$\Mfrak_{2;2}$ & 3 & Ruled of genus 1\\
$\Mfrak_{1;\emptyset}$ & 21 & K3-Surface\\
$\Mfrak_{1;1}$, $\Mfrak_{1;\olone}$ & 12,11 & Rational\\
$\Mfrak_{1;2}$, $\Mfrak_{1;\oltwo}$ & 13,12 & Rational\\
$\Mfrak_{1;11}$ & 3 & Ruled of genus 1\\
\bottomrule
\end{tabular}
\caption{The strata for non-normal surfaces}\label{table:birational-classification-by-minimal-resolution-non-normal}
\end{table}
\begin{proposition}\label{prop:dim-non-normal-strata}
	All the strata $\Mfrak_{1;\emptyset}$, $\Mfrak_{1;1}$, $\Mfrak_{1;\olone}$, $\Mfrak_{1;2}$, $\Mfrak_{1;\oltwo}$, $\Mfrak_{1;11}$, $\Mfrak_{2;\emptyset}$, $\Mfrak_{2;2}$, $\Mfrak_{3;\emptyset}$, $\Mfrak_{4;\emptyset}$ are irreducible and of dimension as indicated in \ref{table:birational-classification-by-minimal-resolution-non-normal}.
\end{proposition}
\begin{proof}
We first show that $\Mfrak^{(n)}$ is irreducible for all $n = 1,\dots,4$.
As in Theorem~\ref{theorem:stack-isomorphism}, we denote by $U\subset |\osh{\PP^{2}}(8)|$ the locus of half-log-canonical plane octics.
The pre-image $U_{n}\subset |\osh{\PP^{2}}(8-2n)|\times |\osh{\PP^{2}}(n)|$ of $U$ under the closed embedding $|\osh{\PP^{2}}(8-2n)|\times |\osh{\PP^{2}}(n)|\to |\osh{\PP^{2}}(8)|$, $(B',B'')\mapsto B'+2B''$ is open, hence smooth and irreducible.
By construction, the composition $U_{n}\hookrightarrow U\to U/\PGL(3,\CC)\cong \XXX$ identifies $U_{n}/\PGL(3,\CC)$ with $\Mfrak^{(n)}$.
Since $U_{n}$ is irreducible, so is $\Mfrak^{(n)}$, as claimed.
In addition, this proves $$\dim\Mfrak^{(n)} = \dim |\osh{\PP^{2}}(8-2n)| + \dim |\osh{\PP^{2}}(n)| - 8,$$
which gives $\dim \Mfrak^{(1)} = 21$, $\dim\Mfrak^{(2)} = 11$ and $\dim\Mfrak^{(3)} = \dim\Mfrak^{(4)} = 6$.
(The stabiliser of a general plane curve of degree $\geq 3$ is discrete and in the cases under consideration we have either $n\geq 3$ or $8-2n\geq 4$.)
Since $\Mfrak_{n;\emptyset}$ is open in $\Mfrak^{(n)}$ we get the claimed results for these strata.

In the remaining cases, we argue similarly.
Let $\Mfrak\subset\Mfrak^{(n)}$ be any of the strata.
We let $V\subset U_{n}\subset|\osh{\PP^{2}}(8-2n)|\times |\osh{\PP^{2}}(n)|$ be the pre-image of $\Mfrak$ under the restricted classifying map $U_{n}\to\Mfrak^{(n)}$.
Then $V$ dominates $\Mfrak$, so that it would be enough to show that $V$ is irreducible.
However, it will be customary to restrict to certain sub-spaces in order to gain more control.

For the strata $\Mfrak\subset\Mfrak^{(1)}$, where $B''$ is a line, we can fix this line; then the condition for $B'$ along $B''$ is that their local intersection multiplicities are at most $2$ everywhere.
That this is an open condition follows as in the proof of Proposition~\ref{prop:normality-stratification}.

The easiest case is $\Mfrak = \Mfrak_{1;2}$.
For every member $B'+2B''\in V$, where $p\in \PP^{2}$ is the quadruple-point of $B'$, there is a plane automorphism mapping $p$ to the point $(0;0;1)$ and the line $B''$ to the line at infinity $L = \{z = 0\}$, where we are using homogeneous coordinates $(x;y;z)\in\PP^{2}$.
(We could have used any pair of a point and a line missing the point, of course.)
The linear system of sextics with multiplicity at least $4$ in $(0;0;1)$ is of dimension $17$ \cite[sextics.m2~I.1]{Anthes:2018}.
Let $V'\subset|\osh{\PP^{2}}(6)|$ be the locus of sextics $B'$ such that $B'+2L$ is a member of $V$.
Then the $\PGL(3,\CC)$-orbit of $V'+2L\subset V$ is all of $V$ and so $V'/\PGL(3,\CC)\cong \Mfrak$.
Note that a sextic $B'$ with a quadruple-point at $(0;0;1)$ lies in $V'$ if and only if it is reduced, the quadruple-point is non-degenerate, all remaining singularities are simple and every intersection with $B''$ has multiplicity at most $2$.
All these conditions are open in the linear system of sextics with multiplicity $\geq4$ in $(0;0;1)$.
Therefore, $V'$ is irreducible and, hence, so is $V'/\PGL(3,\CC)\cong \Mfrak_{1;2}$.
Since the group of automorphisms fixing a point and a line missing the point is of dimension $4$, we conclude that $\Mfrak_{1;2}$ is irreducible and of dimension $17-4=13$.

The stratum $\Mfrak_{1;\oltwo}$ is handled similarly; the difference is that we also have to fix the special tangent direction, which we can still do using automorphisms.
This way, we get an open sub-set of a linear sub-space of dimension $15$ \cite[sextics.m2~I.2]{Anthes:2018}, with stabiliser of dimension $3$, so that this stratum is irreducible of dimension $12$.

Let us turn to $\Mfrak_{1;1}$, where we argue similarly.
Again, we can fix the singular point, which we want to be a $[3;3]$-point, so we should also fix the distinguished tangent line and we still have automorphisms left to fix the line $B''$.
The linear systems of sextics with at least a $[3;3]$-point in a fixed point and with fixed special tangent direction is of dimension $15$ \cite[sextics.m2~II.1]{Anthes:2018}.
By Lemma~\ref{lemma:degree-bounds}~\ref{lemma:degree-bounds:item:n-and-m-fold}, the only other singularities a reduced sextic can have besides a $[3;3]$-point are at most triple-points and they have to be off the distinguished tangent line.
Furthermore, if a reduced sextic has a $[3;3]$-point, then it has at most one other non-simple singularity, which is another $[3;3]$-point; a closed condition.
Since the stabiliser is of dimension $3$, again we conclude that the stratum under consideration is irreducible and $12$-dimensional.

The same argument shows that $\Mfrak_{1;\olone}$ is irreducible and of dimension $14-3 = 11$ since the sub-space of $|\osh{\PP^{2}}(6)|$ parametrising sextics with a degenerate $[3;3]$-point at a fixed point and a fixed tangent (but variable second order direction) is irreducible and of dimension $14$ as computed in \cite[sextics.m2~II.2]{Anthes:2018}.

We argue a little differently for $\Mfrak_{1;11}$.
Note that after fixing the locus of $[3;3]$-points with their distinguished tangent directions, the locus of admissible lines $B''$ is independent of the sextic, for the sextic is a union of three distinct conics passing through the points in distinguished tangent direction, it meets lines with multiplicity $3$ only in the $[3;3]$-points and so the double-line may be any line missing those two points.
Since the corresponding space of sextics is irreducible and one-dimensional \cite[sextics.m2~II.3]{Anthes:2018}, $\Mfrak_{1;11}$ is irreducible and of dimension $3$.

The last case we have to work out is $\Mfrak_{2;2}$.
The inverse image $V\subset U_{2}$ of $\Mfrak_{2;2}$ consists of the octics decomposing as $B'+2B''$ with a reduced quartic $B'$ and a reduced conic $B''$, where $B'$ has a quadruple-point, hence, is a union of four concurrent lines, and $B''$ has at worst nodes.
Furthermore, the nodes of $B''$ have to be off $B'$.
Since there is a $1$-parameter family of analytically distinct quadruple-points, we can neither fix the quartic, nor the conic, which could be smooth or a union of two lines.
However, we can consider the linear system in $|\osh{\PP^{2}}(4)|\times|\osh{\PP^{2}}(2)|$ given by pairs $(B',B'')$ where $B'$ is a quartic with a quadruple-point in $(1;1;1)$ and $B''$ is a conic in the pencil $\{\lambda xy+\mu z^2 = 0\}_{(\lambda;\mu)\in\PP^{1}}$.
In addition, we ask that the quartic contains the lines $\{x=z\}$ and $\{y = z\}$.
Since for any $B = B'+2B''\in V$, at most two lines in $B'$ can be tangent to $B''$, we find for at least two of the lines in $B'$ a transversal intersection point with $B''$.
Therefore, $B$ is projectively equivalent to a member of this $3$-dimensional linear system, up to finitely many choices of parameters.
The only exceptional parameters are either $(\lambda;\mu) = (0;1)$, or those where the quadruple-point is degenerate, or where $(\lambda;\mu) = (1;0)$ and where the quartic passes through the point $(0;0;1)$.
These conditions are clearly closed, so that we have an open, irreducible sub-scheme which is a finite cover of $\Mfrak_{2;2}$.
\end{proof}

Finally, we discuss the birational geometry of the non-normal surfaces:

\begin{proposition}\label{prop:birational-models-non-normal}
The minimal models of the minimal resolution of the possible non-normal Gorenstein stable surfaces $X$ satisfying $K_{X}^2 = 2$ and $\chi(\osh{X}) = 4$ are as listed in \ref{table:birational-classification-by-minimal-resolution-non-normal} above.
\end{proposition}
\begin{proof}
Let $X$ be a non-normal Gorenstein stable surface satisfying $K_{X}^2 = 2$ and $\chi(\osh{X}) = 4$, let $\overline{X}$ be its normalisation, denote its minimal resolution by $g\colon Y\to \overline{X}$ and let $\sigma\colon Y\to \minmod{Y}$ be a minimal model of $Y$.
By our description of the normalisation (Proposition~\ref{prop:description-of-normalisation}), the branch curve $B\subset\PP^{2}$ of the canonical double-cover decomposes as $B = B'+2B''$ for two reduced effective divisors $B',B''$ and $\overline{X}$ is the double-cover of $\PP^{2}$ branched over $B'$.
We have to have $\deg(B')\in\{0,2,4,6\}$, where $X\in\Mfrak^{(i)}$ if and only if $\deg(B') = 2i$.

In case $B' = 0$, observe that $X = \PP^{2}\amalg_{B''}\PP^{2}$, as the double-cover branched over $2B''$ and $Y = \overline{X} = \PP^{2}\amalg\PP^{2}$ is the unbranched double cover of the plane.

For the remaining cases, we make use of formulas and basic facts concerning double-covers which can be found in Barth, Hulek, Peters, Van de Ven \cite[V~22]{BHPV}.

If $\deg(B') = 2$, we have two cases: either $B'$ is a smooth conic, or the union of two lines.
In the first case, $\overline{X}$ is $\PP^{1}\times\PP^{1}$ and in the latter, it is the quadric cone, which is resolved by the Hirzebruch surface $\FF_{2}$.
That is, the minimal models are all rational.

If $X\in \Mfrak^{(2)}$, then $B'$ is a quartic with only simple singularities, unless it is a union of four concurrent lines (cf. Hui's classification~\cite{Hui:1979}).
If $X\in\Mfrak_{2;\emptyset}$, i.e., $B'$ has only simple singularities, $\overline{X}$ is a del Pezzo surface of degree $2$ (possibly singular, with ADE-singularities corresponding to those of $B'$), the double-cover being defined by the anti-canonical linear system.
In fact, $-K_{Y} = -g^{*}K_{\overline{X}}$ is ample, as the pull-back of an ample bundle along a finite morphism and $K_{Y}^2 = K_{\overline{X}}^2 = 2$ since the degree is $2$.

If $X\in\Mfrak_{2;2}$, i.e., $B'$ is the union of four concurrent lines, then the pencil of lines through their common intersection point gives rise to a ruling of $Y$ over a curve of genus $1$.
Explicitly, the blow up of $\PP^{2}$ in the quadruple-point is the Hirzebruch surface $\FF_{1}$ and the double-cover over the four fibres of the ruling $\FF_{1}\to \PP^{1}$ coming from the four branches of $C$ induces a ruling $Y\to E$, where the elliptic curve $E$ is the double-cover of $\PP^{1}$ branched over the four points corresponding to the lines in question.

Finally, in case $X\in \Mfrak^{(1)}$, where $B'$ is a sextic, the only non-simple singularities $B'$ can have are: $(a)$ none; $(b)$ a (not necessarily ordinary) quadruple-point, $(c)$ a (possibly degenerate) [3;3]-point (see~\ref{def:n-n-sings} for the definition), $(d)$ a pair of non-degenerate $[3;3]$-points. In case $(d)$, the sextic decomposes as the union of three conics (of which at least two are smooth), passing through the two $[3,3]$-points.
Here, the cases $(a)$, $(b)$, $(c)$ and $(d)$ correspond to the cases that $X\in \Mfrak_{1;\emptyset}$, or $X\in \Mfrak_{1;2}\cup\Mfrak_{1;\oltwo}$, or $X\in \Mfrak_{1;1}\cup\Mfrak_{1;\olone}$, or $X\in \Mfrak_{1;11}$, respectively.

That $Y$ is a K3-surface in case $(a)$ is well-known.
Furthermore, by the canonical bundle formula, in all cases, $\omega_{\overline{X}} = \osh{\overline{X}}$.
In the remaining cases $(b)$--$(d)$, this implies $\kappa(Y) = -\infty$, for, $-K_{Y}$ is the sum of the exceptional divisors over the elliptic singularities.
Therefore, $\minmod{Y}$ is either rational (if $\chi(\osh{Y}) = 1$), or ruled of genus $1-\chi(\osh{Y})\geq 1$.
Thus, we only have to compute the holomorphic Euler characteristic of $Y$.
Since $\overline{X}$ is a flat degeneration of a K3-surface, $\chi(\osh{\overline{X}}) = 2$.
From this we finally conclude $\chi(\osh{Y}) = 1$ in case $(b)$ or $(c)$ and $\chi(\osh{Y}) = 0$ in case $(d)$, as claimed.
\end{proof}

\subsubsection{The Hodge type in the non-normal case}\label{sec:non-normal-Hodge type}

For a double-cover of the plane $X\to\PP^{2}$ branched over a half-log-canonical curve $B = B'+2B''$ where $B'$ and $B''$ are reduced, the Hodge type of $X$ depends not just on the irrational singularities of $B'$, but also on the nodes of $B''$ and the way how $B'$ and $B''$ meet.
In particular, the Hodge type is not constant on the strata of the irrationality stratification of the locus of non-normal surfaces.
The list of possibilities gets quite complicated for the cases we would have to consider here.
By way of example, we indicate the possible Hodge types on $\Mfrak_{2;\emptyset}$ and $\Mfrak^{(4)} = \Mfrak_{4;\emptyset}$.
The remaining strata can be dealt with analogously.

\begin{example}\label{example:wild-Hodge type-on-non-normal1}
Given $X\in\Mfrak_{2;\emptyset}$, we let $\overline{X}\to X$ be the normalisation and denote the conductor loci by $F\subset X$ and $\overline{F}\subset \overline{X}$.
As explained in the introduction, $X$ is the push-out of the diagram $F\leftarrow\olF\to \overline{X}$ and by Peters and Steenbrink \cite[Corollary-Definition 5.37]{PS:2008} we get the associated Mayer--Vietoris  exact sequence
$$0\to H^1(\overline{F};\CC)\to H^2(X;\CC)\to H^2(\overline{X};\CC)\oplus H^2(F;\CC)\to H^2(\overline{F};\CC).$$
In fact, the left-most term is $H^1(\overline{X};\CC)\oplus H^1(F;\CC) = 0$, which can be seen as follows:
Recall from Proposition~\ref{prop:description-of-normalisation} that if $B = B'+2B''$ is the branch curve of $X$, where $B'$ is a reduced quartic and $B''$ is a reduced conic, then $\overline{X}$ is the double-cover branched over $B'$ and $F\cong B''$.
In particular, $H^{1}(F;\CC) = 0$.
Since $\overline{X}$ is rational, $H^{1}(\overline{X};\CC) = 0$ as well.

If, in addition, $\olF$ is connected, then $H^{2}(F;\CC)\to H^2(\overline{F};\CC)$ is an isomorphism.
Thus, $\dim(H^2(X;\CC))^{2,0} = h^{2,0}(\overline{X})=0$, $\dim(H^2(X;\CC))^{1,0} = \dim(H^1(\overline{F};\CC))^{1,0}$ and $\dim(H^2(X;\CC))^{0,0} = \dim(H^1(\overline{F};\CC))^{0,0}$.
Now recall that $\overline{F}$ is a double-cover of $B''$ branched over $B''|_{B'}$.
Thus, if $B'$ and $B''$ meet transversely, then $\overline{F}$ is a smooth curve of genus $3$, hence $X$ has Hodge type $\lozenge_{0,3}$.
As $B''|_{B'}$ gets doubled points, either due to tangency or due to double-points of $B'$ along $B''$, $\overline{F}$ acquires nodes.
This results either in a nodal curve of genus $2$, a curve of genus $1$ with $2$ nodes, a curve of genus $0$ with $3$ nodes, or the union of $2$ rational curves meeting transversely in $4$ points.
The first two cases give Hodge types $\lozenge_{1,2}$, $\lozenge_{2,1}$, the latter two have Hodge type $\lozenge_{3,0}$.

If $B''$ is the union of two lines, then by the same argument as above, the $(2,0)$-part is trivial, so that there are no more possible Hodge types than those above.
The same applies if we pass to $\Mfrak_{2;2}$, where $B'$ has a quadruple-point.
\end{example}

\begin{example}\label{example:wild-Hodge type-on-non-normal2}
Let $B''$ be a smooth or nodal but reduced quartic in $\PP^2$.
With branch curve $B = 2B''$, we get that $X = \PP^{2}\amalg_{B''}\PP^{2}$ and $\overline X = \PP^{2}\amalg \PP^{2}$ with $\olF = B''\amalg B''$ and $F\cong B''$ in such a way that $\pi|_{\olF}\colon \olF\to F$ is the trivial double-cover.
Tracing through the maps in the Mayer--Vietoris sequence for $X$ as the push-out of the diagram $F\leftarrow\overline{F}\to \overline{X}$, one quickly finds an isomorphism of Deligne's mixed Hodge structures $H^2(X;\CC)\cong H^1(B'';\CC)$.
Therefore, the surfaces parametrised by the irreducible stratum $\Mfrak^{(4)} = \Mfrak_{4;\emptyset}$ realise all Hodge types $\lozenge_{r,s}$ with $r+s = 3$.
\end{example}

\section{Further remarks and questions}
\label{section:further-remarks}
\subsection{Comparison with known compact moduli spaces of curves}
There are at least three related compactifications of the moduli space of smooth plane curves of degree $8$, namely, the GIT-quotient of $|\osh{\PP^{2}}(8)|$ under the action of $\PGL(3,\CC)$, Hassett's moduli space of stable log-surfaces which admit a smoothing to $(\PP^{2},C)$ with $C$ a curve of degree $8$ \cite{Hassett:1999} and Hacking's moduli space $\Mcal_{8}$ of so-called \emph{stable pairs of degree $8$}, namely, pairs $(X,D)$ consisting of a surface $X$ and an effective $\QQ$-Cartier $\ZZ$-divisor $D$ on $X$, where $\osh{X}(3D+8K_{X})\cong\osh{X}$ and such that $(X,(\tfrac{3}{8}+\epsilon)D)$ is a stable log-surface for some $\epsilon > 0$, subject to a smoothability condition that makes sure that the plane octics are dense \cite{Hacking:2004}.

\begin{question}\label{rem:hacking-compare}
Perhaps, a $\QQ$-Gorenstein degeneration of a smooth (or Gorenstein) stable surface $X\in\Mfrak_{2,4}$, say canonically the double-covers of $\PP^2$ branched over $B\in |\osh{\PP^2}(8)|$, will itself be a double-cover of a surface $X'$ branched over some curve $B'$, where $(X',\tfrac{1}{2}B')$ is semi-log-canonical.
This raises the question whether (the closure of $\XXX$ in) $\overline{\Mfrak}_{2,4}$ is isomorphic to some projective moduli space of semi-log-canonical pairs $(X,D)$ which are in some sense degenerations of log-canonical pairs of the form $(\PP^2,\tfrac{1}{2}B)$ with an octic $B$.
(For an example, see \ref{ex:p114} below.)

More concretely, let $\Mcal_{8}$ be Hacking's moduli stack of  $\QQ$-Gorenstein smoothable families of stable pairs of degree $8$, which is a separated, proper and smooth Deligne--Mumford stack (Hacking \cite[Theorem 4.4 \& 7.2]{Hacking:2004}), with coarse moduli space denoted by $M_{8}$.

Each half-log-canonical octic $B\in U$, gives rise to a stable pair of degree $8$, $(\PP^{2},B)$, for, $(\PP^{2},(\tfrac{3}{8}+\epsilon)B)$ is a stable log-surface for all $0<\epsilon\leq \tfrac{1}{8}$.
This induces a morphism $\XXXst\to \Mcal_{8}$ which seems worthwhile to study.
\emph{Does it extend to a morphism $\olMcal_{2,4}\to \Mcal_{8}$?}

In that case, one naive hope would be that the locus of stable pairs $(X,D)$ of degree $8$ such that $(X,\tfrac{1}{2}D)$ is semi-log-canonical is closed and isomorphic to $\olMfrak_{2,4}$, but this locus is not closed in $M_{8}$ (see Example~\ref{ex:blabla} below).
\end{question}

\begin{remark}\label{rem:GIT-embedding}
For a stable pair $(\PP^2,D)$ of degree $d$ in the sense of Hacking, the curve $D$ is GIT-stable, see Hacking \cite[Section 10]{Hacking:2004}.
In particular, we get a morphism from $\XXX$ to the GIT quotient $|\osh{\PP^2}(8)|^{\textnormal{ss}}/\PGL(3,\CC)$ and, thus, yet another possible compactification which could be studied. 
\end{remark}

\begin{remark}
While Hacking's moduli space $M_{8}$ properly contains $\XXX$, Hassett's space $\olPcal_{8}$ is too small; the plane curves it parametrises have to have log-canonical threshold at least $1$.
\end{remark}

\begin{question}
Recall that the stratum $\Mfrak^{(4)} = \Mfrak_{4;\emptyset}$ is isomorphic to the moduli space of nodal plane quartics.
Is the closure of $\Mfrak^{(4)}$ in $\olMfrak_{2,4}$ isomorphic to Hassett's  compactification of the space of smooth plane quartics \cite{Hassett:1999}?
\end{question}

\subsection{Beyond the Gorenstein locus}\label{section:beyond-gorenstein}

We briefly demonstrate that the Gorenstein locus $\XXX$ is properly contained and not closed in $\olMfrak_{2,4}$.

\begin{example}\label{ex:p114}
The log-canonical surface $\PP(1,1,4)$ has an essentially unique $1$-pa\-ra\-me\-ter $\QQ$-Gorenstein smoothing $\Zscr\to\AA^1$, where $\Zscr\subset \PP(2,2,2,4)\times\AA^{1}$ is given as the vanishing locus of the polynomial $x_{1}^2+ty-x_{0}x_{2}$, where $t$ is the coordinate of $\AA^{1}$ and where $x_{0},x_{1},x_{2}$ and $y$ are the coordinates of $\PP(2,2,2,4)$, cf.\ Hacking's exposition \cite[p.\ 52 f]{Hacking:2016}.
For all $t\not= 0$, the fibre $\Zscr_{t}$ is isomorphic to $\PP(2,2,2)\cong\PP^{2}$ and $\Zscr_{0}$ is isomorphic to $\PP(1,1,4)$.
Let $Q\subset \PP(2,2,2,4)$ be a sufficiently general hypersurface of degree $16$, defining a relative Cartier divisor $\Zscr\cap Q$ missing the singular point in the special fibre.
Assume furthermore that each pair $(\Zscr_{t},\tfrac{1}{2}\Zscr_{t}\cap Q)$ is log-canonical, at least for all $t$ sufficiently close to $0$.
For example, we may assume that each of the curves $Q_{t} := \Zscr_{t}\cap Q$ is smooth.
Let $\Xscr\to\Zscr$ be the double-cover branched over $Q\cap\Zscr$; this defines a $\QQ$-Gorenstein family over $\AA^{1}$ where each fibre $\Xscr_{t}$, $t\not=0$, is a double-cover of $\PP^2$ branched over an octic, hence, a Gorenstein stable surface with $K_{\Xscr_{t}}^2 = 2$ and $\chi(\osh{\Xscr_{t}}) = 4$, whilst the central fibre $\Xscr_{0}$ is a double-cover of $\PP(1,1,4)$, branched over a curve $Q_{0}$ of degree $16$ missing the singular point and such that the pair $(\PP(1,1,4),\tfrac{1}{2}Q_{0})$ is log-canonical.
Hence, $\Xscr_{0}$ is semi-log-canonical and $K_{\Xscr_{0}}$ is the pull-back of the $\QQ$-Cartier divisor $K_{\PP(1,1,4)}+\tfrac{1}{2}Q_{0}\in|\osh{\PP(1,1,4)}(2)|$, so that $\Xscr_{0}$ is stable of Gorenstein-index $2$.
Since the family is $\QQ$-Gorenstein, we conclude that $\Xscr_{0}$ is in the closure of $\XXX$ in $\olMfrak_{2,4}$.
\end{example}

Note that the pair $(\Zscr_{0},Q_{0})$ of the above example is an element of Hacking's moduli space $\Mcal_{8}$ at the boundary of the image of $\XXX$.
There are, however, many elements in $\Mcal_{8}$ which are not contained in the image of $\XXX$, despite the fact that $\XXX$ is dense in $\Mcal_{8}$.
There is also something to say about the boundary of (the image of) $\XXX$ in Hacking's moduli space $\Mcal_{8}$ of stable pairs of degree $8$ (and, therefore, also in the GIT-quotient).
Namely, there are curves $C\subset\PP^2$ such that $(\PP^{2},C)$ is a stable pair of degree $8$, but where $(\PP^{2},\tfrac{1}{2}C)$ is not log-canonical; since the image of $\XXX$ in $\Mcal_{8}$ is dense, they occur as limits of classes of half-log-canonical curves, though.

\begin{example}\label{ex:blabla}
Every octic $C\subset\PP^2$ with global log-canonical threshold between $\tfrac{3}{8}$ and $\tfrac{1}{2}$ gives rise to a pair in Hacking's moduli space $\Mcal_{8}$ which is not contained in the image of $\XXX$.

One kind of example is given by general octics with a singular point of type $Z_{11}$, whose log-canonical threshold is $\tfrac{7}{15}$.
(A curve singularity of type $Z_{11}$ is analytically locally the union of an $E_{6}$-singularity and a general line passing through it.
To get an explicit example, we can just take a quartic with an $E_{6}$-singularity and a general quartic passing through the $E_{6}$-singularity, resulting in an octic with one singularity of type $Z_{11}$ and $13$ ordinary double-points.)
This is by far not the only type of singularity that can occur on an octic which has no (or only admissible) other singularities.
\end{example}
\appendix
\section{Half-log-canonical plane curves of small degree}
\label{section:curves}
In this appendix, we prove the results about plane curves of degree at most eight with $[3;3]$- and quadruple-points which were used in the earlier chapters.
We obstruct the existence of certain configurations using basic intersection theory and well-known results about the Milnor number.

The classification of possible (configurations of) singularities on plane conics or cubics is easy.
In the case of quartics, a complete classification is known; it can be found in Hui's thesis \cite{Hui:1979}.
It turns out that the only reduced quartics with a non-simple singularity are the unions of four concurrent lines, admitting a unique ordinary quadruple-point, also called singularity of type $X_{9}$.
Degtyarev \cite{Degtyarev:1989} has classified all plane quintics up to rigid isotopy and all the possible singularities on quintics.
A list of possible configurations with total Milnor number at least $12$ can be found in Wall \cite{Wall:1996}.
The last case where the complete classification is known, the sextics, is mostly due to Urabe \cite{Urabe:1989}, Yang \cite{Yang:1996} and Degtyarev, see \cite[Section~7.2.3]{Degtyarev:2012} and the references therein.

Already the list provided by Yang \cite{Yang:1996} (even though restricted to the sextics with maximal total Milnor number $19$) is so long that for certain questions, it is not easy to read off the relevant informations from the data.
There are 128 irreducible maximising sextics and many more reducible ones and \emph{``[t]he list [of the remaining reduced sextics] is too long to be printed [in an article]''} \cite[Remark~4.1]{Yang:1996}.
In conclusion, the classification of possible configurations of singularities on octic curves is clearly out of reach.
Therefore, we study here just as much as we need to understand the strata; that is, we ignore simple singularities and concentrate only on those with log-canonical threshold exactly $\tfrac{1}{2}$.

For a start, we recall the notion of an $n$-fold-point with an infinitely near $n$-fold-point, an $[n;n]$-point, for short.
A non-degenerate $[n;n]$-point should be pictured as $n$-fold-points with $n$ local branches with a common tangent direction; however, degenerate $[n;n]$-points may have less branches.

\begin{definition}\label{def:n-n-sings}
Let $0\in C\subset\CC^2$ be the germ of an isolated curve singularity, defined by a convergent power series $f\in\CC\{x,y\}$.
\begin{enumerate}[label=\arabic*.]
\item The germ (or the point, slightly abusively) is said to be a \emph{$n$-fold-point}, if $f$ has multiplicity $n$, i.e., $f\in(x,y)^n-(x,y)^{n+1}$.
\item If, moreover, the degree $n$-part  of $f$ is the product of $n$ pairwise distinct linear factors, then the germ is said to be an \emph{ordinary $n$-fold-point}. That is, $C$ is a union of $n$ smooth curves meeting transversely in $0$.
\item An $n$-fold-point $0\in C\subset\CC^2$ is said to have an \emph{infinitely near $n$-fold-point} if the strict transform $C'\subset X$ of $C$ in the blow-up $(X,E)\to(\CC^2,0)$ has an $n$-fold-point along $E$.
In this case, the germ is called an $[n;n]$-point.
It is called \emph{non-degenerate} if the $n$-fold-point of $C'$ is ordinary.
\item If $C\subset\PP^2$ is a plane curve with an $[n;n]$-point $p\in C$, the point on the exceptional line $E\subset\mathrm{Bl}_{p}\PP^2$ where the strict transform $C'$ has its $n$-fold-point corresponds to a tangent direction at $p$ in $\PP^2$; we refer to this as the \emph{distinguished tangent} of the $[n;n]$-point. The unique line in $\PP^2$ passing through $p$ in this direction is called the \emph{distinguished tangent line}.
\end{enumerate}
\end{definition}
\begin{remark}
Note that the distinguished tangent line $\ell\subset\PP^2$ of an $[n;n]$-point $p\in C\subset\PP^2$ is determined by the property that the intersection multiplicity of $\ell$ and $C$ at $p$ exceeds $n$.
In particular, if $C$ contains a line through $p$, then this must be its distinguished tangent line.
\end{remark}

We will mostly be concerned with certain $[2;2]$-points, $[3;3]$-points and $4$-fold-points, also known as \emph{quadruple-points}.
We give a quick overview:

A $[2;2]$-point is a double-point whose strict transform in the blow-up has a double-point along the exceptional line.
That is, the $[2;2]$-points are the singularities of type $A_{n}$ as $n\geq 3$, where $A_{3}$ is the non-degenerate $[2;2]$-point.

The half-log-canonical $[3;3]$-points are the singularities of type $J_{10}$ (the non-degenerate $[3;3]$-point) and $J_{2,p}$ for $p\geq 1$.
Blowing up once, the strict transform of a $J_{10}$ has a non-degenerate triple-point (a $D_{4}$) along the exceptional line and the strict transform of a $J_{2,p}$ has a $D_{4+p}$ along the exceptional line.
Moreover, the branches are transversal to the exceptional line, for otherwise it would be a quadruple-point (or worse).
In particular, no component of a $[3;3]$-point is an ordinary cusp $A_{2}$.

Likewise, the ordinary quadruple-points are the singularities of type $X_{9}$, whereas the half-log-canonical degenerate quadruple-points split up into the two families $X_{p}$, $p\geq 10$ and $Y_{r,s}$ for $r,s\geq 1$.
The singularities of type $X_{p}$ are, locally analytically, the union of a degenerate double-point of type $A_{p-8}$ and a non-degenerate double-point and those of type $Y_{r,s}$ are unions of two degenerate double-points of type $A_{r+1}$ and $A_{s+1}$.
In particular, an $X_{p}$, $p\geq 10$, has a single special tangent direction (the distinguished tangent direction of the underlying degenerate double-point) and a $Y_{r,s}$ has two such.

\begin{lemma}\label{lemma:degree-bounds}
Let $C$ be a plane curve of degree $d$.
Then the following hold:
\begin{enumerate}[label=\alph*)]
\item If $C$ has an $n$-fold-point, then $d\geq n$ and $d = n$ if and only if $C$ is a union of $n$ concurrent lines, the intersection-point being the $n$-fold-point.
\item\label{lemma:degree-bounds:item:n-and-m-fold} If $C$ has an $m$-fold-point and an $n$-fold-point, then $d\geq m+n-1$ and if $d = m+n-1$, then $C$ contains the line joining those two points.
\item\label{lemma:degree-bounds:item:collinear} More generally, if $C$ has $s$ collinear singular points of multiplicity $n_{i}$, $i=1,\dots,s$, then $d\geq 1-s+\sum_{i=1}^s n_{i}$ and if $d = 1-s+\sum_{i=1}^s n_{i}$, then the line joining them is contained in $C$.
\item\label{lemma:degree-bounds:item:one-nn-point} If $C$ has an $[n;n]$-point, then $d\geq 2n-1$ and if $d = 2n-1$, then $C$ contains the distinguished tangent line.
\item\label{lemma:degree-bounds:item:nn-and-m} If $C$ has an $[n;n]$-point and an $m$-fold point on the distinguished tangent line of the $[n;n]$-point, then $d\geq 2n+m$, unless $C$ contains the distinguished tangent, in which case $d \geq 2n+m-2$.
\item If $C$ has an $[m;m]$- and an $[n;n]$-point with a common distinguished tangent line, then $d\geq 2n+2m-3$ and if  $d< 2n+2m$, then $C$ contains the distinguished tangent line.
\end{enumerate}
\end{lemma}
\begin{proof}
The proofs of those statements are very similar; by way of example, we only prove a few of them.
Note that since we assumed $n$-fold-points and $[n;n]$-points to be isolated singularities, if $C$ contains a line $L$ through such a point, is does so with multiplicity $1$. In particular, the residual curve $C-L$ (in divisor-notation) has degree $d-1$ and does not contain $L$.
\begin{enumerate}[label=\emph{\alph*)}]
\setcounter{enumi}{2}
\item If $C$ is as claimed, then the line $L$ joining the $s$ singular points witnesses $d = CL\geq \sum_{i=1}^sn_{i}$, unless $L\subset C$, in which case the same argument applied to the residual curve $C' = C-L$ yields $d-1 = C'L\geq \sum_{i=1}^s(n_{i}-1)$, hence the claim.
\item If $C$ has an $[n;n]$-point, then the distinguished tangent line $L$ witnesses that $d = CL\geq 2n$, unless $C$ contains $L$. In this case, the residual curve $C' = C-L$ has an $[n-1;n-1]$-point with distinguished tangent direction $L$; thus, we get $d-1 = C'L\geq 2(n-1)$, hence, $d\geq 2n-1$, as claimed.
\end{enumerate}
\end{proof}

The following is an immediate corollary.
\begin{corollary}\label{cor:constraints-33-pts}
Let $C$ be a plane octic curve. Then the following holds:
\begin{enumerate}[label=\alph*)]
\item\label{cor:constraints-33-pts:item:distinct-tangents} Any two $[3,3]$-points on $C$ have distinct distinguished tangent lines.
\item\label{cor:constraints-33-pts:item:no-3-collinear} No three $[3;3]$-points on $C$ are collinear.
\end{enumerate}
\end{corollary}

This implies that two $[3;3]$-points on a conic are in general position such that there is exactly a pencil of conics joining both points, passing through them in distinguished tangent direction.
Three $[3;3]$-points on a conic can be in special position in the sense that the tangents may align along a conic.
It turns out that there are at most three $[3;3]$-points on a plane octic, but before we can prove this, we have to recall a few basic facts from singularity theory.
We refer to Milnor's seminal book \cite{Milnor:1968} or Wall \cite[Chapter~6]{Wall:2004} for the local theory.

Recall that with a holomorphic function germ $f\in\CC\{x,y\}$ we can associate the \emph{Milnor number} $\mu(f) = \dim_{\CC}(\CC\{x,y\}/J_{f})$ where $J_{f} = (\tfrac{\partial f}{\partial x},\tfrac{\partial f}{\partial y})$ is the Jacobian ideal of $f$ generated by the partial derivatives.
Let $C\subset\PP^{2}$ be a plane curve passing through a point $p\in C$ and choose local holomorphic coordinates $x,y$ at $p\in\PP^{2}$.
Then the curve $C$ is the vanishing locus of a function germ $f\in\CC\{x,y\}$ and it makes sense to define the \emph{Milnor number of $C$ at $p$}, $\mu_{p}(C):=\mu(f)$.
If $C$ is reduced, we define its \emph{total Milnor number} $\mu(C) = \sum_{p\in C_{\textnormal{sing}}}\mu_{p}(C)$.

We recall from Wall's exposition \cite[Sections~7.1 \&~7.5]{Wall:2004}:
\begin{lemma}\label{lemma:milnor-number-stuff}
Let $C\subset\PP^{2}$ be a reduced plane curve of degree $d$.
\begin{enumerate}[label=\alph*)]
\item\label{lemma:milnor-number-stuff:item1} The total Milnor number of $C$ is bounded by $\mu(C)\leq (d-1)^2$ and the maximum $\mu(C) = (d-1)^2$ is attained only by the union of $d$ concurrent lines with its $(d-1)$-fold-point.
\item\label{lemma:milnor-number-stuff:item2} Let $C^\nu$ be the normalisation of $C$.
Then
$$\chi\nolimits_{\textnormal{top}}(C^\nu) = (3-d)d+\sum_{p\in C_{\textnormal{sing}}}(\mu_{p}(C)+r_{p}(C)-1),$$
where $r_{p}(C)$ is the number of analytically local branches of $C$ through $p$.
\end{enumerate}
In particular, if $C$ is a reduced octic with four $[3;3]$- or quadruple-points, then $C$ has at least four rational components, $\chi_{\textnormal{top}}(C^{\nu})\geq 8$, and if $C$ has exactly four components, then there are no additional singular points.
\end{lemma}
\begin{proof}
For \ref{lemma:milnor-number-stuff:item1} and \ref{lemma:milnor-number-stuff:item2} see Wall \cite[Section~7.5, p.~177 \& Corollary~7.1.3]{Wall:2004}.
The \emph{in particular}-part can be derived from \ref{lemma:milnor-number-stuff:item2} as follows:
It suffices to show that for a $[3;3]$- or quadruple-point $p\in C$, we have $\mu_{p}(C)+r_{p}(C)-1\geq 12$, since then, if there are at least four such, $\chi_{\textnormal{top}}(C^\nu)\geq 8$ and equality implies that there are no more singular points and that $\mu_{p}(C)+r_{p}(C)-1 = 12$ for all four $p\in C_{\textnormal{sing}}$.
As the normalisation decomposes into disjoint components $C^{\nu} = \coprod_{i=1}^sC^{\nu}_{i}$, we get $\sum_{i=1}^s 2-2g(C_{i}^\nu) = \chi_{\textnormal{top}}(C^\nu)\geq 8$, hence $g(C_{i}^\nu) = 0$ at least four times, which yields four rational components.
Therefore, if $C$ has exactly four components, then all of them are rational and $\chi_{\textnormal{top}}(C^\nu) = 8$.

To complete the proof, we have to show that $[3;3]$- and quadruple-points $p\in C$ indeed satisfy $\mu_{p}(C)+r_{p}(C)-1 \geq 12$. 
For the half-log-canonical singularities this follows from the classification (Proposition~\ref{prop:possible-curve-singularities}) since $\mu(X_{p}) = p$ as $p\geq 9$, $\mu(J_{10}) = 10$, $\mu(Y_{r,s}) = 9+r+s$ as $r,s\geq 1$ and $\mu(J_{2,p}) = 10+p$ for $p\geq 1$; cf.\ Arnold, Gusein-Zade, Varchenko \cite[Chapter 15, p.~246 ff]{AGV:2012}.
The more general case follows from Wall \cite[Theorem~6.5.9]{Wall:2004} using the  concept of infinitely near points; we omit the details.
\end{proof}

Using this, we can prove that an octic has at most three $[3;3]$-points:

\begin{proposition}\label{prop:max-no-of-33-pts}~
\begin{enumerate}[label=\alph*)]
\item\label{prop:max-no-of-33-pts:item:one-33-point} If a plane curve has a $[3;3]$-point, then its degree is at least $5$.
\item\label{prop:max-no-of-33-pts:item:three-33-points} If a plane curve has three $[3;3]$-points, then its degree is at least $8$.
\item\label{prop:max-no-of-33-pts:item:three-33-points-along-conic} If a plane octic curve has three $[3;3]$-points with distinguished tangents along a conic, then the octic contains the conic.
\item\label{prop:max-no-of-33-pts:item:octic} There does not exist a plane octic curve with four $[3;3]$-points.
\end{enumerate}
\end{proposition}
\begin{proof}
Statement~\ref{prop:max-no-of-33-pts:item:one-33-point} follows immediately from Lemma \ref{lemma:degree-bounds}~\ref{lemma:degree-bounds:item:one-nn-point}.
To prove \ref{prop:max-no-of-33-pts:item:three-33-points}, let $C$ be a plane curve of degree $d$ with three $[3;3]$-points $p_{i}$, $i=1,2,3$.
By Corollary~\ref{cor:constraints-33-pts}, they are in general position insofar as that there exists a smooth conic $D\subset\PP^2$ through $p_{1}$, $p_{2}$ and $p_{3}$ which passes through $p_{1}$ and $p_{2}$ in distinguished tangent direction.
If $D$ is not contained in $C$, then $2d = DC\geq 2\cdot 6+3$, hence, $d\geq 8$.
If $D$ is contained in $C$, then the residual curve $C' = C-D$ of degree $d-2$ has three $[2;2]$-points along $D$; hence, $2(d-2) = DC' \geq 12$, which yields $d\geq 8$.
We analogously conclude part~\ref{prop:max-no-of-33-pts:item:three-33-points-along-conic} since if $D$ also passes through $p_{3}$ in distinguished tangent direction but is not contained in $C$, then $2d\geq 3\cdot 6 = 18$, hence $d\geq 9$.

To prove~\ref{prop:max-no-of-33-pts:item:octic}, we will derive a contradiction from the assumption that there exists such a curve $C$.
It follows from part~\ref{prop:max-no-of-33-pts:item:three-33-points} above that such a  $C$ has to be reduced: since all four $[3;3]$-points are, by assumption, isolated singularities, they have to lie on the reduced part, which has to have degree at least $7$ then, which is impossible unless $C$ is reduced.

Lemma~\ref{lemma:milnor-number-stuff} implies that $C$ has at least four rational components and that if $C$ has only $4$ components, it has no more singularities than the four $[3;3]$-points.
To prove the claim, we have to rule out all possible cases.
We distinguish the cases according to the number of lines in $C$.

If $C$ would not contain any line, it had to be a union of four smooth conics.
Since a sextic has at most two $[3;3]$-points by part~\ref{prop:max-no-of-33-pts:item:three-33-points}, we had to have three of the four $[3;3]$-points on each of the four conics.
Thus, if there were such a conic $C$, it had to be given as follows:
Suppose $p_{i}\in C$, $i=1,\dots,4$, are pairwise distinct $[3;3]$-points.
By Corollary~\ref{cor:constraints-33-pts}, those points and their distinguished tangents are sufficiently general such that there are four conics $C_{i}$, $i=0,\dots,3$, uniquely determined by the following properties:
\begin{itemize}[noitemsep,label=--]
\item $C_{0}$ contains $p_{1}, p_{2}$ and $p_{3}$, passing through $p_{1}$ and $p_{2}$ in distinguished tangent direction.
\item $C_{1}$ contains $p_{2}, p_{3}$ and $p_{4}$, passing through $p_{2}$ and $p_{3}$ in distinguished tangent direction.
\item $C_{2}$ contains $p_{1}, p_{3}$ and $p_{4}$, passing through $p_{1}$ and $p_{3}$ in distinguished tangent direction.
\item $C_{3}$ contains $p_{1}, p_{2}$ and $p_{4}$, passing through $p_{1}$ and $p_{2}$ in distinguished tangent direction.
\end{itemize}
Then $C$ is the union of those four.
In particular, $C$ is uniquely determined by the four points and two of the distinguished tangents.
This is the key to the proof that no such $C$ can exist.
With the help of projective automorphism, we can fix three points $p_{i}$, $i=1,2,3$, and two tangent directions in $p_{1}$ and $p_{2}$, as long as they do not point towards any of the remaining two points.
From this, we can compute $C_{0}$ and derive the distinguished tangent at $p_{3}$.
Then we compute $C_{1}$, $C_{2}$ and $C_{3}$ in dependence of a variable fourth point $p_{4}\in\PP^2$ and consider their tangent lines at $p_{4}$.
If an octic $C$ as desired existed, then for at least one point $p_{4}$, all three tangent lines would agree.
A computation shows that this happens if and only if $p_{4}$ lies on $C_{0}$, but then $C_{0} = C_{1} = C_{2} = C_{3}$, which is an irrelevant degenerate case.
An explicit calculation in Macaulay2 can be found in \cite[1111.m2]{Anthes:2018}.

Now suppose that $C$ contains exactly one line $L\subset C$.
Then the only possibility is that $C = E+D_{1}+D_{2}+L$, where $E$ is an irreducible cubic and $D_{1}$, $D_{2}$ are irreducible conics.
Since $C$ has only four components, all four of them are rational and $C$ has no extra singularities.
In particular, $E$ must be rational, hence either nodal or with an ordinary cusp.
But both are impossible since neither nodes nor ordinary cusps can contribute to $[3;3]$-points.
Thus, $C$ cannot contain just one line.

It remains to consider the possibility that $C$ decomposes into the union of two lines and a possibly reducible sextic $D\subset C$.
By Lemma~\ref{lemma:milnor-number-stuff} \ref{lemma:milnor-number-stuff:item1}, $\mu(D)\leq 25$.
From this or Proposition~\ref{prop:max-no-of-33-pts} \ref{prop:max-no-of-33-pts:item:three-33-points} we conclude that $D$ can have at most two $[3;3]$-points.
On the other hand, by Corollary~\ref{cor:constraints-33-pts}~\ref{cor:constraints-33-pts:item:distinct-tangents}, two lines on an octic can give rise to branches of at most two $[3;3]$-points of $C$, so that the residual sextic $D$ had to have at least two $[3;3]$-points.
Thus, $D$ had to have exactly two such and along the lines it had to have two $[2;2]$-points, which have Milnor number at least $3$, so that $D$ had to have total Milnor number $\mu(D)\geq 26 > 25$, yet again a contradiction.

This completes the proof.
\end{proof}

\begin{proposition}\label{prop:constraints-quad-pts}
If $C\subset \PP^2$ is a plane octic, then the following holds:
\begin{enumerate}[label=\alph*)]
\item\label{prop:constraints-quad-pts:item:no-3-collinear} No three quadruple-points on $C$ are collinear.
\item\label{prop:constraints-quad-pts:item:four} If $C$ has four quadruple-points, then $C$ is a union of four (possibly reducible) conics, meeting precisely in the (necessarily ordinary) quadruple-points.
\end{enumerate}
\end{proposition}
\begin{proof}
Statement~\ref{prop:constraints-quad-pts:item:no-3-collinear} is a special case of Lemma~\ref{lemma:degree-bounds}~\ref{lemma:degree-bounds:item:collinear}.

To prove~\ref{prop:constraints-quad-pts:item:four}, let $C$ be a plane octic with four quadruple-points.
By part~\ref{prop:constraints-quad-pts:item:no-3-collinear}, the quadruple-points are in general position so that there is a pencil of conics through those four points which spans the tangent spaces.
First, observe that if a conic $D$ of this pencil is tangent to a local analytic branch of one of the quadruple-points, then $D$ has to be contained in $C$, for otherwise we had to have $16 = DC \geq 3\cdot 4 + 5 = 17$.
Thus, to conclude that $C$ is a union of four members of the pencil, it suffices to show that at least one of the quadruple-points is non-degenerate.
In fact, they all are: If one of them were degenerate, then there would exist a conic $D$ of the pencil passing through it in the corresponding tangent direction, so that $D\subset C$ as above and $D(C-D) \geq 3\cdot 3 + 4 = 13 > 12$, unless $2D\subset C$, which is excluded since the quadruple-points are assumed to be isolated singularities.
\end{proof}

\begin{lemma}\label{lemma:quadruple-on-dist-tangent}
Let $C$ be a plane octic curve.
Suppose that $C$ has a $[3;3]$-point and a quadruple-point along the distinguished tangent line $L$ of the $[3;3]$-point.
Then $C$ contains $L$ and the residual septic $C-L$ meets $L$ only in those two points.
In particular, there is no second quadruple-point along $L$.
\end{lemma}
\begin{proof}
Let $p_{1}$ and $p_{2}$ be the $[3;3]$- and quadruple-point in question.
Then $C$ contains $L$ since
$CL = 8 \leq 6+4 \leq \mathrm{I}_{p_{1}}(C,L)+\mathrm{I}_{p_{2}}(C,L)$, where $\mathop{I}_{p}(-,-)$ denotes the intersection multiplicity at $p$.
Since $C$ is reduced at $p_{1}$, $L$ is not contained in the residual septic $D = C-L$ and so
$7 = DL \geq \mathrm{I}_{p_{1}}(D,L)+\mathrm{I}_{p_{2}}(D,L)\geq 4+3 = 7;$
thus, $D$ and $L$ meet only in $p_{0}$ and $p_{1}$ and $\mathrm{I}_{p_{1}}(D,L) = 4$, $\mathrm{I}_{p_{2}}(D,L) = 3$.
\end{proof}

\begin{remark}
The proof of Lemma~\ref{lemma:quadruple-on-dist-tangent} furthermore shows that if an octic has a quadruple-point on the distinguished tangent line of a $[3;3]$-point, then it is not a special tangent line of the quadruple-point as well.
Nonetheless, the quadruple-point could be degenerate.
\end{remark}

\begin{proposition}\label{prop:N1122-and-N1222}
There exists no plane octic curve admitting two $[3;3]$-points and two quadruple-points, or with one $[3;3]$-point and three quadruple-points.
\end{proposition}
\begin{proof}
By Proposition~\ref{cor:constraints-33-pts}~\ref{cor:constraints-33-pts:item:distinct-tangents}, \ref{cor:constraints-33-pts:item:no-3-collinear} and Lemma~\ref{lemma:quadruple-on-dist-tangent}, in both cases, the four points in question are general enough so that there exists a pencil of conics through all four points in question.
In particular, there exists a conic through all four points, passing through one of the $[3;3]$-points in distinguished tangent direction.
In both cases, this easily gives a contradiction comparing intersection numbers and local intersection multiplicities as before.
We omit the details.
%
\end{proof}

\begin{proposition}\label{prop:N1112}
Let $C\subset \PP^2$ be a plane octic with three $[3;3]$-points and a quadruple-point.
Then $C$ decomposes into four rational components; more precisely:

Let $p_{1},p_{2},p_{3}\in C$ be the $[3;3]$-points with distinguished tangent lines $L_{1},L_{2},L_{3}$, respectively.
Then only the following two configurations are possible.
\begin{enumerate}[label=\roman*)]
\item The three distinguished tangent lines $L_{1},L_{2},L_{3}$ meet in a point $p_{4}\in \bigcap_{i=1}^3L_{i}$ and $C = D_{5}+L_{1}+L_{2}+L_{3}$, where $D_{5}$ is a rational quintic which passes through $p_{4}$ and has three $A_{3}$-singularities (non-degenerate tac-nodes) at $p_{1},p_{2},p_{3}$ with distinguished tangent lines $L_{1},L_{2},L_{3}$, respectively.
\item There exists a conic $D_{2}$ passing through the three $[3;3]$-points in distinguished tangent direction and $C = D_{4}+D_{2}+L_{1}+L_{2}$, where $D_{4}$ is a rational quartic with an ordinary double-point in the intersection $p_{4}\in L_{1}\cap L_{2}$. Furthermore, $D_{4}$ is tangent to $D_{2}$ in $p_{1}$ and $p_{2}$ and has an $A_{3}$-singularity at $p_{3}$ with distinguished tangent line $L_{3}$.
\end{enumerate}
In either case, the $[3;3]$-points at $p_{1},p_{2},p_{3}$ and the quadruple-point at $p_{4}$ are non-degenerate and $C$ has no further singularities.
\end{proposition}
\begin{proof}
Let $C$ be a plane octic with three $[3;3]$-points and one quadruple-point.
Then, as in the proof of Proposition~\ref{prop:max-no-of-33-pts} above, we conclude that $C$ is reduced.
By Lemma~\ref{lemma:milnor-number-stuff} \ref{lemma:milnor-number-stuff:item2}, $C$ has at least four rational components and if $C$ has exactly four, then there are no extra singularities.

Let $C = C_{0}+C_{1}+C_{2}+C_{3}+C_{4}$ with $d_{i}:=\deg(C_{i})$,  $d_{1}\geq d_{2}\geq d_{3}\geq d_{4}$, and $C_{1},\dots,C_{4}$ rational.
(Note that $C_{0}$ could be empty or reducible).

Clearly, $d_{4}\leq 2$ and if $d_{4} = 2$, then $d_{0} = 0$ and $d_{1} = d_{2} = d_{3} = 2$ as well.
That is, $C$ would be a union of four smooth conics.
Going through the list of possible intersections of pairs of conics shows that there is no way for their union to have a quadruple- and three $[3;3]$-points.
Hence, we conclude $d_{4} = 1$, i.e., $C$ contains at least one line.

Before we continue, note that if $C$ has only four components and contains a cubic, then the cubic is rational and so has a node or ordinary cusp.
Either singularity does not contribute to a $[3;3]$-point and so has to be part of the quadruple-point, since $C$ has no extra singularities.

Now assume that $C$ has no more than the four singularities and contains a line.
Then the line must meet the residual septic in the singular points of $C$.
For the intersection multiplicities to add up to $7$, this has to be the quadruple- and one of the $[3;3]$-points.
This shows that $d_{\bullet} = (0,3,3,1,1)$ is impossible since the quadruple-point of $C$ had to be the union of the double-points of the two rational cubics, not allowing either line to pass through the quadruple-point as well.
Similarly, this excludes $d_{\bullet} = (0,3,2,2,1)$, for, the cubic $C_{1}$ had to have a double-point at, say, $p\in C_{1}$ and the line $C_{4}$ would have to pass through it.
But in order not to introduce an extra singularity of $C$, they must not meet anywhere else, so that the intersection at $p$ had to be with multiplicity $3$, which implies $\mu_{p}(C)+r_{p}(C)-1\geq 13$, a contradiction to (the proof of) Lemma~\ref{lemma:milnor-number-stuff}.
The only possibilities with $d_{0} = 0$ which remain are those corresponding to the claim and since every line has to pass through one of the $[3;3]$-points and the quadruple-point, the configurations have to be as claimed.

It remains to show that $C$ cannot have five or more components.
Note that in this case, $C$ had to contain at least two lines.

We first exclude the case that $C$ contains at least four lines, i.e., $d_{\bullet} = (4,1,1,1,1)$ or $(1,4,1,1,1)$.
Suppose $C$ were the union of four lines and a possibly reducible quartic.
Then the quartic had to have at least three non-collinear $[2;2]$-points, which is impossible, e.g., by Hui's classification \cite{Hui:1979}.

If $C$ had at least five components but at most three lines, then $C$ had to contain at least two lines, $L_{1}$, $L_{2}$, and a smooth conic $D$.
The residual quartic then had to have exactly two components, of which at least one had to be rational, hence, either $C = L_{1}+L_{2}+L_{3}+D+D'$ for a third line $L_{3}$ and an irreducible cubic $D'$, or $C = L_{1}+L_{2}+D+D'+D''$ for two more irreducible conics $D'$, $D''$.
Thus, we are left with these two cases.

If we had $C = L_{1}+L_{2}+L_{3}+D+D'$ with $D'$ an irreducible cubic, then $\chi_{\textnormal{top}}(C^\nu) = 8$ or $10$, so that $C$ could have at most one additional singularity, necessarily of type $A_{1}$ or $A_{2}$.
Therefore, the lines had to be concurrent, for otherwise $C$ had to have at least three singularities with local branches having different tangent directions.
In particular, the conic $D$ had to meet at least one of the lines transversely.
But in that case, one of the intersection points had to be the quadruple-point of $C$; hence, $D$ would pass through the common intersection point of the three lines.
Therefore, $D$ could be tangent to at most one of them, which would result in too many singularities for $C$, a contradiction.

Finally, if $C = L_{1}+L_{2}+D+D'+D''$ for two more irreducible conics $D'$, $D''$, then $\chi_{\textnormal{top}}(C^{\nu}) = 10$, so as before, $C$ could have at most one additional singularity, of type $A_{1}$ (since lines and conics have no cusps).
Therefore, the intersection point of $L_{1}$ and $L_{2}$ had to be an extra double-point or the quadruple-point of $C$.
If it were an extra double-point, then the residual sextic $D+D'+D''$ had to have two $[2;2]$-points, a $[3;3]$-point and a quadruple-point, resulting in a total Milnor number of at least $25$.
But the only sextic with this (maximal) Milnor number is a union of six concurrent lines by Lemma~\ref{lemma:milnor-number-stuff} \ref{lemma:milnor-number-stuff:item1}, so this is impossible.
Hence, the two lines had to meet in the quadruple-point of $C$, so that two of the conics would necessarily pass through this point as well, say $D$ and $D'$.
But then they could not be tangent to the lines in the $[3;3]$-points, as would be necessary, a contradiction.

This rules out all cases as claimed and completes the proof.
\end{proof}

We also have to study the possible configurations of non-simple singularities a reduced sextic can have.
Note that the upper bounds we give below are most likely far from optimal.
\begin{proposition}\label{prop:sextic-singularities}
Let $C$ be a reduced half-log-canonical plane sextic curve (cf.\ \ref{table:singularity-classification}).
Then the only possible non-simple singularities $C$ can have are:
\begin{enumerate}[label=\roman*)]
\item One $X_{p}$, $9\leq p\leq 24$.
\item One $Y_{r,s}$, $r,s\geq 1$, $r+s\leq 15$.
\item One $J_{10}$.
\item One $J_{2,p}$, $1\leq p\leq 14$.
\item Two $J_{10}$.
\end{enumerate}
In the last case, $C$ decomposes as a union of three conics, at most one of which degenerates into a union of two lines.
\end{proposition}
\begin{proof}
The maximal total Milnor number $\mu(C)$ of a reduced sextic $C\subset\PP^2$ without a $6$-fold-point is $24$.
Since all reduced non-simple plane curve singularities with log-canonical threshold at least $\tfrac{1}{2}$ have Milnor number greater or equal than $9$, we conclude that $C$ can have at most two such.
Moreover, this explains the given upper bounds since the Milnor number of a singularity of type $X_{p}$, $Y_{r,s}$ and $J_{2,p}$ is $p$, $9+r+s$ and $10+p$, respectively.

If $C$ has a quadruple-point, then $C$ cannot have a second non-simple singularity:
By Lemma~\ref{lemma:degree-bounds}~\ref{lemma:degree-bounds:item:n-and-m-fold}, it must contain the line $L$ joining them.
But then the residual quintic $C' = C-L$ has to have either two triple-points, or a triple-point and a $[2;2]$-point with $L$ as distinguished; both options are impossible by Lemma~\ref{lemma:degree-bounds}~\ref{lemma:degree-bounds:item:n-and-m-fold} and~\ref{lemma:degree-bounds:item:nn-and-m}, respectively.

Thus, it remains to show that if $C$ has two $[3;3]$-points, then they are both non-degenerate and that $C$ is a union of three conics then, meeting tangentially in both points.
In fact, as in the proof of Proposition~\ref{prop:constraints-quad-pts}, we can conclude that if $C$ has two $[3;3]$-points, then it has at least three rational components and that if it has exactly three, then there are no further singular points.
Then either $C$ is the union of three rational conics, or $C$ contains a line.
The only possibility for $C$ to contain exactly one line is that $C = C_{1}+C_{2}+C_{3}$, with $C_{1}$ a line, $C_{2}$ a rational conic and $C_{3}$ a rational cubic.
But a rational cubic has an $A_{1}$ or $A_{2}$-singularity, which does not contribute to a $[3;3]$-point, so that $C$ had to have an extra singular point, a contradiction.
Thus, $C$ contains a conic $D$ (which may be the union of two lines).
Since the residual quartic $C' = C-D$ cannot have a $[3;3]$-point, the conic $D$ must pass through both $[3;3]$-points in distinguished tangent direction and $C'$ has to have two $[2;2]$-points where $C$ has its $[3;3]$-points.
But then $C'$ decomposes as a union of two conics meeting tangentially in those two points.
Since the intersection number of any pair of these three conics has to be $4 = 2+2$, the $[3;3]$-points have to be non-degenerate.
\end{proof}

The quartics are easy to deal with: There is only one quartic with a non-simple singularity, namely, the union of four concurrent lines, giving rise to a single $X_{9}$-singularity, see, e.g., Hui \cite{Hui:1979}.
Alternatively, this also follows from Lemma~\ref{lemma:milnor-number-stuff} since the Milnor number of a non-simple singularity is at least $9 = (4-1)^2$.

Since this is used in the Macaulay2-code \cite[sextics.m2~II.2]{Anthes:2018}, note that from Hui's classification we also know that there are quartics with globally two different kinds of $A_{5}$-singularities, namely, some where the higher order directions are non-trivial, and some where they are not.

\section{The Macaulay2-code}
\label{section:M2}

We conclude this article with a quick tour through the arguments used in the Macaulay2-code which computes the dimension of the various strata and shows that the components are disjoint.
All scripts can be obtained from a GitLab repository \cite{Anthes:2018}.
They are inspired by a similar script by S\"onke Rollenske.

If the singularities we want a plane curve to have are controlled by a configuration which can be fixed by a suitable automorphism, then the dimension of this component is easy to compute using Macaulay2.
One puts all constraints in an ideal and asks the system for a minimal generating set of the module of octics satisfying these equations.
An example illustrating this is given in \ref{listing:one-quadruple-point}.

\begin{lstlisting}[language=macaulay2,caption={Example without parameters},label={listing:one-quadruple-point}]
S = QQ[x,y,z]; -- Homog. coordinate ring of PP^2
Point = ideal(x,y); -- Homog. ideal of (0;0;1) in PP^2
QuadruplePoint = Point^4;
m = super basis(8,QuadruplePoint) -- outputs:
--  | x8 x7y x7z x6y2 x6yz x6z2 x5y3 x5y2z x5yz2 x5z3 x4y4 x4y3z
--   ----------------------------------------------------------------
--   x4y2z2 x4yz3 x4z4 x3y5 x3y4z x3y3z2 x3y2z3 x3yz4 x2y6 x2y5z
--   ----------------------------------------------------------------
--   x2y4z2 x2y3z3 x2y2z4 xy7 xy6z xy5z2 xy4z3 xy3z4 y8 y7z y6z2 y5z3
--   ----------------------------------------------------------------
--   y4z4 |
assert(numgens source m == 35);
\end{lstlisting}

It shows that the sub-space of the vector space of octic forms in $x,y,z$ whose associated plane curve has multiplicity at least four in $(0;0;1)\in\PP^{2}$ is of dimension $35$.
Thus, their linear system is of dimension $34$ and there is an open sub-space $V$ of octics where the quadruple-point is non-degenerate and which has no further non-simple singularities.
Every plane octic curve with a quadruple-point is projectively equivalent to one of those with a quadruple-point in $(0;0;1)$.
Therefore, the space of octic curves with exactly one quadruple-point and no other non-simple singularities is the quotient $V/G$, where $G\subset\PGL(3,\CC) = \mathrm{Aut}(\PP^{2})$ is the stabiliser of the point $(0;0;1)$.
Since the dimension of $G$ is $6$, we conclude $\dim(\Nfrak_{2}) = 34-6 = 28$.

When the configuration cannot be fixed by an automorphism, then we have to consider parameters.
As an example \ref{listing:N12}, we consider the case of a $[3;3]$-point and a quadruple-point.
Up to projective automorphism, there are two distinct configurations; one where the distinguished tangent of the $[3;3]$-point points towards the quadruple-point and one where it does not.

\begin{lstlisting}[language=macaulay2,caption={Example with a parameter},label={listing:N12}]
A = QQ[t]; -- Affine coordinate ring of parameter space
S = A[x,y,z]; -- Homog. coordinate ring of trivial PP^2-family
P = ideal(x,y); -- Homog. ideal of (0;0;1) in PP^2
PwT = ideal(x^2,y-t*x); -- (0;0;1) with tangent direction y-tx
P33 = PwT^3; -- Corresponding [3;3]-point constraints
Q = ideal(y,z); -- Homog ideal of (1;0;0)
I0 = intersect(Q^4,sub(P33,{t=>0}));
I1 = intersect(Q^4,sub(P33,{t=>1}));
m0 = super basis(8,I0);
m1 = super basis(8,I1);
assert(numgens source m0 > numgens source m1);
-- Thus, something special is going on if t = 0. In fact, that is where y-t*x lies in Q. Another component?
-- We consider the universal octic with a [3;3]-point as prescribed:
m = super basis(8,P33);
n = numgens source m;
RA = A[a_0..a_(n-1)];
params = gens RA;
RS = RA[gens S];
inc = map(RS,S);
f = sum for i from 0 to (n-1) list a_i*inc(m_(0,i));
-- The conditions that it has a quadruple-point at Q:
toBeZero = f%inc(Q^4);
toBeZeroCoefficients =
for term in terms toBeZero list leadCoefficient(term);
M = matrix for eq in toBeZeroCoefficients list
for g in gens RA list sub(leadCoefficient(eq//g),A);
-- For every t = t_0, the kernel of sub(M,{t=>t_0}) corresponds to the space of octic forms with a quadruple-point in Q and a [3;3]-point in P with distinguished tangent direction y-t_0*x. Thus, the rank of M drops where something interesting is happening:
droppingRankConditions = minors(numgens target M, mingens image M);
assert(droppingRankConditions == ideal(t));
-- Thus, generically, the rank of M is maximal (10, in fact) and it drops if and only if t = 0. Furthermore, the difference between the octic forms obtained for t = 0 and those arising as limits t --> 0, t != 0, corresponds to the difference between the kernels.
Kspecial = mingens kernel sub(M,{t=>0});
Kgeneral = mingens sub(kernel M, {t=>0});
assert isSubset(image Kgeneral, image Kspecial);
assert(image Kspecial != image Kgeneral);
-- They give rise to the following octics:
special = sub(matrix{ for j from 0 to numgens source Kspecial-1 list
sub(f,for i from 0 to numgens target Kspecial-1 list
params_i=>Kspecial_(i,j))},{t=>0});
special = sub(special, S);
general = sub(matrix{ for j from 0 to numgens source Kgeneral-1 list
sub(f,for i from 0 to numgens target Kgeneral-1 list
params_i=>Kgeneral_(i,j))},{t=>0});
general = sub(general, S);
-- The line y is contained once in every member of the special locus, but it is contained twice in every member of the general locus:
use S;
assert( special%y == 0 and not special%y^2 == 0 );
assert( general%y^2 == 0 );
-- Since non-reduced octics are not allowed in this stratum, the components are disjoint.
\end{lstlisting}

Since jobs like creating the ideals containing the constraints or building a universal family etc.\ have to be done multiple times, they are provided as functions in the file \cite[octicsFunctions.m2]{Anthes:2018}.
Explanations how they work can be found in the comments there.
For example, there is also a function checking whether a quadruple-point is ordinary.
(It blows up once and checks that the discriminant is non-trivial.
Therefore, if applied over a coefficient ring which is not a field, it only means that it is generically non-degenerate.)
Similarly, there is a function checking if a $[3;3]$-point is non-degenerate.
\bibliographystyle{plain}
\bibliography{references}{}

\begin{thebibliography}{10}

\bibitem{Alexeev:1994}
Valery Alexeev.
\newblock Boundedness and {$K^2$} for log surfaces.
\newblock {\em Internat. J. Math.}, 5(6):779--810, 1994.

\bibitem{AM:2004}
Valery Alexeev and Shigefumi Mori.
\newblock Bounding singular surfaces of general type.
\newblock In {\em Algebra, arithmetic and geometry with applications ({W}est
  {L}afayette, {IN}, 2000)}, pages 143--174. Springer, Berlin, 2004.

\bibitem{AP:2012}
Valery Alexeev and Rita Pardini.
\newblock Non-normal abelian covers.
\newblock {\em Compos. Math.}, 148(4):1051--1084, 2012.

\bibitem{Anthes:thesis}
{B}en {A}nthes.
\newblock {\em {G}orenstein stable surfaces satisfying {$K_X^2=2$} and
  {$\chi(\mathcal{O}_{X})=4$}}.
\newblock PhD thesis, Philipps-Universit{\"a}t Marburg, 2018.

\bibitem{Anthes:2018}
{B}en {A}nthes.
\newblock {O}ctics{W}ith{N}on{S}imple{S}ingularities
  \url{https://gitlab.com/anthes/octics}.
\newblock {G}it{L}ab repository, 2018.
\newblock commit 4431508588856583eaa5b3ef814ab3b15ed857ef.

\bibitem{AGV:2012}
V.~I. Arnold, S.~M. Gusein-Zade, and A.~N. Varchenko.
\newblock {\em Singularities of differentiable maps. {V}olume 1}.
\newblock Modern Birkh\"{a}user Classics. Birkh\"{a}user/Springer, New York,
  2012.
\newblock Classification of critical points, caustics and wave fronts,
  Translated from the Russian by Ian Porteous based on a previous translation
  by Mark Reynolds, Reprint of the 1985 edition.

\bibitem{BHPV}
W.~P. Barth, K.~Hulek, C.~Peters, and A.~Van~de Ven.
\newblock {\em Compact complex surfaces}, volume~4 of {\em Ergebnisse der
  Mathematik und ihrer Grenzgebiete. 3. Folge. A Series of Modern Surveys in
  Mathematics [Results in Mathematics and Related Areas. 3rd Series. A Series
  of Modern Surveys in Mathematics]}.
\newblock Springer-Verlag, Berlin, second edition, 2004.

\bibitem{BGvBK:2009}
Christian B\"ohning, Hans-Christian Graf~von Bothmer, and Jakob Kr\"oker.
\newblock Rationality of moduli spaces of plane curves of small degree.
\newblock {\em Experiment. Math.}, 18(4):499--508, 2009.

\bibitem{Brieskorn:1979}
E.~Brieskorn.
\newblock Die {H}ierarchie der {$1$}-modularen {S}ingularit\"{a}ten.
\newblock {\em Manuscripta Math.}, 27(2):183--219, 1979.

\bibitem{CFHR:1999}
Fabrizio Catanese, Marco Franciosi, Klaus Hulek, and Miles Reid.
\newblock Embeddings of curves and surfaces.
\newblock {\em Nagoya Math. J.}, 154:185--220, 1999.

\bibitem{Degtyarev:1989}
A.~I. Degtyar\"{e}v.
\newblock Isotopic classification of complex plane projective curves of degree
  {$5$}.
\newblock {\em Algebra i Analiz}, 1(4):78--101, 1989.

\bibitem{Degtyarev:2012}
Alex Degtyarev.
\newblock {\em Topology of algebraic curves}, volume~44 of {\em De Gruyter
  Studies in Mathematics}.
\newblock Walter de Gruyter \& Co., Berlin, 2012.
\newblock An approach via dessins d'enfants.

\bibitem{Doherty:2008}
Davis~C. Doherty.
\newblock Singularities of generic projection hypersurfaces.
\newblock {\em Proc. Amer. Math. Soc.}, 136(7):2407--2415, 2008.

\bibitem{Dolgachev:1980}
Igor Dolgachev.
\newblock Cohomologically insignificant degenerations of algebraic varieties.
\newblock {\em Compositio Math.}, 42(3):279--313, 1980/81.

\bibitem{Durfee:1983}
Alan~H. Durfee.
\newblock A naive guide to mixed {H}odge theory.
\newblock In {\em Singularities, {P}art 1 ({A}rcata, {C}alif., 1981)},
  volume~40 of {\em Proc. Sympos. Pure Math.}, pages 313--320. Amer. Math.
  Soc., Providence, RI, 1983.

\bibitem{FPR:201X}
Marco Franciosi, Rita Pardini, and S\"onke Rollenske.
\newblock {G}orenstein stable surfaces with {$K_X^2=2$} and $\p_{g}>0$.
\newblock In preparation.

\bibitem{FPR:2015b}
Marco Franciosi, Rita Pardini, and S\"onke Rollenske.
\newblock Computing invariants of semi-log-canonical surfaces.
\newblock {\em Math. Z.}, 280(3-4):1107--1123, 2015.

\bibitem{FPR:2015}
Marco Franciosi, Rita Pardini, and S\"onke Rollenske.
\newblock Log-canonical pairs and {G}orenstein stable surfaces with
  {$K_X^2=1$}.
\newblock {\em Compos. Math.}, 151(8):1529--1542, 2015.

\bibitem{FPR:2017}
Marco Franciosi, Rita Pardini, and S\"onke Rollenske.
\newblock Gorenstein stable surfaces with {$K^2_X=1$} and {$p_g>0$}.
\newblock {\em Math. Nachr.}, 290(5-6):794--814, 2017.

\bibitem{M2}
Daniel~R. Grayson and Michael~E. Stillman.
\newblock Macaulay2, a software system for research in algebraic geometry.
\newblock Available at \url{http://www.math.uiuc.edu/Macaulay2/}.

\bibitem{GGLR:201X}
Mark Green, Phillip Griffiths, Radu Laza, and Colleen Robles.
\newblock {H}odge theory and moduli of {H}-surfaces.
\newblock In preparation.

\bibitem{GGR:2014}
Mark {Green}, Phillip {Griffiths}, and Colleen {Robles}.
\newblock {Extremal degenerations of polarized Hodge structures}.
\newblock {\em ArXiv e-prints}, March 2014.

\bibitem{Hacking:2004}
Paul Hacking.
\newblock Compact moduli of plane curves.
\newblock {\em Duke Math. J.}, 124(2):213--257, 2004.

\bibitem{Hacking:2016}
Paul Hacking.
\newblock Compact moduli spaces of surfaces and exceptional vector bundles.
\newblock In {\em Compactifying moduli spaces}, Adv. Courses Math. CRM
  Barcelona, pages 41--67. Birkh\"auser/Springer, Basel, 2016.

\bibitem{Hartshorne:1994}
R.~Hartshorne.
\newblock Generalized divisors on {G}orenstein schemes.
\newblock In {\em Proceedings of {C}onference on {A}lgebraic {G}eometry and
  {R}ing {T}heory in honor of {M}ichael {A}rtin, {P}art {III} ({A}ntwerp,
  1992)}, pages 287--339, 1994.

\bibitem{Hartshorne:2007}
Robin Hartshorne.
\newblock Generalized divisors and biliaision.
\newblock {\em Illinois Journal of Mathematics}, 51(1):83--98, 2007.

\bibitem{Hassett:1999}
Brendan Hassett.
\newblock Stable log surfaces and limits of quartic plane curves.
\newblock {\em Manuscripta Math.}, 100(4):469--487, 1999.

\bibitem{Horikawa:1976}
Eiji Horikawa.
\newblock Algebraic surfaces of general type with small {$C^{2}_{1}.$}\ {I}.
\newblock {\em Ann. of Math. (2)}, 104(2):357--387, 1976.

\bibitem{Hui:1979}
Chung-Man Hui.
\newblock {\em Plane quartic curves}.
\newblock PhD thesis, The university of Liverpool, 1979.

\bibitem{Katsylo:1992}
P.~I. Katsylo.
\newblock On the birational geometry of the space of ternary quartics.
\newblock In {\em Lie groups, their discrete subgroups, and invariant theory},
  volume~8 of {\em Adv. Soviet Math.}, pages 95--103. Amer. Math. Soc.,
  Providence, RI, 1992.

\bibitem{KR:201X}
Matt Kerr and Colleen Robles.
\newblock {P}artial orders and polarized relations on limit mixed hodge
  structures.
\newblock In preparation.

\bibitem{KleimanMatrins:2009}
Steven~Lawrence Kleiman and Renato~Vidal Martins.
\newblock The canonical model of a singular curve.
\newblock {\em Geom. Dedicata}, 139:139--166, 2009.

\bibitem{KSB:1988}
J.~Koll\'ar and N.~I. Shepherd-Barron.
\newblock Threefolds and deformations of surface singularities.
\newblock {\em Invent. Math.}, 91(2):299--338, 1988.

\bibitem{Kollar:1990}
J\'anos Koll\'ar.
\newblock Projectivity of complete moduli.
\newblock {\em J. Differential Geom.}, 32(1):235--268, 1990.

\bibitem{Kollar:2013}
J{{\'a}}nos Koll{{\'a}}r.
\newblock {\em Singularities of the minimal model program}, volume 200 of {\em
  Cambridge Tracts in Mathematics}.
\newblock Cambridge University Press, Cambridge, 2013.
\newblock With a collaboration of S{{\'a}}ndor Kov{{\'a}}cs.

\bibitem{MO262045}
S\'andor Kov\'acs.
\newblock semi-log canonial singularities is an open condition?
\newblock MathOverflow.
\newblock \url{https://mathoverflow.net/q/262045} (version: 2017-02-12).

\bibitem{KS:2016}
S\'andor Kov\'acs and Karl Schwede.
\newblock Inversion of adjunction for rational and {D}u {B}ois pairs.
\newblock {\em Algebra Number Theory}, 10(5):969--1000, 2016.

\bibitem{KSS:2010}
S\'{a}ndor~J. Kov\'{a}cs, Karl Schwede, and Karen~E. Smith.
\newblock The canonical sheaf of {D}u {B}ois singularities.
\newblock {\em Adv. Math.}, 224(4):1618--1640, 2010.

\bibitem{LR-hypersurfaces}
Wenfei Liu and S{{\"o}}nke Rollenske.
\newblock Two-dimensional semi-log-canonical hypersurfaces.
\newblock {\em Matematiche (Catania)}, 67(2):185--202, 2012.

\bibitem{LR-pluri}
Wenfei Liu and S{{\"o}}nke Rollenske.
\newblock Pluricanonical maps of stable log surfaces.
\newblock {\em Adv. Math.}, 258:69--126, 2014.

\bibitem{LR-geo}
Wenfei Liu and S\"onke Rollenske.
\newblock Geography of {G}orenstein stable log surfaces.
\newblock {\em Trans. Amer. Math. Soc.}, 368(4):2563--2588, 2016.

\bibitem{Milnor:1968}
John Milnor.
\newblock {\em Singular points of complex hypersurfaces}.
\newblock Annals of Mathematics Studies, No. 61. Princeton University Press,
  Princeton, N.J.; University of Tokyo Press, Tokyo, 1968.

\bibitem{GIT}
David Mumford.
\newblock {\em Geometric invariant theory}.
\newblock Ergebnisse der Mathematik und ihrer Grenzgebiete, Neue Folge, Band
  34. Springer-Verlag, Berlin-New York, 1965.

\bibitem{Pardini:1991}
Rita Pardini.
\newblock Abelian covers of algebraic varieties.
\newblock {\em J. Reine Angew. Math.}, 417:191--213, 1991.

\bibitem{PS:2008}
{C}hris A.~M. {P}eters and Joseph H.~M. Steenbrink.
\newblock {\em Mixed {H}odge structures}, volume~52 of {\em Ergebnisse der
  Mathematik und ihrer Grenzgebiete. 3. Folge. A Series of Modern Surveys in
  Mathematics [Results in Mathematics and Related Areas. 3rd Series. A Series
  of Modern Surveys in Mathematics]}.
\newblock Springer-Verlag, Berlin, 2008.

\bibitem{Robles:2016a}
Colleen Robles.
\newblock Classification of horizontal $\text{SL}(2)$ s.
\newblock {\em Compositio Mathematica}, 152(5):918--954, 2016.

\bibitem{Robles:2016}
Colleen {Robles}.
\newblock {Degenerations of Hodge structure}.
\newblock {\em ArXiv e-prints}, July 2016.

\bibitem{Rosenlicht:1952}
Maxwell Rosenlicht.
\newblock Equivalence relations on algebraic curves.
\newblock {\em Ann. of Math. (2)}, 56:169--191, 1952.

\bibitem{Steenbrink:1980}
Joseph H.~M. Steenbrink.
\newblock Cohomologically insignificant degenerations.
\newblock {\em Compositio Math.}, 42(3):315--320, 1980/81.

\bibitem{stacks-project}
{The~Stacks~Project~Authors}.
\newblock Stacks project.
\newblock \url{http://stacks.math.columbia.edu}, 2017.

\bibitem{Urabe:1989}
Tohsuke Urabe.
\newblock Dynkin graphs and combinations of singularities on plane sextic
  curves.
\newblock In {\em Singularities ({I}owa {C}ity, {IA}, 1986)}, volume~90 of {\em
  Contemp. Math.}, pages 295--316. Amer. Math. Soc., Providence, RI, 1989.

\bibitem{MO:singularities}
MathOverflow user JNS~(https://mathoverflow.net/users/127497/jns).
\newblock Classification of singularities of plane curves of fixed degree
  (reference request).
\newblock MathOverflow.
\newblock URL:https://mathoverflow.net/q/307835 (version: 2018-08-09).

\bibitem{Wall:1996}
C.~T.~C. Wall.
\newblock Highly singular quintic curves.
\newblock {\em Math. Proc. Cambridge Philos. Soc.}, 119(2):257--277, 1996.

\bibitem{Wall:2004}
C.~T.~C. Wall.
\newblock {\em Singular points of plane curves}, volume~63 of {\em London
  Mathematical Society Student Texts}.
\newblock Cambridge University Press, Cambridge, 2004.

\bibitem{Yang:1996}
Jin-Gen Yang.
\newblock Sextic curves with simple singularities.
\newblock {\em Tohoku Math. J. (2)}, 48(2):203--227, 1996.

\end{thebibliography}
\end{document}